\documentclass[11pt,a4paper,reqno]{amsart}
\usepackage{amsthm,amsmath,amsfonts,amssymb,amsxtra,appendix,bookmark,dsfont,latexsym,bm,hyperref,delarray,color,euscript,amsgen,amsbsy,amsopn,amscd,latexsym,mathrsfs}

\makeatletter

\usepackage{hyperref}

\makeatletter
\def\@tocline#1#2#3#4#5#6#7{\relax
  \ifnum #1>\c@tocdepth 
  \else
    \par \addpenalty\@secpenalty\addvspace{#2}%
    \begingroup \hyphenpenalty\@M
    \@ifempty{#4}{%
      \@tempdima\csname r@tocindent\number#1\endcsname\relax
    }{%
      \@tempdima#4\relax
    }%
    \parindent\z@ \leftskip#3\relax \advance\leftskip\@tempdima\relax
    \rightskip\@pnumwidth plus4em \parfillskip-\@pnumwidth
    #5\leavevmode\hskip-\@tempdima
      \ifcase #1
       \or\or \hskip 1em \or \hskip 2em \else \hskip 3em \fi%
      #6\nobreak\relax
      \dotfill
      \hbox to\@pnumwidth{\@tocpagenum{#7}}
    \par
    \nobreak
    \endgroup
  \fi}
\makeatother

\newtheorem{theorem}{Theorem}[section]
\newtheorem{lemma}[theorem]{Lemma}
\newtheorem{proposition}[theorem]{Proposition}

\theoremstyle{definition}
\newtheorem{definition}[theorem]{Definition}
\newtheorem{asumption}[theorem]{Asumption}

\theoremstyle{remark}
\newtheorem{remark}[theorem]{Remark}

\newcommand{\one}{{\ensuremath {\mathds 1} }}

\newcommand{\dps}{\displaystyle}
\newcommand{\ii}{\infty}
\newcommand\R{{\ensuremath {\mathbb R} }}
\newcommand\C{{\ensuremath {\mathbb C} }}
\newcommand\N{{\ensuremath {\mathbb N} }}
\newcommand\Z{{\ensuremath {\mathbb Z} }}

\newcommand\1{{\ensuremath {\mathds 1} }}
\newcommand\dGamma{{\rm d}\Gamma}
\newcommand{\<}{\langle}
\renewcommand{\>}{\rangle}
\newcommand\nn{\nonumber}

\renewcommand\phi{\varphi}
\newcommand{\bT}{\mathbb{T}}
\newcommand{\bH}{\mathbb{H}}
\newcommand{\bX}{\mathbb{X}}

\newcommand{\bA}{\mathbb{A}}
\newcommand{\bB}{\mathbb{B}}
\newcommand{\bW}{\mathbb{W}}
\newcommand{\bWren}{\mathbb{W}^{\rm ren}}
\newcommand{\gH}{\mathfrak{H}}
\newcommand{\gF}{\mathfrak{F}}
\newcommand{\gS}{\mathfrak{S}}

\newcommand{\wto}{\rightharpoonup}
\newcommand{\cS}{\mathcal{S}}
\newcommand{\cP}{\mathcal{P}}

\newcommand{\cM}{\mathcal{M}}

\newcommand{\cE}{\mathcal{E}}

\newcommand{\cF}{\mathcal{F}}
\newcommand{\cN}{\mathcal{N}}
\newcommand{\cD}{\mathcal{D}}
\newcommand{\cH}{\mathcal{H}}

\newcommand{\cZ}{\mathcal{Z}}
\newcommand{\cU}{\mathcal{U}}

\newcommand{\F}{\mathcal{F}}

\renewcommand{\epsilon}{\varepsilon}
\newcommand\pscal[1]{{\ensuremath{\left\langle #1 \right\rangle}}}
\newcommand{\norm}[1]{ \left| \! \left| #1 \right| \! \right| }
\DeclareMathOperator{\tr}{{\rm Tr}}
\DeclareMathOperator{\Tr}{{\rm Tr}}

\renewcommand{\ge}{\geqslant}
\renewcommand{\le}{\leqslant}
\renewcommand{\geq}{\geqslant}
\renewcommand{\leq}{\leqslant}
\renewcommand{\hat}{\widehat}
\renewcommand{\tilde}{\widetilde}
\newcommand{\cHcl}{\cH_{\rm cl}}

\newcommand{\dG}{\mathrm{d}\Gamma}

\newcommand{\rhoO}{\varrho_0}
\newcommand{\lam}{\lambda}
\newcommand{\eps}{\varepsilon}

\newcommand{\cZl}{\cZ(\lambda)}

\newcommand{\av}[1]{\left\langle#1\right\rangle_{\mu_0}}
\newcommand{\Gammat}{\widetilde{\Gamma}}

\newcommand{\FNL}{\mathcal{D}}

\numberwithin{equation}{section}

\begin{document}

\title{Classical field theory limit of many-body quantum Gibbs states in 2D and 3D}

\author[M. Lewin]{Mathieu Lewin}
\address{CNRS \& Universit\'e Paris-Dauphine, PSL University, CEREMADE, Place de Lattre de Tassigny, F-75016 PARIS, France} 
\email{mathieu.lewin@math.cnrs.fr}

\author[P.T. Nam]{Phan Th\`anh Nam}
\address{Department of Mathematics, LMU Munich, Theresienstrasse 39, 80333 Munich, and Munich Center for Quantum Science and Technology (MCQST), Schellingstr. 4, 80799 Munich, Germany} 
\email{nam@math.lmu.de}

\author[N. Rougerie]{Nicolas Rougerie}
\address{Universit\'e Grenoble-Alpes \& CNRS,  LPMMC (UMR 5493), B.P. 166, F-38042 Grenoble, France}
\email{nicolas.rougerie@lpmmc.cnrs.fr}

\date{October, 2020}

\begin{abstract} 
We provide a rigorous derivation of nonlinear Gibbs measures in two and three space dimensions, starting from many-body quantum systems in thermal equilibrium. More precisely, we prove that the grand-canonical Gibbs state of a large bosonic quantum system converges to the Gibbs measure of a nonlinear Schr\"odinger-type classical field theory, in terms of partition functions and reduced density matrices. The Gibbs measure thus describes the behavior of the infinite Bose gas at criticality, that is, close to the phase transition to a Bose-Einstein condensate. The Gibbs measure is concentrated on singular distributions and has to be appropriately renormalized, while the quantum system is well defined without any renormalization. By tuning a single real parameter (the chemical potential), we obtain a counter-term for the diverging repulsive interactions which provides the desired Wick renormalization of the limit classical theory. The proof relies on a new estimate on the entropy relative to quasi-free states and a novel method to control quantum variances.
\end{abstract}

\maketitle

\tableofcontents

\section{Introduction}

A \emph{nonlinear Gibbs measure} $\mu$ is a probability measure in infinite dimension of the form
\begin{equation}
{\rm d}\mu(u)=\frac{e^{-\cD[u]}}{z}{\rm d}\mu_0(u)
\label{eq:intro_form_mu}
\end{equation}
where $z$ is a normalization factor, $\mu_0$ is a Gaussian probability measure and $\cD$ is a non quadratic positive function. In this article we consider the case where $\mu_0$ has the covariance operator $(-\Delta+V_0)^{-1}$ over a (bounded or unbounded) open set $\Omega\subset\R^d$, for some function $V_0:\Omega\to\R$ (when $\Omega$ is unbounded we assume that $V_0\to+\ii$ at infinity to ensure that the spectrum of $-\Delta+V_0$ is discrete). One should therefore think that 
\begin{equation}\label{eq:intro form mu0}
{\rm d}\mu_0(u)= \;\text{``}\; (Z_0)^{-1}\exp\left(-\int_\Omega(|\nabla u|^2+V_0|u|^2)\right){\rm d}u\;\text{''} 
\end{equation}
so that 
$${\rm d}\mu(u)= \;\text{``}\; (Z_0z)^{-1}\exp\left(-\left(\int_\Omega(|\nabla u|^2+V_0|u|^2)+\cD[u]\right)\right){\rm d}u\;\text{''},$$
but this is of course purely formal.  
The nonlinear part $\cD$ is taken of the form
\begin{equation}
 \cD[u]={\rm Ren.}\left\{\frac12\iint_{\Omega\times\Omega}w(x-y)|u(x)|^2|u(y)|^2\,{\rm d}x\,{\rm d}y\right\}
 \label{eq:form_F_NL}
\end{equation}
for a sufficiently regular function $w$ of positive Fourier transform. Nonlinearities of this type are ubiquitous in the literature and have an important physical meaning, as we will recall.

In dimension $d=1$ and under appropriate assumptions on the function $V_0$, the Wiener-type Gaussian probability measure $\mu_0$ concentrates on continuous functions. The double integral appearing in the definition~\eqref{eq:form_F_NL} of $\cD$ therefore makes sense $\mu_0$-almost surely without any special care, even for $w$ a Dirac delta. However, this is not the case in dimensions $d\geq2$ where the Gaussian measure $\mu_0$ concentrates on distributions rather than functions. Then the terms $|u(x)|^2$ and $|u(y)|^2$ are ill-defined and the integral requires to be renormalized, which we have indicated by the notation `Ren.' in~\eqref{eq:form_F_NL}. The higher the dimension, the more difficult the renormalization procedure. For a smooth function $w$ the simplest scheme called Wick renormalization~\cite{GliJaf-87} works in dimensions $d=2,3$ and this is the situation which we consider here. 

In this work we provide the rigorous derivation of the Wick-renormalized nonlinear measures~\eqref{eq:intro_form_mu} in dimensions $d=2,3$, starting from a quantum mechanical microscopic theory without divergences, in a mean-field-type limit. Before describing this limit in detail, we review some important results about the nonlinear measure $\mu$. 

\medskip

\noindent\textbf{Nonlinear Gibbs measures} of the form~\eqref{eq:intro_form_mu} play a central role in many areas of mathematics. These measures were first defined in the 60s and 70s in the context of \emph{Constructive Quantum Field Theory}~\cite{DerGer-13,GliJaf-87,Simon-74,Summers-12}, where they were used to construct interacting quantum fields in the Euclidean framework (imaginary time), through Feynman-Kac-type formulas (then $d=d'+1$ where $d'$ is the space dimension of the quantum field). Important results in this direction include those of Symanzik~\cite{Symanzik-66}, Nelson~\cite{Nelson-73}, Glimm-Jaffe-Spencer~\cite{GliJafSpe-74} and Guerra-Rosen-Simon~\cite{GueRosSim-75}.

The same measures have later re-appeared in the study of some \emph{deterministic nonlinear partial differential equations with random initial data}, for which they are (formal) invariants of motion. This covers for instance the (renormalized) nonlinear Schr\"odinger equation
\begin{equation}\label{eq:NLS}
 i\partial_t u = -\Delta u + V_0 u + {\rm Ren.}\left\{\left(w * |u| ^2\right) u\right\}.
\end{equation}
After pioneering works by Lebowitz-Rose-Speer~\cite{LebRosSpe-88} and Bourgain~\cite{Bourgain-94,Bourgain-96,Bourgain-97} in the 1980s and 90s, this idea has been made rigorous in many recent articles including~\cite{BouBul-14a,BouBul-14b,BurThoTzv-13,ThoTzv-10,CacSuz-14,OhTho-18,Tzvetkov-08}. Most of these works consider the more complicated case where $w$ is replaced by a Dirac delta, but~\cite{Bourgain-97} deals with a smooth function $w$. Note that in this context the measure $\mu$ is often used to give a proper meaning to the renormalized equation~\eqref{eq:NLS}, which is typically well-posed for $\mu$-almost all initial data.

The Gibbs measure $\mu$ is also a central object for \emph{stochastic nonlinear partial differential equations}, where it now appears as the long-time asymptote of the random flow. This is for instance the case of the nonlinear heat equation driven by space-time white noise $\xi$
\begin{equation}
\partial_tu=-\left(-\Delta u + V_0 u + {\rm Ren.}\left\{\left(w * |u| ^2\right) u\right\}\right)+\xi, 
\label{eq:SPDE}
\end{equation}
which has been studied in many recent works including~\cite{ParWu-81,AlbRoc-91,PraDeb-03,Hairer-14,Kupiainen-16,MouWeb-15,RocZhuZhu-16,TsaWeb-18,MouWeb-17,BouDebFuk-18,CatCho-18}. The lack of regularity of typical fields drawn from the measure is related to that of the noise, which is inherited by solutions to the equation and makes the renormalization of nonlinear terms necessary. 

Finally, in \emph{statistical mechanics} nonlinear Gibbs measures of the form~\eqref{eq:intro_form_mu} are believed to describe the universal behavior of large systems close to certain \emph{phase transitions}~\cite{ZinnJustin-89,Cardy-96,ZinnJustin-13,BauBrySla-19} (at least for $w=\delta_0$). In this context the leading behavior close to the transition is often captured by mean-field theory whereas fluctuations around it are properly captured by the classical Gibbs measure $\mu$. In the Physics literature this has been predicted to happen for Bose-Einstein condensation~\cite{ArnMoo-01,BayBlaiHolLalVau-99,BayBlaiHolLalVau-01,HolBay-03,KasProSvi-01} or for Berezinskii-Kosterlitz-Thouless transitions~\cite{BisDavSimBla-09,GioCarCas-07,HolCheKra-08,HolKra-08,ProSvi-02,ProRueSvi-01,SimDavBlak-08}. For the classical Ising model rigorous mathematical results in this spirit for equilibrium states can be found in~\cite{GueRosSim-75,SimGri-73,BauBrySla-19}, whereas works about the derivation of the dynamical equation~\eqref{eq:SPDE} include~\cite{BerPreRudSaa-93,FriRud-95,GiaLebPre-99,MouWeb-17b}.

\medskip

\noindent\textbf{Main result.} Our contribution is in the spirit of the latter situation of phase transitions. We start with a model describing quantum particles in the domain $\Omega\subset\R^d$, submitted to an external potential $V$ and interacting through a pair potential $w$. We consider the Bose equilibrium state of this system at an appropriate temperature. This system has no divergence and no renormalization is needed. We then study a specific scaling regime where the (average) number of particles diverges to infinity, in which we can prove the convergence to the classical nonlinear Gibbs measure $\mu$ including the renormalization, for an appropriate potential $V_0$ which is in general different from $V$. Our result will be stated in macroscopic variables (taking the size of the full system as reference length scale) but, when interpreted at the microscopic scale (taking the typical inter-particle distance as reference length scale), the limit corresponds to zooming just below the critical density for Bose-Einstein condensation, as we will explain. The physical interpretation of our result is therefore that the measure $\mu$ describes how a Bose-Einstein condensate forms at criticality. We will in addition exhibit a kind of universality of the renormalization scheme, which turns out to be largely independent of the model. 

We now describe our main results. An ensemble of $n$ bosonic quantum particles is described by the following Schr\"odinger operator
$$H_{\lambda,\nu,n}=\sum_{j=1}^n \big(-\Delta_{x_j}+V(x_j)-\nu\big)+\lambda\sum_{1\leq j<k\leq n}w(x_j-x_k),$$
which acts on the subspace of symmetric functions in $L^2(\Omega^n,\C)$, denoted henceforth by $L^2_s(\Omega^n,\C)$. We consider Dirichlet boundary conditions for the Laplacian in case $\Omega\neq\R^d$, or periodic boundary conditions if $\Omega$ is a cube. We have introduced here two parameters: $\lambda$ is a \emph{coupling constant} allowing to vary the intensity of the interaction, which will be taken to zero in our limit, whereas $\nu$ is called the \emph{chemical potential}. It plays a decisive role in our study since it will be used to renormalize the theory. After averaging over all the particle numbers $n$ at fixed $\nu$, we obtain the following function of our parameters $\lambda,\nu$
\begin{equation}
Z(\lambda,\nu)=1+\sum_{n\geq1}\tr\left(e^{-\lambda\, H_{\lambda,\nu,n}}\right)=1+\sum_{n\geq1}e^{\lambda \nu n}\tr\left(e^{-\lambda\, H_{\lambda,0,n}}\right),
\label{eq:intro_Z}
\end{equation}
which is called the ``grand-canonical partition function''~\cite{GusSig-06,Ruelle}. The trace is taken over the symmetric space $L^2_s(\Omega^n,\C)$. This kind of Laplace transform of $n\mapsto \tr\left(e^{-\lambda\, H_{\lambda,0,n}}\right)$ is the main object of interest in this paper. 
Note that in the exponential we have multiplied our Hamiltonian by $\lambda$, which sets the system at an effective large temperature $1/\lambda\to\ii$. In this regime the average number of particles in $\Omega$ diverges:
$$\lim_{\lambda\to0^+}Z(\lambda,\nu)^{-1}\sum_{n\geq1}n\;\tr\left(e^{-\lambda\, H_{\lambda,\nu,n}}\right)=+\ii.$$
Our main result is that for a well chosen divergent function $\nu=\nu(\lambda)$ discussed below, we obtain the expansion
\begin{equation}
\boxed{\log Z\big(\lambda,\nu(\lambda)\big)=\log Z_{\rm MF}\big(\lambda,\nu(\lambda)\big)+\log z+o(1)_{\lambda\to0^+}.}
\label{eq:summary_main_result}
\end{equation}

The first term of the right side diverges very fast when $\lambda\to0^+$ and it is given by ``mean-field theory''. In our setting, mean-field theory corresponds to restricting the problem to the subclass of ``Gaussian quantum states'' also known as ``quasi-free states''~\cite{BacLieSol-94,Solovej-notes}. In the limit $\lambda\to0^+$ such Gaussian quantum states provide classical Gaussian measures and in our case this will give the reference measure $\mu_0$ in~\eqref{eq:intro_form_mu}. The most natural reference Gaussian quantum state is obtained by minimizing the free energy and then provides the partition function $Z_{\rm MF}\big(\lambda,\nu(\lambda)\big)$. Without entering too much into the details, this reference Gaussian quantum state is associated with the one-particle  Hamiltonian $-\Delta+V_\lambda$ where the potential 
\begin{equation}
V_\lambda=V-\nu(\lambda)+\lambda\rho_\lambda\ast w
\label{eq:V_lambda_intro}
\end{equation}
solves the self-consistent nonlinear equation
$$\frac{1}{e^{\lambda(-\Delta+V_\lambda)}-1}(x,x)=\rho_\lambda(x).$$
The limiting classical Gaussian measure $\mu_0$ then has covariance $(-\Delta+V_0)^{-1}$ where 
$$V_0=\lim_{\lambda\to0^+}V_\lambda.$$ 
In general the limiting potential $V_0$ may be different from $V$. We give more details about mean-field theory later in Section~\ref{sec:inhom dir}. 

The non-Gaussian classical measure $\mu$ includes the nonlinear term $\cD[u]$ and it cannot be obtained from mean-field theory. In our setting, $\mu$ arises by expanding the partition function to the next order as in~\eqref{eq:summary_main_result}. The correction involves the constant
$$z=\int e^{-\cD[u]}{\rm d}\mu_0(u)$$
which normalizes the measure $\mu$ as in~\eqref{eq:intro_form_mu} and where the renormalized interaction $\cD$ is given by~\eqref{eq:form_F_NL}. The expansion~\eqref{eq:summary_main_result} therefore provides the validity of mean-field theory to leading order, as well as fluctuations around it which involve the classical nonlinear Gibbs measure $\mu$. This is in the spirit of the Physics works mentioned above concerning the phase transition for Bose-Einstein condensation~\cite{ArnMoo-01,BayBlaiHolLalVau-99,BayBlaiHolLalVau-01,HolBay-03,KasProSvi-01}. 

The crucial role of the measure $\mu$ is better seen when looking at the quantum density matrices, which are similar to correlation functions in classical statistical mechanics. We prove below that the $k$-particle density matrix $\Gamma^{(k)}_\lambda$ of the interacting quantum system converges to its classical analogue
\begin{equation}
\boxed{\lim_{\lambda\to0^+}k! \lambda^k\, \Gamma^{(k)}_\lambda=\int |u^{\otimes k}\rangle\langle u^{\otimes k}|\,d\mu(u)} 
 \label{eq:CV_PDM_intro}
\end{equation}
for every $k\geq1$. This now characterizes completely the measure $\mu$ when $k$ is varied. This is the proper justification that the system is fully described by the nonlinear Gibbs measure $\mu$ in the limit. We defer the precise definition of the density matrix $\Gamma^{(k)}_\lambda$ to Section~\ref{def:quant stat}.

The behavior of the diverging function $\nu(\lambda)$ needed for the limit~\eqref{eq:summary_main_result} to hold depends on the dimension and on $\widehat{w}(0)$, but it is otherwise essentially universal. We take (in the second line $\zeta$ is the Riemann zeta function)
\begin{equation}
\nu(\lambda)=
\begin{cases}
\dps\frac{\hat{w} (0)}{4\pi} \log (\lambda^{-1})-\nu_0+o(1)_{\lambda\to0^+}&\text{in dimension $d=2$,}\\[0.2cm]
\dps \frac{\hat{w} (0)\,\zeta\!\left(\frac32\right)}{8\pi^{\frac32}}\frac1{\sqrt\lambda}-\nu_0+o(1)_{\lambda\to0^+}&\text{in dimension $d=3$,}
\end{cases}
\label{eq:nu_lambda_intro} 
\end{equation}
for some fixed $\nu_0$, and obtain the limit~\eqref{eq:summary_main_result} with a potential $V_0$ solving a nonlinear equation depending only on $\nu_0$. For instance, when the problem is settled on the $d$-dimensional torus with $V\equiv0$, then the limiting Gaussian measure $\mu_0$ has the covariance $(-\Delta+V_0)^{-1}$, with the constant potential $V_0$ solving the equation
\begin{equation}
\begin{cases}
\dps\frac{\hat{w} (0)}{4\pi}\log(V_0)+V_0-\hat{w} (0)\phi_2(V_0)=\nu_0&\text{in dimension $d=2$,}\\[0.2cm]
\dps\frac{\hat{w} (0)}{4\pi}\sqrt{V_0}+V_0-\hat{w} (0)\phi_3(V_0)=\nu_0&\text{in dimension $d=3$,}
\end{cases}
\label{eq:SCF_V_0_intro}
\end{equation}
for a positive decreasing function $\phi_d$ defined in Lemma~\ref{lem:equivalent} below. The first divergent term in~\eqref{eq:nu_lambda_intro} is completely independent of the model, if we except the multiplicative factor $\widehat{w}(0)$. As we will explain in Appendix~\ref{app:interpretation}, the form of this divergent term is related to the behavior of the infinite free Bose gas at criticality, that is, close to the phase transition to a Bose-Einstein condensate~\cite{Thirring}.

We have proved a result similar to~\eqref{eq:summary_main_result} in dimension $d=1$ in our previous works~\cite{LewNamRou-15,LewNamRou-18a} but no renormalization is necessary in this case and the result is much simper. A different proof based on Borel summation was later provided by Fr\"ohlich-Knowles-Schlein-Sohinger in~\cite{FroKnoSchSoh-17}. Their proof carries over to dimensions $d=2,3$ as well and provided the first derivation of the renormalized Gibbs measure $\mu$, but for technical reasons the problem had to be regularized by changing $Z(\lambda,\nu)$ in~\eqref{eq:intro_Z} into
$$\tilde Z(\lambda,\nu)=1+\sum_{N\geq1}\tr\left(\exp\left(-(1-\eta)\lambda\, H_{\frac{\lambda}{1-\eta},\nu,N}\right)\exp\left(-\eta H_{0,\nu,N}\right)\right)$$
for some $\eta\in(0,1)$~\cite{FroKnoSchSoh-17}. This amounts to pulling out of the exponential a little part of the Laplacian, which then acts as a kind of regulator. Our results deal with the physical partition function $Z(\lambda,\nu)$ in dimensions $d=2,3$ and this involves a very significant jump in difficulty, both from a conceptual and technical point of view. The earlier versions of this paper only included the 2D case~\cite{LewNamRou-18_2D} and were announced in \cite{LewNamRou-18b,LewNamRou-19}. Simultaneously to the completion of the present paper, a completely different proof for both the 2D and 3D cases has been announced by Fr\"ohlich-Knowles-Schlein-Sohinger~\cite{FroKnoSchSoh-20}. It relies on techniques from Constructive Quantum Field Theory, whereas our approach is variational. 

\medskip

\noindent\textbf{A mean-field / semi-classical limit.} It is well known that the regime $\lambda\to0^+$ (with the average number of particles tending to infinity) corresponds to a \emph{mean-field} or \emph{semi-classical limit}, where the quantum model converges towards the nonlinear Hartree model, based on the energy functional
\begin{equation}
 \cE_{\rm H}(u)=\int_\Omega \big(|\nabla u(x)|^2+V(x)|u(x)|^2\big)\,{\rm d}x+\frac12 \iint_{\Omega\times\Omega}w(x-y)|u(x)|^2|u(y)|^2\,{\rm d}x\,{\rm d}y.
 \label{eq:GP}
\end{equation}
The study of this limit is a research topic almost as old as quantum mechanics itself. It has been spectacularly rejuvenated by the birth of cold atoms physics in the 1990's, most notably by the landmark experimental observation of Bose-Einstein condensates in alkali gases~\cite{CorWie-02,Ketterle-02}. Following pioneer contributions~\cite{Hepp-74,GinVel-79,GinVel-79b,Spohn-80,FanSpoVer-80,BenLie-83,LieYau-87,PetRagVer-89,RagWer-89,Werner-92}, the last two decades have seen a great deal of progress on the derivation of such nonlinear effective models. This includes the case of minimizers (see~\cite{Lewin-15,LieSeiSolYng-05,Rougerie-LMU,Rougerie-cdf,Rougerie-20} for reviews) as well as the time-evolution of such ground states after an initial perturbation (see~\cite{BenPorSch-16,Golse-13,Schlein-08} for reviews). 

Refinements of the Hartree description have also been derived. The corresponding ``Bogoliubov approximation'' can be seen as a quantum field theory based on the Hessian of the Hartree functional~\eqref{eq:GP}. Recent results bear both on the low-lying eigenfunctions of the many-body Hamiltonian~\cite{Seiringer-11,GreSei-13,LewNamSerSol-15,NamSei-15,DerNap-14,BocBreCenSch-19_ppt,BocBreCenSch-19} and on the time-evolution thereof after an initial perturbation~\cite{GriMacMar-10,GriMacMar-11,LewNamSch-15,NamNap-17a,NamNap-17b,MitPetPic-19,BocCenSch-17,BreNamNapSch-19}.

In~\cite{LewNamRou-14} we have studied the limit of $\tr(e^{-\lambda_0H_{\lambda,\nu,n}})$ for $n\sim1/\lambda$, with a fixed $\lambda_0>0$ instead of the small parameter $\lambda$. At the leading order we obtained full Bose-Einstein condensation in the minimizer of the Hartree functional~\eqref{eq:GP}. In other words, $\lambda_0$ has no effect to this order and it is only visible in the Bogoliubov next order correction~\cite{LewNamSerSol-15}, which has the effective temperature $T_0=1/\lambda_0$. In the present work $\lambda_0$ is replaced by $\lambda$ and we obtain the renormalized Gibbs measure $\mu$, which physically models a statistical mixture of Bose-Einstein condensates and the eventual appearance of a single condensate at criticality. 

Other rigorous mathematical works on the Bose gas taking temperature into account include~\cite{BetUel-10,Seiringer-06,Seiringer-08,SeiUel-09,Yin-10,DeuSeiYng-19,DeuSei-19_ppt}. In particular, the rigorous derivation of the Bose-Einstein phase transition in interacting Bose gases still seems way out of reach, except for special lattice models~\cite[Chapter~11]{LieSeiSolYng-05} and for the trapped case in the Gross-Pitaevskii limit which was recently solved by  Deuchert-Seiringer-Yngvason~\cite{DeuSeiYng-19,DeuSei-19_ppt}. To our knowledge, the only mathematical works devoted to the study of the behavior close to the Bose-Einstein phase transition are~\cite{LewNamRou-15,LewNamRou-18a,FroKnoSchSoh-17,Sohinger-19_ppt,FroKnoSchSoh-20} for equilibrium states and~\cite{FroKnoSchSoh-19} in the one-dimensional dynamical case. 

Since we work in the same limiting regime $\lambda\to0$ as many other previous works, the emergence of the nonlinear Gibbs measure $\mu$ formally based on the Hartree energy~\eqref{eq:GP} is of course not a surprise. Similar results have been known for some time in finite dimensions~\cite{Gottlieb-05,Knowles-thesis}, where the convergence can be reformulated in terms of a usual semi-classical limit~\cite{LewNamRou-15} with no renormalization. The main difficulty is to handle the infinite dimensional case and the emergence of singular objects requiring renormalization. Another difficulty is to achieve this using only the real parameter $\nu(\lambda)$ introduced above. This is really in the spirit of renormalization in Quantum Field Theory, as initiated by Dyson in~\cite{Dyson-49b} and further developed within statistical physics using renormalization group techniques~\cite{Wilson-75,Delamotte-04,ZinnJustin-13,BauBrySla-19}.

\medskip

\noindent\textbf{Method of proof.}
Our mathematical approach in this paper is \emph{variational}, like in~\cite{LewNamRou-15,LewNamRou-18a}. We crucially use that the equilibrium Gibbs quantum state as well as the measure $\mu$ are the unique solutions to some minimization problems and our goal is to prove the convergence of the quantum problem to the classical one. The way to connect quantum objects (positive self-adjoint operators with unit trace on a Hilbert space) to classical ones (probability measures on a space of functions or distributions) is via so-called de Finetti measures (or Wigner measures, depending on the point of view)~\cite{Ammari-hdr,Rougerie-LMU,Rougerie-cdf}. This technique generalizes ideas from semi-classical analysis to infinite dimensions, cf~\cite{AmmNie-08,LewNamRou-14,LewNamRou-15}. 

Our main goal is therefore to show that the limiting de Finetti measure of the quantum problem minimizes the variational problem characterizing the nonlinear Gibbs measure $\mu$, hence must be equal to $\mu$. The difficulty here is that $\mu$ is a very singular object and that its variational characterization involves the renormalized interaction $\cD$. Passing to the limit requires a fine understanding of the way that  singularities appear in the quantum de Finetti measure when $\lambda\to0^+$. In our case, this reduces to finding good estimates on the high-momentum part of the one- and two-particle density matrices $\Gamma^{(1)}_\lambda$ and $\Gamma^{(2)}_\lambda$. We achieve this goal by using two new inequalities of independent interest. 

For $\Gamma^{(1)}_\lambda$ we prove an inequality which controls the difference of two one-particle density matrices in terms of the quantum relative entropy of the corresponding quantum states in Fock space (one of them being Gaussian). This takes the simple form
\begin{equation}
\tr\left|h^{1/2}\big(\Gamma^{(1)}-\Gamma^{(1)}_0\big)h^{1/2}\right|^2\leq 4\left(\sqrt2 \sqrt{\cH(\Gamma,\Gamma_0)}+\cH(\Gamma,\Gamma_0)\right)^2,
 \label{eq:estim_HS_relative_entropy_intro}
\end{equation}
see Theorem~\ref{thm:estim_relative_entropy} below. Here $\Gamma_0$ is any Gaussian (a.k.a.~quasi-free) quantum state over the Fock space~\cite{BacLieSol-94,Solovej-notes}, of one-body Hamiltonian $h$, which in practice is taken to be the mean-field solution. On the other hand, $\Gamma$ is an arbitrary state and $\cH(\Gamma,\Gamma_0)=\tr\Gamma(\log\Gamma-\log\Gamma_0)$ is the quantum relative entropy. Related bounds were independently derived by Deuchert-Seiringer-Yngvason in~\cite[Lemma~4.1]{DeuSeiYng-19} and~\cite[Lemma~4.1]{DeuSei-19_ppt}. The important difference here is that we are able to include the operator $h$ explicitly. This is all explained in Section~\ref{sec:relative entropy bound}.

The two-particle density matrix $\Gamma^{(2)}_\lambda$ is way more difficult to handle. To deal with it we prove another general inequality which could be useful in other contexts and occupies the whole Section~\ref{sec:variance-by-first-moment}. It is one of the main new ingredients of this paper. This bound is 
\begin{equation}
\frac{\tr\left(A^2e^{-H}\right)}{\tr(e^{-H})}\leq \frac{17\eta}{a} 
\label{eq:2body_by_1body}
\end{equation}
for all $0<a\leq 1$ and all bounded operators $A$, assuming
\begin{equation}
 \eta:= \sup_{\eps\in[-a,a]}\frac{\left|\tr\left(Ae^{-H+\eps A}\right)\right|}{\tr(e^{-H+\eps A})}+a\big\|[[H,A],A]\big\|\sqrt{1+\|A\|^2}\leq 1,
 \label{eq:formula_eta_intro}
\end{equation}
see Theorem~\ref{thm:general-variance-by-linear-response} below. We have simplified things a little bit here, for the sake of exposition. Our full estimate relies on a more complicated $\eta$ which provides a tighter bound. It can also be stated without assuming $a,\eta \leq 1$, at the price of a more complicated right-hand side in~\eqref{eq:2body_by_1body}. In our application we use this for $A$ a one-particle operator over the Fock space, so that the expectation against $A^2$ gives us access to the two-particle density matrix $\Gamma^{(2)}_\lambda$. What~\eqref{eq:2body_by_1body} then says is that one can control two-particle expectations of a Gibbs state by one-particle expectations, at the expense of perturbing the corresponding Hamiltonian $H$ by $\eps A$ for $\eps$ in a small window $[-a,a]$. The error then solely depends on the commutator $[[H,A],A]$. Note that our more precise inequality in Theorem~\ref{thm:general-variance-by-linear-response} involves the quadruple commutator $[[[[H,A],A],A],A]$ as well and it is the one useful in our context. 

We can call~\eqref{eq:2body_by_1body} a ``variance'' or ``correlation'' inequality. If we replace $A$ by $A-\tr(Ae^{-H})/\tr(e^{-H})$, then the left side of~\eqref{eq:2body_by_1body} is now exactly the variance of $A$, whereas the supremum on the right side of~\eqref{eq:formula_eta_intro} will typically be small (the function whose supremum is taken vanishes for $\eps=0$). Correlation inequalities have historically played an important role in statistical physics~\cite{Griffiths-67,Ginibre-70,ForKasGin-71,Ginibre-72,FroSimSpe-76,DysLieSim-78} and the proof of~\eqref{eq:2body_by_1body} uses and/or improves several important estimates from the literature. Note that the difficulties we face here are (almost) purely of a quantum nature. Estimates like \eqref{eq:2body_by_1body}  are significantly easier in a classical theory, or when $A$ and $H$ commute.

Inequalities of the type of~\eqref{eq:2body_by_1body} are reminiscent of the fluctuation-dissipation theorem and our proof is indeed inspired by linear response theory \`a la Kubo~\cite{Kubo-66,KubTodHas-91}. We use the fact that one can access the expectation of $A^2$ (variance/fluctuation) by differentiating expectations of $A$ in the perturbed Gibbs state with Hamiltonian $H-\eps A$ (linear response/dissipation). In classical statistical mechanics the relation between variance and linear response is an identity. In the quantum case, differentiating leads  to the {\em Duhamel two-point function}\footnote{Also known as canonical correlation or Bogoliubov scalar product.} instead of the variance. There exists known relations between these quantities, for instance the Falk-Bruch inequality~\cite{FalBru-69}. Thanks to such inequalities (that we revisit here), the variance is under good control if one can control the derivative of the first moment. Ideally, a strong bound of the form
\begin{equation} \label{eq:too-strong-bound}
\frac{\tr\left(Ae^{-H+\eps A}\right)}{\tr \left( e ^{-H + \eps A} \right)}-\frac{\tr\left(Ae^{-H}\right)}{\tr(e^{-H})} = \eps \times {\rm error}
\end{equation}
could be used,  as in the seminal work~\cite{DysLieSim-78} of Dyson-Lieb-Simon on the phase transition of the Heisenberg antiferromagnet. A bound of the type~\eqref{eq:too-strong-bound} (for a particular $A$) was there obtained using reflection positivity, see~\cite[Theorem~4.2]{DysLieSim-78}. For classical systems this goes back to an earlier breakthrough of Fr\"ohlich-Simon-Spencer \cite{FroSimSpe-76}. However~\eqref{eq:too-strong-bound}  is not available to us, and our new inequalities~\eqref{eq:estim_HS_relative_entropy_intro}~\eqref{eq:2body_by_1body} serve as a replacement. The former essentially yields~\eqref{eq:too-strong-bound} without an $\eps$ factor on the right-hand side, which is a sufficient bound for our purpose, once inserted in~\eqref{eq:2body_by_1body}. 

Inequality~\eqref{eq:2body_by_1body} is, perhaps, our main contribution, and we refer to Section~\ref{sec:variance-by-first-moment} for a proof. Briefly, bounds on the discrepancy between variance and linear response (Duhamel two-point function) give an estimate on the average of the ``perturbed variance'' $\tr\left(A^2e^{-H+\eps A}\right)$ over a small window in $\eps\in [-a,a]$. To get rid of the averaging, we prove that this function is approximately convex in $\eps$. In the classical case, when $A$ and $H$ commute, the convexity is obvious as the second derivative in $\eps$ is $\tr\left(A^4e^{-H+\eps A}\right) \ge 0$. In the quantum case, we will prove a lower bound for the second derivative in terms of several commutators. 
%

With~\eqref{eq:estim_HS_relative_entropy_intro} and~\eqref{eq:2body_by_1body} at hand, we are able to control in Section~\ref{sec:correl} the correlations in our quantum Gibbs state at high energies. High energies are exactly where renormalization takes place and the estimates will tell us that the true quantum state has there essentially the same behavior as the mean-field one, which has already been studied in detail in~\cite{FroKnoSchSoh-17}. On the other hand, at low energies we use a quantitative de Finetti theorem from~\cite{ChrKonMitRen-07,LewNamRou-15} which gives explicit bounds on the difference between the quantum problem and the classical one, allowing to pass to the limit in the variational problem. This is the general idea of our approach for proving~\eqref{eq:summary_main_result}.

In this paper, we first discuss the case where $\Omega=(0,1)^d$ with periodic boundary conditions (an interacting Bose gas on the unit torus), which is easier to state because the density is always constant. This case clarifies the link with phase transitions in the Bose gas in the thermodynamic limit, after re-interpreting our result in microscopic variables, see Appendix~\ref{app:interpretation}. Then we turn to the more complicated case of $\Omega=\R^d$ with a potential $V$ diverging fast enough at infinity, which requires the introduction of mean-field theory as described above. We only make comments about the case $\Omega\subsetneq\R^d$.

In the next section we properly define the quantum and classical models and we give some hints on the relation between the two. Then, in Section~\ref{sec:results} we state all our results. Section~\ref{sec:strategy} contains a detailed explanation of the strategy of proof, which is then carried over in the rest of the paper.

\subsubsection*{\bf Acknowledgements.} Insightful discussions with J\"urg Fr\"ohlich, Markus Holzmann, Antti Knowles, Benjamin Schlein, Robert Seiringer, Vedran Sohinger,  Jan Philip Solovej, Laurent Thomann, Daniel Ueltschi and Jakob Yngvason are gratefully acknowledged. This project has received funding from the European Research Council (ERC) under the European Union's Horizon 2020 Research and Innovation Programme (Grant agreements MDFT No 725528 and CORFRONMAT No 758620), and Deutsche Forschungsgemeinschaft (DFG, German Research Foundation) under Germany's Excellence Strategy (EXC-2111-390814868).

\section{Setting and definitions}\label{sec:defs}

We recap here all the standard and less-standard notions needed to state and discuss our main results. Perhaps the acquainted reader will want to jump directly to Section~\ref{sec:results}, and come back to this one if some notation is unclear later.

\subsection{Fock space formalism}\label{sec:defs Fock}

Our basic Hilbert space for one particle is
\begin{equation}\label{eq:Hilbert space}
\gH =  L^2 (\Omega)
\end{equation}
with $\Omega$ an open domain in $\R^d$. The reader might think of the two model cases, when $\Omega=(0,1)^d$ (to which we add periodic boundary conditions, which is then the same as taking $\Omega=\bT^d$, the torus) or the full space $\R^d$. 

For the many-body problem we work grand-canonically, that is, we do not fix the particle number. The many-body Hilbert space is thus the {\em bosonic Fock space} 
\begin{equation}\label{eq:Fock}
\gF = \C \oplus \gH \oplus \ldots \oplus \gH ^{\otimes_s n} \oplus \ldots 
\end{equation}
The symbol $\otimes_s n $ stands for the $n$-fold symmetric tensor product, as is appropriate for the $n$-body configuration space of bosons. Operators acting on finitely many particles are lifted to the Fock space in the usual way:

\begin{definition}[\textbf{Second quantization}]\label{def:quantiz}\mbox{}\\
Let $A_k$ be a self-adjoint operator on $\gH^{\otimes_s k}$. We define its action on the Fock space as  
\begin{equation}\label{eq:2nd_quantized}
\mathbb{A}_k:=0\oplus\cdots\oplus\bigoplus_{n= k } ^\infty \left(\sum_{1\leq i_1 < \ldots < i_k \leq n} (A_k)_{i_1,\ldots,i_k}\right) 
\end{equation}
where $(A_k)_{i_1,\ldots,i_k}$ denotes the operator $A_k$ acting on the variables labeled $i_1,\ldots,i_k$ in $\gH^{\otimes_s n}$.

\hfill$\diamond$
\end{definition}

When $k=1$, it is customary to use the notation
\begin{equation}\label{eq:dGamma}
\dGamma (A):=\mathbb{A}= 0\oplus\bigoplus_{n= 1 } ^\infty \left(\sum_{1\leq i \leq n} A_{i}\right) 
\end{equation}
for one-body operators, a tradition that we will also follow throughout. For example,  the {\em particle number operator} is 
$$ 
\cN = \dG (\1_\gH) =\bigoplus_{n=0}^\infty n.
$$
Next, quantum states are as usual:

\begin{definition}[\textbf{Quantum states and reduced density matrices}]\label{def:quant stat}\mbox{}\\
A pure state is an orthogonal projection $|\Psi\rangle \langle \Psi|$ on some normalized vector $\Psi$ of the Fock space $\gF$. A mixed state $\Gamma$ is a convex superposition of pure states, i.e. a positive trace-class operator on $\gF$ with unit trace. We denote 
\begin{equation}
 \cS \left(\gF\right) := \left\{ \Gamma \mbox{ self-adjoint operator on } \gF, \: \Gamma \geq 0, \: \tr_{\gF} [\Gamma] = 1 \right\}
\end{equation}
the set of all mixed states. 

The reduced $k$-body density matrix $\Gamma^{(k)}$ of a state $\Gamma$ is the operator on $\gH ^{\otimes_s k}$ defined by duality as 
\begin{equation}\label{eq:def k body}
\tr_{\gH^{\otimes_s k}} \left[ A_k \Gamma ^{(k)} \right]:= \tr_\gF \left[ \mathbb{A}_k \Gamma \right] 
\end{equation}
for any self-adjoint operator $A_k$ on $\gH^{\otimes_s k}$, with $\mathbb{A}_k$ the second quantization~\eqref{eq:2nd_quantized} of $A_k$.\hfill$\diamond$
\end{definition}

If $\Gamma$ is of the diagonal form 
$$ \Gamma = \Gamma_0 \oplus \Gamma_1 \oplus \ldots \oplus \Gamma_n \oplus \ldots $$
then the reduced density matrices are equivalently given via partial traces as 
$$ \Gamma ^{(k)} = \sum_{n\geq k } {n \choose k }  \tr_{k+1\to n} [ \Gamma_n ].$$
Also recall that the expected particle number of a state is given as
$$\tr_\gF [\cN \Gamma] = \tr_\gH [ \Gamma^{(1)}].$$

A \emph{quantum Gibbs state} is a state of the special form 
$$\Gamma=e^{-\bH}/\tr_\gF(e^{-\bH})$$
where $\bH$ is a self-adjoint operator on the Fock space $\gF$ such that $\tr_\gF(e^{-\bH})<\ii$. A \emph{Gaussian quantum state} or \emph{quasi-free state} corresponds to the case where $\bH=\dG(h)$ for a one-particle operator $h$ with $h>0$ and $\tr_\gH(e^{-h})<\ii$. More about quasi-free states can be read in~\cite{BacLieSol-94,Solovej-notes} and in Section~\ref{sec:inhom dir}. They are called Gaussian because $\bH=\dG(h)$ is a quadratic operator in the bosonic creation/annihilation operators, the definition of which we now recall. 

\begin{definition}[\textbf{Creation/annihilation operators}]\label{eq:def ani cre}\mbox{}\\
Let $f\in \gH$. The associated annihilation operator acts on the Fock space as specified by
$$
a(f) u_1 \otimes_s \ldots \otimes_s u_n = n ^{-\frac{1}{2}} \sum_{j=1} ^n \langle f , u_j\rangle u_1 \otimes_s \ldots u_{j-1} \otimes_s  u_{j+1} \otimes_s \ldots \otimes_s u_n
$$
and then extended by linearity. Its formal adjoint, the creation operator $a^{\dagger} (f)$ acts as 
$$
a^\dagger (f) u_1 \otimes_s \ldots \otimes_s u_n = (n+1) ^{\frac{1}{2}}  f \otimes_s u_1 \otimes_s \ldots  \otimes_s u_n.
$$
The canonical commutation relations (CCR) hold: for all $f,g\in\gH$
\begin{equation}\label{eq:CCR}
[a(f),a(g)] = [a^{\dagger} (f),a^{\dagger} (g)] = 0, \quad [a(f),a^{\dagger} (g)] = \langle f , g \rangle. 
\end{equation}
\hfill$\diamond$
\end{definition}

The reduced density matrices of a state $\Gamma$ can alternatively be defined by the relations 
\begin{equation}\label{eq:DM ann cre}
 \left\langle g_1\otimes_s \ldots \otimes_s g_k ,\Gamma ^{(k)} f_1\otimes_s \ldots \otimes_s f_k \right\rangle = \tr\left[ a^{\dagger} (f_1) \ldots a^{\dagger} (f_k) a (g_1) \ldots a (g_k) \Gamma \right].
\end{equation}
For a one-body self-adjoint operator $h>0$ with $\tr_\gH(e^{-h})<\ii$ (hence of compact resolvent) diagonalized in the form $h=\sum_j h_j|u_j\rangle\langle u_j|$, we can then express
$$\dG(h)=\sum_j h_j\,a^\dagger(u_j)\,a(u_j)$$
which is quadratic as was mentioned above.

\subsection{Quantum model}\label{sec:defs quant}

The many-body operators we shall study are of the form 
\begin{equation}\label{eq:many body hamil}
\boxed{\bH_\lambda = \dGamma (h) + \lambda\mathbb{W} - \nu(\lambda) \cN + E_0(\lambda).}
\end{equation}
Here $h$ is a self-adjoint operator on $\gH$ such that $\tr_\gH(e^{-\beta h})<\ii$ for all $\beta>0$. The reader might think of the case when 
$$h=\begin{cases}
  -\Delta &\text{on $\bT^d$,}\\
  -\Delta + V& \text{on $\R^d$.} 
  \end{cases}$$
The interaction term $\mathbb{W}$ is the second quantization of the multiplication operator by $w(x-y)$ on the two-body space $\gH ^{\otimes_s 2}$~:
\begin{equation}
\mathbb{W}:=0\oplus0\oplus\bigoplus_{n= 2 } ^\infty \left(\sum_{1\leq i < j \leq n} w(x_i-x_j)\right).
\label{eq:def_interaction}
\end{equation}
The coupling constant $\lambda>0$ models the interaction strength and the chemical potential $\nu(\lambda)$ will be tuned to serve as a counter-term. The constant $E_0(\lambda)$ is just an energy shift, which we will use in order for the renormalized interaction 
\begin{equation}\label{eq:intro quant interaction}
\boxed{\lambda \bWren = \lambda \mathbb{W} -\nu(\lambda) \cN + E_0(\lambda)}
\end{equation}
to stay a positive operator. The {\em quantum Gibbs state} associated with the above Hamiltonian is the unique minimizer of the free-energy functional (energy minus temperature times entropy) 
\begin{equation}\label{eq:free ener func}
\cF _{\lambda,T} [\Gamma] = \tr\left[ \bH_\lambda \Gamma \right] + T \tr [\Gamma \log \Gamma]  
\end{equation}
over all quantum states $\Gamma$ on the Fock space. Explicitly
\begin{equation}\label{eq:intro quantum Gibbs}
\boxed{\Gamma_{\lambda} = \frac{1}{\cZ(\lambda)} \exp\left( - \frac{\bH_{\lambda}}{T}  \right)  }
\end{equation}
where the {\em partition function} $\cZ(\lambda)$ normalizes the state,
$$ \cZ(\lambda)= \tr \left[ \exp\left( - \frac{\bH_{\lambda}}{T}  \right) \right],$$
and satisfies 
\begin{equation}\label{eq:intro part func}
\boxed{F_{\lambda} :=\min_{\Gamma\in  \cS \left(\gF\right)}\cF _{\lambda,T} [\Gamma]= -T \log \cZ(\lambda).}
\end{equation}
In the whole paper we work in the regime
$$\boxed{T=\frac1\lambda \to \infty}$$
with $\nu(\lambda)$ and $E_0(\lambda)$ appropriately tuned. Note that the energy shift $E_0$ does not modify the Gibbs state. We could take instead $T=T_0/\lambda$ with a fixed $T_0>0$, and this would eventually lead to a nonlinear classical Gibbs measure $\mu_{T_0}$ at temperature $T_0$. For simplicity we only discuss the case $T_0=1$ and retain  $\lambda\to0^+$ as our sole parameter. 

\subsection{Classical model}\label{sec:defs class}

Let us briefly recap the definitions related to the nonlinear Gibbs measure. More details are in Section~\ref{sec:renormalized-measure}.
Consider a one-particle self-adjoint operator $h_0>0$ of compact resolvent, with the following spectral decomposition:
\begin{equation}\label{eq:spectral h}
 h_0 = \sum_{j=1} ^\infty \lambda_j |u_j \rangle \langle u_j |.
\end{equation}
We introduce the scale of spaces 
$$ 
\gH ^s = \left\{ u = \sum_{j=1} ^\infty \alpha_j u_j, \quad \sum_{j=1} ^{\infty} |\alpha_j| ^2 \lambda_j ^{s/2} < \infty \right\}.
$$
The Gaussian probability measure $\mu_0$ of covariance $h^{-1}$ is by definition given by
\begin{align}\label{eq:def_mu0}
d\mu _0(u) := \mathop  \bigotimes \limits_{i = 1}^\infty  \left( \frac{\lambda_i}{\pi}e^{ - {\lambda} _i|\alpha_i|^2}\,d\alpha_i \right)
\end{align}
with $\alpha_i=\langle u_i, u\rangle$ and $d\alpha=d\Re(\alpha)\,d\Im(\alpha)$ the Lebesgue measure on $\C\simeq\R^2$. The formula~\eqref{eq:def_mu0} must be interpreted in the sense that the cylindrical projection of $\mu_0$ onto the finite-dimensional space 
${\rm Span}\{u_1,...,u_K\}$
is given by 
\begin{align}\label{eq:def_mu0K}
d\mu _{0,K}(\alpha_1,...\alpha_K) := \prod_{i = 1}^K \left( \frac{\lambda_i}{\pi}e^{ - {\lambda} _i|\alpha_i|^2}\,d\alpha_i \right)
\end{align}
for every $K\geq1$. Assuming that for some $p>0$ 
\begin{equation}\label{eq:Schatten h} 
\tr[h_0 ^{-p}] < \infty,
\end{equation}
the limit measure $\mu_0$ then concentrates on $\gH ^{1-p}$~\cite[Section~3.1]{LewNamRou-15}. In the cases of interest to this paper we have $h_0=-\Delta+V_0$ for some $V_0$ on $\Omega\subset\R^d$. Therefore $\gH^s$ is a kind of Sobolev space and, since we have $p>1$, $\mu_0$ is supported on distributions with negative Sobolev regularity, whence the need for renormalization in the definition of the interacting measure. 

Let $P_K$ be the orthogonal projector on ${\rm Span}\{u_1,...,u_K\}$. Consider the interaction energy with local mass renormalization
\begin{equation}\label{eq:def int}
 \FNL_K [u] =  \frac{1}{2} \iint_{\Omega \times \Omega} \left( |P_K u(x)| ^2 - \left\langle |P_K u(x)| ^2 \right\rangle_{\mu_0} \right) w (x-y) \left( |P_K u(y)| ^2 - \left\langle |P_K u(y)| ^2 \right\rangle_{\mu_0} \right) dx dy. 
\end{equation}
Here, for any $f\in L ^1 (d\mu_0)$,  
\begin{equation}\label{eq:intro class exp}
\left\langle f(u) \right\rangle_{\mu_0} := \int f (u) d\mu_0 (u)  
\end{equation}
denotes the expectation in the measure $\mu_0$. We shall assume that 
\begin{equation}\label{eq:w-hatw}
w(x)=\int_{\Omega^*} \widehat w(k) e^{ik \cdot x} dk
\end{equation}
where the Fourier transform $\hat{w}$ satisfies 
\begin{equation}\label{eq:asum w 1}
0 \leq \hat{w} (k) \in L^1 (\Omega^*).
\end{equation}
Here by convention $\Omega^*=\R^d$, except if  $\Omega=\bT^d$ then $\Omega^*=2\pi \mathbb{Z}^d$ (and the integral in \eqref{eq:w-hatw} becomes a sum). Then, as recalled in Lemma~\ref{lem:re-interaction} below, when $\widehat{w}\geq0$ and $\tr(h_0^{-2})<\ii$, the sequence $\FNL_K [u]$ converges to a limit $\FNL[u]$ in $L^1 (d\mu_0)$, hence we may define the renormalized interacting probability measure by
\begin{equation}\label{eq:NL measure}
\boxed{d\mu (u) := z^{-1} \exp\left(- \FNL [u] \right) d\mu_0 (u)}  
\end{equation}
with $0<z<\infty$ a normalization constant (to make $\mu$ a probability measure).

Note that the reduced one-body density matrix 
\begin{equation}\label{eq:DM classical}
\gamma ^{(1)}_\mu := \int |u  \rangle \langle u | d\mu (u) 
\end{equation}
is a priori an operator from $\gH ^{p-1}$ to $\gH ^{1-p}$ (since $|u  \rangle \langle u |$ is not any better, $\mu$-almost surely). However, averaging with respect to $\mu$ has a regularizing effect, so that $\gamma^{(1)}_\mu$ turns out to be a compact operator from $\gH$ to $\gH$. In fact, one can show that 
\begin{equation}\label{eq:DM classical Green}
\gamma_{\mu_0} ^{(1)} =h_0^{-1},
\end{equation}
which is called the covariance of $\mu_0$. Similarly, the reduced $k$-body density matrix 
\begin{equation}\label{eq:DM classical bis}
\gamma ^{(k)}_{\mu_0} := \int |u ^{\otimes k} \rangle \langle u ^{\otimes k} | d\mu (u) =k!\,P_s^k (h_0^{-1})^{\otimes k} P_s ^k
\end{equation}
belongs to the $p$-th Schatten class $\gS ^{p} (\gH ^{\otimes_s k})$, see~\cite[Lemma~3.3]{LewNamRou-15}, and the same is true of the reduced $k$-body matrix of the interacting measure. In the right-hand side of \eqref{eq:DM classical bis}, $P_s ^k$ denotes the orthogonal projector on the symmetric subspace.

\subsection{Formal quantum/classical correspondence}\label{sec:defs formal}

Our aim is to relate the quantum Gibbs state \eqref{eq:intro quantum Gibbs} to the classical Gibbs measure \eqref{eq:NL measure}. If we ignore the renormalizing terms for the moment (in particular, think of $\nu=E_0 = 0$), the formal correspondence between the two objects can be seen as follows. Recall that the Gibbs measure can be interpreted as a rigorous version of the formal
\begin{equation}\label{eq:discu class Gibbs}
d\mu(u)=\text{``}\;Z^{-1}e^{-\cE_{\rm H}[u]}\,d u \;\text{''}
\end{equation}
with $\cE_{\rm H}[u]$ the nonlinear Hartree energy functional 
$$
\cE_{\rm H}[u]= \int_\Omega \overline{u(x)} (hu)(x) \,d x  +\frac{1}{2}\iint_{\Omega \times \Omega} |u(x)|^2\, w(x-y)|u(y)|^2 \,dx\,dy.
$$

Define the quantum fields (operator-valued distributions) $a^{\dagger }(x),a(x)$, creating/annihilating a particle at position $x$ by the formulae
\begin{equation}\label{eq:cre anih local}
a (f) = \int a (x) f(x) dx, \quad a ^{\dagger}(f) = \int a ^{\dagger} (x) f(x) dx
\end{equation}
for all $f\in \gH$. Inherited from \eqref{eq:CCR} we have the canonical commutation relations 
\begin{equation}\label{eq:local CCR}
 \left[a (x), a  (y)\right] = [a ^{\dagger} (x), a ^{\dagger} (y)] = 0, \quad [a (x), a ^{\dagger} (y)] = \delta_{x=y}. 
\end{equation}
These operator-valued distributions allow us to rewrite the many-body Hamiltonian with $\nu=E_0=0$ as  
\begin{equation}\label{eq:hamil field ope}
\lambda \bH_\lambda = \lambda\int_{\Omega} a^{\dagger }(x) h_x a (x) \,d x+ \frac{\lambda^2}{2} \iint_{\Omega \times \Omega} a^{\dagger} (x) a ^{\dagger} (y) w (x-y) a (x) a (y) \,dx\,dy.  
\end{equation}
The formal manipulation relating~\eqref{eq:intro quantum Gibbs} and~\eqref{eq:discu class Gibbs} is then to replace the quantum fields $a^{\dagger }(x),a (x) $ by classical fields, i.e. operators by functions. This involves in particular that the commutation relations~\eqref{eq:local CCR} become trivial in some limit, all fields commuting at any position. How this can come about is further explained in~\cite[Section~5.2]{LewNamRou-15} and the introduction to~\cite{FroKnoSchSoh-17} (in these works the link between the classical and quantum problems has been made rigorous in 1D). Basically, the order of magnitude of commutators stays fixed by definition, but the typical value of the 
fields $a(x)$ and $a^\dagger(x)$ is of order $\lambda^{-1/2}$ when computing expectations against the quantum Gibbs state. This suggests to introduce new fields $b(x)=\sqrt\lambda\, a(x)$ and $b^\dagger(x)=\sqrt\lambda\,  a^\dagger(x)$.
This is now a clean semi-classical limit, since the commutators of the new fields is of order $\lambda\to0$. 

Let us discuss now the inclusion of counter-terms. It is useful to write the mean-field interaction, using Fourier variables, 
\begin{equation}\label{eq:intro Fourier int}
\iint_{\Omega \times \Omega} |u(x)|^2\, w(x-y)|u(y)|^2 \,dx\,dy = \int \hat{w} (k) \left| \widehat{|u| ^2} (k) \right|^2 dk =  \int \hat{w} (k) \left|\int |u(x)| ^2 e^{ik\cdot x} \right| ^2 dk.
\end{equation}
On the other hand, the quantum interaction can be expressed as 
\begin{equation}\label{eq:intro Fourier quant int}
 \mathbb{W} = \frac{1}{2}\int \hat{w} (k) \left|\dGamma (e^{ik\cdot x})  \right| ^2 dk - \frac{w (0)}{2} \cN
\end{equation}
where the second term is typically of lower order and may be ignored. Thus, one formally obtains the quantum interaction by replacing 
$$ \int |u(x)| ^2 f (x) dx \rightsquigarrow \dG (f)$$
with $f(x) = e^{ik\cdot x}$, identified with the corresponding multiplication operator on $\gH$. 

To see how to include the renormalization, observe that \eqref{eq:def int} formally leads to 
$$ \FNL [u] = \text{``}\;\frac{1}{2} \int \hat{w} (k) \left| \widehat{|u| ^2} (k) - \left\langle \widehat{|u| ^2} (k) \right\rangle_{\mu_0} \right|^2 dk\; \text{''}.$$
Thus the appropriate renormalized quantum interaction should be 
\begin{equation}\label{eq:discu renorm quant}
  \bWren = \frac{1}{2}\int \hat{w} (k) \left|\dGamma (e^{ik\cdot x}) - \left\langle \dGamma (e^{ik\cdot x}) \right\rangle_{\Gamma_{0}} \right| ^2 dk. 
\end{equation}
After expanding the square, this suggests a natural choice for the chemical potential $\nu(\lambda)$ and the energy shift $E_0(\lambda)$, as we will see. Making the above formal quantum/classical correspondence rigorous is the goal of our paper. 

Note that here we use Fourier variables mostly for convenience. What they help accomplish is rewriting interactions (two-body terms) as sums of products of one-body terms. Other methods to accomplish this, such as Fefferman-de la Llave type decompositions~\cite{FefLla-86,HaiSei-02} could replace the Fourier transform.

\smallskip 

In the next two sections we present our main results in the following order: 

\smallskip 

\noindent $\bullet$ \emph{Homogeneous case}. We consider the emblematic case where $\Omega=\bT^d$ and $h=-\Delta$. Since the system is translation-invariant, the density is always constant. Modulo an appropriate choice of parameters $\nu(\lambda),E_0(\lambda)$, the many-body interaction in~\eqref{eq:intro quant interaction} can be made to coincide with~\eqref{eq:discu renorm quant}. This amounts to using as reference the mean-field quasi-free state, which is determined by one constant ``potential'' $V_\lambda$ solving a simple equation. The final reference Gaussian measure has the covariance $h_0^{-1}=(-\Delta+V_0)^{-1}$ with $V_0=\lim_{\lambda\to0}V_\lambda$ a constant solving~\eqref{eq:SCF_V_0_intro}. Theorem~\ref{thm:main-1} provides a rigorous connection between the classical renormalized and quantum problems. 

\smallskip

\noindent $\bullet$ \emph{Inhomogeneous case}. We consider here the case
$$h=-\Delta+V(x)$$
where $V(x)\to+\ii$ when $|x|\to\ii$. The correct reference Gaussian measure has the covariance $h_0^{-1}$ where $h_0=-\Delta+V_0(x)$ for a potential $V_0$ which solves a nonlinear nonlocal equation. First we reinterpret the results of~\cite{FroKnoSchSoh-17} on this Gaussian measure, in light of the mean-field approximation at the quantum level. Then we state our main result, Theorem~\ref{thm:main-2}, on the mean-field limit, using the optimal quasi-free quantum energy as a reference.

\smallskip 

\noindent $\bullet$ \emph{Inhomogeneous case, inverse statement}. It is also possible to start with a one-particle Hamiltonian $h$ and modify the interaction as in~\eqref{eq:discu renorm quant}. We then do not have to solve any nonlinear equation and in the limit we end up with the interacting measure based on the Gaussian measure associated with $h$. This we call an \emph{inverse statement} because we have to modify the initial quantum model such as to find the desired measure in the limit. This is less natural from a physical point of view. Nevertheless, it turns out that the previous \emph{direct statement} where one starts with $h$ and identifies what the limiting measure is, follows from our proof of the inverse statement and the results of~\cite{FroKnoSchSoh-17} on the nonlinear equation. So the proof of the inverse statement is indeed our main new contribution. It occupies most of the paper. We are able to prove an abstract statement, Theorem~\ref{thm:main-3}, which covers a very large class of one-particle Hamiltonians, including $h=-\Delta+V(x)$ in $\R^d$ for a potential $V$ growing sufficiently fast at infinity, and $h= - \Delta$ on a bounded domain. 

\newpage

\section{Main results}\label{sec:results}

\subsection{Homogeneous gas}\label{sec:hom} 

We first consider the case where
$$\boxed{h =-\Delta \,\text {  on  } \, \Omega=\bT^d,\quad d=2, 3} $$
with $-\Delta$ the usual Laplace-Beltrami operator on the torus.

It will be easier to parametrize everything in terms of the reference Gaussian quantum state. For any fixed $\kappa>0$, let 
\begin{equation}
 \Gamma_{0}=\cZ_0(\lambda)^{-1}e^{-\lambda\,\dG(-\Delta+\kappa)},\qquad \cZ_0(\lambda)=\tr\left(e^{-\lambda\,\dG(-\Delta+\kappa)}\right).
 \label{eq:free_Torus}
\end{equation}
Its expected particle number is (see Lemma~\ref{lem:equivalent})
\begin{align}
N_0(\lambda)&:= \tr_{\gF} \left[ \cN \Gamma_{0} \right] = \sum_{k\in 2\pi \Z^d} \frac{1}{e^{\lambda(|k|^2+\kappa)} -1}\nn\\
&= \frac{1}{\lambda^{\frac{d}2}(2\pi)^d}\int_{\R^d} \frac{dk}{e^{|k|^2+\lambda\kappa} -1}+O(\lambda^{-1})
=\begin{cases}\displaystyle
-\frac{\log(\lambda)}{4\pi\lambda}+O(\lambda^{-1})&\text{for $d=2$,}\\[0.2cm]
\displaystyle\frac{\zeta(3/2)}{8\pi^{\frac32}\lambda^{\frac32}}+O(\lambda^{-1})&\text{for $d=3$.}
  \end{cases}
  \label{eq:nu_lambda}
\end{align}
Let now 
\begin{equation}
\Gamma_\lambda=\cZ(\lambda)^{-1}e^{-\lambda\bH_\lambda},\qquad \cZ(\lambda)=\tr\left(e^{-\lambda\bH_\lambda}\right) 
 \label{eq:interacting_Torus}
\end{equation}
be the interacting Gibbs state with $\bH_\lambda = \dGamma (h) + \lambda\mathbb{W} - \nu(\lambda) \cN + E_0(\lambda)$ as in~\eqref{eq:many body hamil}, with the choice of chemical potential and energy reference as 
\begin{align} \label{eq:choice-kappa-E-homogeneous}
\boxed{\nu(\lambda) =  -\kappa+\lambda \widehat w(0) N_0 (\lambda) -\lambda\frac{w(0)}{2}, \qquad E_0(\lambda):= \lambda \frac{\widehat w(0)}{2} N_0(\lambda)^2 . }
\end{align}
This choice allows us to express the physical Hamiltonian in the form
\begin{align}
\bH_\lambda &= \dGamma (h) + \lambda\mathbb{W} - \nu(\lambda) \cN + E_0(\lambda) \nonumber\\
&= \dGamma(h+\kappa) + \underbrace{\frac\lambda2 \sum_{k\in (2\pi \Z)^2}  \hat{w} (k) \left|\dG (e ^{ik \cdot x}) - \left\langle \dG (e ^{ik \cdot x}) \right\rangle_{\Gamma_{0}} \right| ^2}_{=:\lambda\bWren}.\label{eq:def_H_lambda_homogeneous}
\end{align}
This follows from the fact that, by translation invariance,
$$ 
\left\langle \dG (e ^{ik \cdot x}) \right\rangle_{\Gamma_{0}} = \delta_{k=0} N_0 (\lambda),
$$
where $\langle \: \cdot \: \rangle_{\Gamma_{0}}$ denotes expectation against the reference Gaussian state $\Gamma_{0}$. 

We will require that the interaction potential is a bit more regular than stated in~\eqref{eq:asum w 1}. To be precise, we assume that
\begin{equation}\label{eq:interaction-homo}
 \widehat w(k)\ge 0, \quad \sum_{k\in (2\pi \mathbb{Z})^d}  \widehat w(k)  \left( 1 + |k|^2  \right)  <\infty 
\end{equation} 
with the Fourier expansion
$$w(x)= \sum_{k\in (2\pi \mathbb{Z})^d} \widehat w(k) e^{i k\cdot x}.$$
Our first result is

\begin{theorem}[\textbf{Homogeneous gas}]\label{thm:main-1}\mbox{}\\
Let $d=2, 3$. Let $h=-\Delta$ on the torus $\bT^d$. For a fixed $\kappa>0$, let $h_0=-\Delta+\kappa$ and call $\Gamma_{0}=\cZ_0(\lambda)^{-1}e^{-\lambda\,\dG(-\Delta+\kappa)}$ the corresponding Gaussian state as in~\eqref{eq:free_Torus}. 
Let $w:\bT^d\to \R$ be an even function satisfying~\eqref{eq:interaction-homo} and call $\Gamma_\lambda=\cZ(\lambda)^{-1}e^{-\lambda\bH_\lambda}$ the interacting quantum Gibbs state as in~\eqref{eq:interacting_Torus}. Let $\mu_0$ be the Gaussian measure with covariance $h_0^{-1}$ and let $d\mu=z^{-1}e^{-\cD[u]}d\mu_0$ be the associated nonlinear Gibbs measure as in~\eqref{eq:NL measure}. 
Then we have:

\smallskip

\noindent(\textbf{1}) \underline{Convergence of the relative free-energy:}
\begin{equation}\label{eq:CV free ener}
\lim_{\lambda\to0^+}\log\frac{\cZ(\lambda)}{\cZ_0(\lambda)} = \log z =   \log \left(\int e^{-\cD[u]} d \mu_0(u)\right).
\end{equation}

\medskip

\noindent(\textbf{2}) \underline{Hilbert-Schmidt convergence of all density matrices:} for every $k\ge 1$, 
\begin{equation}\label{eq:CV DM HS}
\lim_{\lambda\to0^+}\tr\left|k!\,\lambda^k\, \Gamma_\lambda^{(k)} -  \int |u^{\otimes k} \rangle \langle u^{\otimes k} | d \mu(u)\right|^2 =0.
\end{equation}

\medskip

\noindent(\textbf{3}) \underline{Trace class convergence of the relative one-body density matrix:} 
\begin{equation}\label{eq:CV DM S1}
\lim_{\lambda\to0^+}\Tr \left| \lambda \left( \Gamma_\lambda^{(1)} - \Gamma_0^{(1)} \right) - \int |u \rangle \langle u | \left( d \mu(u) - d \mu_0(u) \right)  \right| =0.
\end{equation}
\end{theorem}

Here are some immediate comments on the homogeneous gas 

\medskip

\noindent\textbf{1.} In space dimensions $d=2,3$, $\mu_0$ concentrates on negative Sobolev spaces $\bigcap_{t<1-d/2} \gH^t$. This leads to a big jump in difficulty in comparison to our previous treatment in 1D \cite{LewNamRou-15,LewNamRou-18a}. 

\smallskip 

\noindent\textbf{2.} By expanding~\eqref{eq:nu_lambda} up to the order $\lambda^{-1}$, we have 
\begin{equation}
 \nu(\lambda)=\begin{cases}\displaystyle
-\frac{\widehat{w}(0)}{4\pi}\log(\lambda)-\nu_0(\kappa)+o(1)_{\lambda\to0^+}&\text{for $d=2$,}\\[0.2cm]
\displaystyle-\frac{\zeta(3/2)\widehat{w}(0)}{8\pi^{\frac32}\sqrt\lambda}-\nu_0(\kappa)+o(1)_{\lambda\to0^+}&\text{for $d=3$.}
  \end{cases}
 \label{eq:formula_nu_lambda}
\end{equation}
The first term is a counter-term compensating the divergence of the interactions. The second part $\nu_0(\kappa)$ is a complicated function of the chemical potential $\kappa$ of the final reference Gaussian measure. We compute it later in Lemma~\ref{lem:equivalent} in Appendix~\ref{app:interpretation} and obtain 
\begin{equation}
\nu_0(\kappa)=\begin{cases}
\displaystyle\kappa+\widehat{w}(0)\frac{\log(\kappa)}{4\pi}-\widehat{w}(0)\,\phi_2(\kappa)&\text{for $d=2$,}\\[0.2cm]
\displaystyle\kappa+\widehat{w}(0)\frac{\sqrt\kappa}{4\pi}-\widehat{w}(0)\,\phi_3(\kappa)&\text{for $d=3$,}
\end{cases}
\label{eq:ren chem pot}
\end{equation}
with the positive decreasing function
\begin{align}
\phi_d(\kappa)&=\sum_{k\in2\pi\Z^d}\left(\frac1{|k|^2+\kappa}-\frac{1}{(2\pi)^d}\int_{(-\pi,\pi)^d}\frac{dp}{|k+p|^2+\kappa}\right) \nn \\
&=\begin{cases}
\dps \frac{1}{4\pi}\int_0^\ii \frac{e^{-\kappa t}}{t}\sum_{\ell\in\Z^2\setminus\{0\}} e^{-\frac{|\ell|^2}{4t}}\,dt&\text{for $d=2$,}\\[0.5cm]
\dps \frac{1}{4\pi}\sum_{\ell\in\Z^3\setminus\{0\}}\frac{e^{-\sqrt{\kappa}|\ell|}}{|\ell|}&\text{for $d=3$.}
  \end{cases}
  \label{eq:phi_d_core}
\end{align}
Here we started for simplicity with $\kappa$ and we found the corresponding term $\nu_0(\kappa)$. It is in fact desirable to deduce $\kappa$ from $\nu_0$ rather than the other way around, which is possible for all $\nu_0>0$ since the function $\nu_0(\kappa)$ is increasing. The $o(1)$ term in~\eqref{eq:formula_nu_lambda} (including the term $\lambda w(0)/2$ in~\eqref{eq:choice-kappa-E-homogeneous}) plays no role and can be removed in the definition of $\nu(\lambda)$ without changing our result. The constant $E_0(\lambda)$ has no effect on the Gibbs state itself. 

\smallskip

\noindent\textbf{3.} In Appendix~\ref{app:interpretation} we re-express the theorem in microscopic variables and explain the link with the phase transition in the Bose gas. Theorem~\ref{thm:main-1} describes the system just before the phase transition (i.e. with a density just below the critical one) and $\mu$ gives the way that the condensate appears at the macroscopic scale.

\smallskip

\noindent\textbf{4.} Our trace-class convergence~\eqref{eq:CV DM S1} implies the convergence of the relative number of particles
$$
\lim_{\lambda\to0^+}\lambda\Big(\langle \cN \rangle_{\Gamma_{\lambda}} -\langle \cN \rangle _{\Gamma_{0}}\Big) =\int \cM(u) \,d\mu(u),
$$
with $\cM(u)$ the renormalized mass defined in Lemma~\ref{lem:re-mass}. This is rather non-trivial, for the two terms on the left-hand side diverge when taken separately. This convergence is in the spirit of the physics literature on Bose gases~\cite{ArnMoo-01,BayBlaiHolLalVau-99,BayBlaiHolLalVau-01,HolBay-03,KasProSvi-01}, where it is argued that the critical densities are related as $ \rho_\lambda ^{\mathrm{quant}} - \rho_0 ^{\mathrm{quant}} \sim \rho_\lambda ^{\mathrm{clas}} - \rho_0 ^{\mathrm{clas}}$. In particular, only the \emph{difference} in critical densities (interacting minus non interacting) can be properly described using classical field theory. Note that we have the formal relation
$$\cM(u)=\text{``}\;\int_{\R^d}|u|^2-\pscal{\int_{\R^d}|u|^2}_{\mu_0}\;\text{''},$$ 
that is, $\cM(u)$ is actually a difference of two quantities which are individually infinite $\mu_0$-almost surely. Estimates on relative one-particle density matrices related to~\eqref{eq:CV DM S1} are recently obtained in~\cite{DeuSeiYng-19,DeuSei-19_ppt}, but in a different setting, without divergences.

\smallskip

\noindent\textbf{5.} Our proof also shows that, in 2D and 3D, for any $k\geq 2$, the difference $\lambda^{k}(\Gamma_\lambda^{(k)} -  \Gamma_0^{(k)})$ is {\em not} bounded in trace class, see Remark \ref{rm:G2}. When $k\geq 2$ one needs to remove from $\Gamma_\lambda^{(k)}$ combinations of $\Gamma_0^{(\ell)}$ for $\ell \leq k$ to obtain an operator converging in trace-class. We do not consider this explicitly here, but results in this direction are in~\cite[Theorem~1.7]{FroKnoSchSoh-20}. There the appropriately renormalized (Wick-ordered) density matrices are considered, and their integral kernels shown to converge uniformly as continuous functions, which implies trace-class convergence as operators.

\subsection{Inhomogeneous gas}\label{sec:inhom dir}
Here we focus on the case 
$$\boxed{h = -\Delta+V(x) \, \text{ on } \,\Omega=\R^d,\quad d=2, 3.}$$ 
We are typically thinking of $V(x)=|x|^s$, but we can work with a larger class of potentials:
\begin{equation}
 \label{eq:cond_V_FKSS 11}
\begin{cases}
0\le V\in L^\ii_{\rm loc}(\R^d), \quad \dps \lim_{|x|\to\ii}V(x)=+\ii,\\[0.2cm]
V(x+y)\leq C(V(x)+1)(V(y)+1),\\[0.2cm]
|\nabla V(x)|\leq C(V(x)+1),\\[0.2cm]
\dps \iint_{\R^d\times\R^d}\frac{dx\,dk}{(|k|^2+V(x)+1)^{2}}<\ii.
  \end{cases}
 \end{equation}
All the conditions above~\eqref{eq:cond_V_FKSS 11} are satisfied for 
\begin{equation}\label{eq:power pot}
V(x) = |x|^s \mbox{ with } s>2d/(4-d). 
\end{equation}
Thanks to the Lieb-Thirring inequality in~\cite[Theorem~1]{DolFelLosPat-06}, 
$$
\Tr[(-\Delta+V+1)^{-p}]\leq \frac{1}{(2\pi)^d}\iint_{\R^d\times\R^d}\frac{dx\,dk}{(|k|^2+V(x)+1)^p}, 
$$
the last assumption in \eqref{eq:cond_V_FKSS 11} ensures that 
$$
\tr[h^{-2}]<\ii.
$$
This is the optimal requirement in our method and it barely fails on a bounded domain ($s=+\infty$ in~\eqref{eq:power pot} formally) in dimension $d=4$. For the interaction potential, we assume that $w:\R^d\to \R$ is an even function satisfying 
\begin{equation}
 \label{eq:w new asum 11}
w \in L^1\left(\R^d, (1+V(x))^2 dx\right), \quad 0\le \widehat w \in L^1(\R^d, (1+|k|^2) dk). 
 \end{equation}

First we discuss the following issue: If we take as reference state the Gaussian state based on $h=-\Delta+V$ and start with a renormalized Hamiltonian in the same form as~\eqref{eq:def_H_lambda_homogeneous}, we actually perturb the original physical Hamiltonian by an external potential which is $x$-dependent. This is physically questionable since it does not correspond to adjusting the two constants $\nu$ and $E_0$. Thus Gaussian states based on $h$ are not a physically good reference to study the limit of $\bH_\lambda$. The determination of the right reference state is discussed next.  

Let ${V_\lambda}$ be a general one-body potential (which will be specified later and can depend on~$\lambda$). Let ${\Gamma}_0$ be the Gaussian state associated with  $-\Delta+{V_\lambda}(x)$, namely
\begin{align} \label{eq:wide-Gamma-0-first}
\Gamma_0 :=\frac{e^{-\lambda\,\dGamma(-\Delta+{V_\lambda})}}{\cZ_0(\lambda)} ,\qquad  \cZ_0(\lambda)=\tr\left(e^{-\lambda\,\dGamma(-\Delta+{V_\lambda})}\right)
\end{align}
and let $\rhoO^{{V_\lambda}}(x)$ be its one-body density defined by
\begin{align}\label{eq:1pdm-free Gibbs wide}
\rhoO^{V_\lambda}  (x):=\Gamma_0 ^{(1)} (x;x)= \left[ \frac{1}{e^{\lambda(-\Delta+V_\lambda)}-1} \right](x;x),
\end{align}
where $\Gamma_0 ^{(1)} (x;y)$ is the integral kernel of the one-body density matrix $\Gamma_0^{(1)}$ (the diagonal part $\Gamma_0 ^{(1)} (x;x)$ can be defined properly for instance by the spectral decomposition). Note that in general $\rhoO^{V_\lambda} (x)$ depends on $x$. 

Following the discussion in Section~\ref{sec:defs formal},  we consider the renormalized Hamiltonian as in~\eqref{eq:discu renorm quant}, but with the reference state $ \Gamma_0$. This results in
\begin{multline}\label{eq:renorm int wide -0}
\dGamma(-\Delta+{V_\lambda})+ \frac{\lambda}{2} \int_{\R^d} \hat{w} (k) \left|\dG (e ^{ik \cdot x}) - \left\langle \dG (e ^{ik \cdot x}) \right\rangle_{\Gamma_0} \right| ^2 dk  \\
= \dGamma \left(-\Delta+V_\lambda - \lambda w\ast   \rhoO^{ V_\lambda} + \lambda w(0)/2\right) + \lambda \bW + E_0(\lambda)
\end{multline}
where $E_0(\lambda)$ is given by
\begin{align} \label{eq:choice-kappa-E-trap-VT}
\quad E_0(\lambda):= \frac{\lambda}{2} \iint_{\R^d\times\R^d}  \rhoO^{V_\lambda}(x) w(x-y)  \rhoO^{V_\lambda} (y) dxdy.
\end{align}
This Hamiltonian coincides with the physical Hamiltonian in~\eqref{eq:many body hamil} with chemical potential $\nu(\lambda)$ if and only if $V_\lambda$ solves the  nonlinear equation
\begin{align} \label{eq:counter-term}
\boxed{V_\lambda - \lambda w\ast   \rhoO^{ V_\lambda} + \lambda w(0)/2= V -\nu(\lambda)}
\end{align}
which was called the {\em counter-term problem} in~\cite{FroKnoSchSoh-17}.
Equation~\eqref{eq:counter-term} is in the same spirit as~\eqref{eq:ren chem pot} seen before for $\kappa$, but the unknown is now a function. 
It  arises naturally when restricting the problem to the subclass of Gaussian quantum states, as we explain now. That the minimizing Gaussian state gives an appropriate reference in renormalization procedures has been used before in several contexts, for instance in quantum electrodynamics~\cite{LieSie-00,HaiLewSol-07}. 

We recall that to any one-body density matrix $\gamma\geq0$ one can associate a unique Gaussian state $\Gamma$ on the Fock space which has the one-particle density matrix $\Gamma^{(1)}=\gamma$~\cite{BacLieSol-94,Solovej-notes}. Its energy terms and entropy can be expressed in terms of $\gamma$ as
\begin{align}\label{eq:quasi free Hartree}
-\tr\left[ \Gamma\log \Gamma \right] &=  \tr\left[(1+\gamma)\log(1+\gamma)-\gamma\log\gamma\right], \nonumber\\
\tr\left[ \dGamma \left( -\Delta+V-\nu\right) \Gamma\right]&= \Tr\left[(-\Delta+V-\nu)\gamma\right],  \nonumber\\
\tr \left[\mathbb{W} \Gamma\right] &= \frac{1}{2} \iint \gamma(x;x) w(x-y) \gamma(y;y) \,dx\, dy \nonumber \\&\qquad+ \frac{1}{2} \iint w(x-y) |\gamma(x;y)|^2\, dx\, dy.
\end{align}
The last term in \eqref{eq:quasi free Hartree} (called the exchange energy) is typically negligible at leading order, resulting in the \emph{mean-field} or  \emph{reduced Hartree free energy}
\begin{multline}
\cF^{\rm MF}[\gamma]:=\Tr\left[(-\Delta+V-\nu)\gamma\right]+\frac{\lambda}{2} \iint \gamma(x;x) w(x-y) \gamma(y;y)\, dx\, dy\\
 -T\tr\left[(1+\gamma)\log(1+\gamma)-\gamma\log\gamma\right].
 \label{eq:red-Hartree_functional}
\end{multline}
Minimizing this energy corresponds to solving the mean-field problem mentioned in the introduction, as stated in the following result, which is a simple consequence of the convexity of the functional $\cF^{\rm MF}$. The proof is given in Appendix~\ref{app:QF}.

\begin{lemma}[\textbf{Existence and uniqueness of the reference quasi-free state}]\label{lem:min quasi free}\mbox{}\\
Let $d\geq1$, $T>0$, $\lambda\geq0$ and $\nu\in\R$. Assume that
$$0\le V\in L^1_{\rm loc}(\R^d), \quad e^{-V/T}\in L^1(\R^d), \quad  w\in L^1(\R^d), \quad 0\le \hat w\in L^1(\R^d).$$
Then the variational problem
\begin{equation}
F^{\rm MF}(T,\lambda,\nu)=\inf_{\gamma=\gamma^*\geq0}\cF^{\rm MF}[\gamma]
\label{eq:red-Hartree-intro}
\end{equation}
admits a unique minimizer $\gamma^{\rm MF}$. This minimizer solves the nonlinear equation 
\begin{equation}\label{eq:gammarH}
\gamma^{\rm MF}=\frac1{e^{T^{-1}(-\Delta+V-\nu+\lambda\rho_{\gamma^{\rm MF}}\ast w)}-1}.
\end{equation}
Hence, when $T=1/\lambda$ and $\nu=\nu(\lambda)+\lambda w(0)/2$, its potential $V_\lambda:=V-\nu(\lambda)-\lambda w(0)/2+\lambda\rho_{\gamma^{\rm MF}}\ast w$ solves the nonlinear equation~\eqref{eq:counter-term}. 
\end{lemma}

The lemma says that we have to use as reference the Gaussian state with minimal free energy (without exchange) and chemical potential $\nu(\lambda)+\lambda w(0)/2$, since this allows to rewrite the full Hamiltonian $\bH_\lambda$ in the desired form~\eqref{eq:discu renorm quant}. The shift $\lambda w(0)/2$ of the chemical potential is the one appearing in~\eqref{eq:intro Fourier quant int} and it is negligible in our regime.

The nonlinear equation~\eqref{eq:counter-term} has been studied by Fr\"ohlich-Knowles-Schlein-Sohinger in ~\cite{FroKnoSchSoh-17}. It is proved herein that, when $T=1/\lambda$ and  the chemical potential $\nu(\lambda)$ is tuned as in the homogeneous case~\eqref{eq:choice-kappa-E-homogeneous}, the potential $V_\lambda$ converges to a limit $V_0$ that we will use as the reference renormalized potential. The following statement summarizes the results of~\cite[Section~5]{FroKnoSchSoh-17}. Everything is again expressed for simplicity in terms of the parameter $\kappa$ instead of $\nu_0$. 

\begin{theorem}[\textbf{Limit renormalized potential}{~\cite{FroKnoSchSoh-17}}]\label{thm:counter}\mbox{}\\
Let $d=2,3$. Let $V,w$ satisfy \eqref{eq:cond_V_FKSS 11}-\eqref{eq:w new asum 11}. Take a constant $\kappa >0$ and set 
\begin{equation}\label{eq:rhoOkappa 11}
\rhoO^\kappa(\lambda) :=  \frac{1}{(2\pi)^d\lambda^{\frac{d}2}} \int_{\R ^d} \frac{dk}{e^{|k|^2+\lambda\kappa} -1}  , \qquad \nu(\lambda) := \lambda \hat{w} (0) \rhoO^\kappa(\lambda) -\kappa-\lambda w(0)/2.
\end{equation}
Then, there exists $\kappa_0<\ii$ such that we have the following statements for all $\kappa\geq\kappa_0$ and all $0<\lambda\leq1$. 

\smallskip

\noindent(\textbf{1}) The unique solution $V_\lambda$ of \eqref{eq:counter-term} satisfies  
\begin{equation}
\frac{V}{2}\le V_\lambda -\kappa \le \frac{3V}2.
\label{eq:bound_V_T 11}
\end{equation}

\smallskip

\noindent(\textbf{2}) There exists a function $V_0$ satisfying
$$ \lim_{\lambda\to0^+}\norm{\frac{V_\lambda-V_{0}}{V}}_{L^\ii(\R^d)}=0$$
and
\begin{align} \label{eq:VT-Vinf-Sp 11}
\lim_{T\to\ii}\Tr \Big| (-\Delta+V_\lambda)^{-1} -(-\Delta+V_0)^{-1}  \Big|^2=0.
\end{align} 

\smallskip

\noindent(\textbf{3})  The limiting potential $V_0$ solves the nonlinear equation
\begin{equation}
\begin{cases}
 V_0  = V +w\ast   \rho_{0}+\kappa,\\[0.2cm]
 \dps\rho_0(x)=\left(\frac{1}{-\Delta+V_0}-\frac{1}{-\Delta+\kappa}\right)(x;x).
\end{cases}
 \label{eq:equation_V_infty}
\end{equation}
\end{theorem}

The above is not stated exactly as in~\cite[Section~5]{FroKnoSchSoh-17}, where the limiting nonlinear equation~\eqref{eq:equation_V_infty} for $V_0$ was indeed not mentioned. We quickly discuss the link with~\cite{FroKnoSchSoh-17} and the proof of~\eqref{eq:equation_V_infty} in Appendix~\ref{app:counter}. Note that the limiting equation~\eqref{eq:equation_V_infty} is formally obtained by replacing the Bose-Einstein entropy $\tr(-\gamma\log\gamma+(1+\gamma)\log(1+\gamma))$  by $\tr(\log\gamma)$ (which is its leading behavior at large $\gamma$) in the variational principle~\eqref{eq:red-Hartree_functional} and then writing the associated variational equation.

The interpretation of Theorem~\ref{thm:counter} is that the divergence of the mean-field density $\rho_\lambda(x)$ is essentially $x$-independent, and given by that of $\rhoO^\kappa(\lambda)$, under the condition that $\nu(\lambda)$ is chosen as in~\eqref{eq:rhoOkappa 11}. The function $\rhoO^\kappa(\lambda)$ has already appeared in~\eqref{eq:nu_lambda} and we recall that 
$$\rhoO^\kappa(\lambda)=\begin{cases}
-\frac{\log(\lambda\kappa)}{4\pi\lambda}+O(1)_{\lambda\to0^+}&\text{for $d=2$,}\\[0.2cm]
-\frac{\zeta(3/2)}{8\pi^{\frac32}\lambda^{\frac32}}-\frac{\sqrt{\kappa}}{4\pi \lambda}+O(1)_{\lambda\to0^+}&\text{for $d=3$,}
  \end{cases}
$$
so that $\nu(\lambda)$ behaves the same as in~\eqref{eq:nu_lambda}. Here $\kappa$ is interpreted as a kind of effective chemical potential in the limit, but note that $V_0$ depends in a nonlinear way on~$\kappa$. The link with the free Bose gas is detailed in Appendix~\ref{app:interpretation}.

In the following we use as reference state the Gaussian quantum state 
\begin{equation}
 \Gamma_0=\cZ_0(\lambda)^{-1}e^{-\lambda\dG(-\Delta+V_\lambda)},\qquad \cZ_0(\lambda)=\tr_\gF\left[e^{-\lambda\dG(-\Delta+V_\lambda)}\right]
 \label{eq:free_inhomogeneous}
\end{equation}
associated with the one-particle operator $-\Delta+V_\lambda$, which provides in the limit the reference Gaussian measure $\mu_0$ with covariance $h_0^{-1}=(-\Delta+V_0)^{-1}$, depending on $\kappa$. 

We can now state the main result relating the physical inhomogeneous Hamiltonian to nonlinear Gibbs measure. 

\begin{theorem}[\textbf{\textbf{Inhomogeneous gas}}]\label{thm:main-2}\mbox{}\\
Let $d=2,3$. Let $V,w$ satisfy \eqref{eq:cond_V_FKSS 11}-\eqref{eq:w new asum 11}. Let $\kappa\geq \kappa_0$ and  $\nu(\lambda),V_\lambda,V_0$ as in Theorem~\ref{thm:counter}. Consider the Gibbs state $\Gamma_\lambda=\cZ(\lambda)^{-1}e^{-\lambda \bH_{\lambda}}$ associated with the physical Hamiltonian $\bH_\lambda$ in \eqref{eq:many body hamil} with  $h=-\Delta+V$ and $E_0(\lambda)$ as in \eqref{eq:choice-kappa-E-trap-VT}. Let $\Gamma_0=\cZ_0(\lambda)^{-1}e^{-\lambda\dG(-\Delta+V_\lambda)}$ be the reference Gaussian quantum state as in~\eqref{eq:free_inhomogeneous}. Let $\mu_0$ be the Gaussian measure with covariance $(-\Delta+V_0)^{-1}$ and let $d\mu=z^{-1}e^{-\cD[u]}d\mu_0$ be the associated nonlinear Gibbs measure as in~\eqref{eq:NL measure}. Then we have:

\smallskip

\noindent(\textbf{1}) \underline{Convergence of the relative free-energy:}
\begin{equation}\label{eq:CV free ener trap II}
\lim_{\lambda\to0^+}\log\frac{\cZ(\lambda)}{\cZ_0(\lambda)} =\log z = \log\left(\int e^{-\cD[u]} d \mu_0(u)\right).
\end{equation}

\smallskip

\noindent(\textbf{2}) \underline{Hilbert-Schmidt convergence of all density matrices:} for every $k\ge 1$, 
\begin{equation}\label{eq:CV DM S2 trap-II}
\lim_{\lambda\to0^+}\tr\left|\lambda^k\, k!\,  \Gamma_\lambda^{(k)} -  \int |u^{\otimes k} \rangle \langle u^{\otimes k} | d \mu(u)\right|^2 =0.
\end{equation}

\smallskip

\noindent(\textbf{3}) \underline{Trace class convergence of the relative one-body density matrix:} 
\begin{equation}\label{eq:CV DM S1 trap-II}
\lim_{\lambda\to0^+}\Tr \left| \lambda \left( \Gamma_\lambda^{(1)} - \Gamma_0^{(1)} \right) - \int |u \rangle \langle u | \left( d \mu(u) - d \mu_0(u) \right)  \right| =0.
\end{equation}
\end{theorem}

\bigskip

Here are some remarks on the inhomogeneous case.

\smallskip

\noindent\textbf{1.} Note that our reference $\Gamma_0$ is not the exact mean-field minimizer since its chemical potential is $\nu(\lambda)+\lambda w(0)/2$ instead of $\nu(\lambda)$, by Theorem~\ref{lem:min quasi free}. This shift is the same as the one in~\eqref{eq:intro Fourier quant int} and it is completely negligible in our regime. It is simpler to work with $\Gamma_0$ as a reference but the result is exactly the same if we use instead the exact mean-field minimizer (as we did in the introduction). Note that when referring to a ``mean-field minimizer'' we always understand that the exchange term is neglected in its  definition. 

\smallskip

\noindent\textbf{2.} The same theorem was shown in dimension $d=1$ for $V(x)\geq C|x|^s$ with $s>2$ in~\cite{LewNamRou-15}  and $s>1$ in~\cite{LewNamRou-18a}. In this case it is not necessary to use the mean-field solution as reference. The final measure is indeed absolutely continuous with respect to the non-interacting measure with covariance $h^{-1}=(-\Delta+V+\kappa)^{-1}$. The assumption that $\kappa$ is large enough is also not necessary. The proof given in this paper applies to dimension $d=1$ as well, with the weaker assumption $s>2/3$. 

\smallskip

\noindent\textbf{3.} What is really needed in our approach is that $\tr[h^{-2}]<\ii$. In Section~\ref{sec:strategy} we state another theorem which covers any $h=-\Delta+V$ on an arbitrary domain $\Omega\subset\R^d$, for instance with Dirichlet boundary condition. Should the mean-field convergence in Theorem~\ref{thm:counter} hold in this setting, as we believe, we then immediately obtain a result similar to Theorem~\ref{thm:main-2} on the domain $\Omega\subset\R^d$.

\section{Proof strategy} \label{sec:strategy}

To make the presentation transparent, we first formulate in Theorem \ref{thm:main-3} below a general \emph{inverse} problem, from which the previous (direct) statements will easily follow. Then we explain the main ideas of the proof of the inverse statement, whose details occupy the core of the paper. 

\subsection{General ``inverse'' statement}\label{sec:inverse}

By ``inverse problem'' we mean the limit of the quantum model (with an arbitrary one-particle Hamiltonian $h$), to which we add properly chosen, $x$-dependent, counter terms so that the limit measure is absolutely continuous with respect to the non-interacting gaussian (instead of the mean-field one as in the previous section). 

Assume that $\Omega$ is an arbitrary smooth domain in $\R^d$ and $h$ is a positive self-adjoint operator on $L^2(\Omega)$ such that $\tr e^{-\beta h}<\ii$ for every $\beta>0$. The reader might think of the typical case $h=-\Delta+V$ on $L^2(\R^d)$ with a trapping potential $V$ diverging fast enough at infinity, as before. However, in order to cover as many practical situations as possible, we will keep $h$ rather general and abstract from now on. This is a difference in approach with respect to~\cite{FroKnoSchSoh-20}, where a path integral formalism is used, based on the Feynman-Kac formula for the Laplacian.

Let 
\begin{equation}
 \Gamma_0=\cZ_0(\lambda)^{-1}\,e^{-\lambda\dG(h)},\qquad \cZ_0(\lambda)=\tr_\gF[e^{-\lambda\dG(h)}]
 \label{eq:quantum_Gaussian_inverse}
\end{equation}
be the corresponding quantum Gaussian state. Its one-body density $\rhoO (x)$ is given by 
\begin{align}\label{eq:1pdm-free Gibbs}
\rhoO  (x):= \Gamma_0 ^{(1)} (x;x)= \left[ \frac{1}{e^{\lambda h}-1} \right](x;x).
\end{align}
We then consider the renormalized interaction 
\begin{align} \label{eq:renorm int trap}
\bWren &=  \frac 12 \int_{\R^d} \hat{w} (k) \left|\dG (e ^{ik \cdot x}) - \left\langle \dG (e ^{ik \cdot x}) \right\rangle_{\Gamma_0} \right| ^2 dk \\
&= \bW -  \dG (w \ast  \rhoO) +\frac{1}{2} \int_{\R^d}  \rhoO(x) w(x-y)  \rhoO(y) dxdy + \frac{\cN}{2} w (0). \nn
\end{align}
and the renormalized Hamiltonian as  
\begin{align} \label{eq:dGh+W-I}
\bH_\lambda := \dGamma(h)+ \lambda \bWren = \dGamma(h - \lambda w  \ast \rhoO + \lambda w(0)/2) + \lambda \bW  + E_0(\lambda)
\end{align}
where
\begin{align} \label{eq:choice-kappa-E-trap-I}
\quad E_0(\lambda):= \frac{\lambda}{2} \iint_{\R^d\times \R^d}  \rhoO(x) w(x-y)  \rhoO(y) dxdy.
\end{align}
Thus instead of varying the chemical potential (we set $\nu = 0$ here), we have replaced the bare one-body operator $h$ by the dressed operator $h - \lambda w  \ast \rhoO- \lambda w(0)/2$. 

We will, for the sake of generality, only make the 

\begin{asumption}[\textbf{The one-body Hamiltonian}]\label{asum:h}\mbox{}\\
Let $h$ be a self-adjoint operator on $L^2 (\Omega)$ satisfying
\begin{align}
\tr[h^{-2}]&<\infty, \label{eq:Schatten h-strong-1}\\
\|[h,e^{ik\cdot x}]h^{-1/2}\|&\le C(1+|k|^2),  \quad \forall k\in \R^d, \label{eq:Schatten h-strong-2}\\
e^{-\beta h}(x,y)  &\ge 0, \quad \forall  \beta >0 \label{eq:Schatten h-strong-3}
\end{align}
where $e^{ik\cdot x}$ is identified with the corresponding multiplication operator.
\end{asumption}

This is satisfied in the case of Theorem~\ref{thm:main-2}, as we explain after the statement of the next theorem.

Our assumptions on $w$ are the same as before:
\begin{equation}
 w(x)=\int_{\Omega^*} \widehat w(k) e^{ik \cdot x} dk, \quad \widehat w(k)\ge 0, \quad \int_{\Omega^*} \widehat w(k)  \left( 1 + |k|^2  \right)  dk <\infty.
 \label{eq:interaction}
\end{equation}
Of course when $\Omega=\bT^d$, the integration is interpreted as a sum as in \eqref{eq:interaction-homo} and the momenta $k$ in~\eqref{eq:Schatten h-strong-2} must be in $2\pi\Z^d$. 
We will prove the following: 

\begin{theorem}[\textbf{General inverse statement}]\label{thm:main-3}\mbox{}\\
Let $h>0$ on $L^2(\Omega)$ satisfy Assumption~\ref{asum:h} and let $w:\R^d \to \R$ satisfy \eqref{eq:interaction}. Consider the Gibbs state 
$\Gamma_\lambda=\cZ(\lambda)^{-1} e^{-\lambda\bH_\lambda}$ associated with the renormalized  Hamiltonian $\bH_\lambda$ in ~\eqref{eq:dGh+W-I}. Let $\Gamma_0=\cZ_0(\lambda)^{-1}\,e^{-\lambda\dG(h)}$ be the reference Gaussian quantum state as in~\eqref{eq:quantum_Gaussian_inverse}. Let $\mu_0$ be the Gaussian measure with covariance $h^{-1}$ and let $d\mu=z^{-1}e^{-\cD[u]}d\mu_0$ be the associated nonlinear Gibbs measure as in~\eqref{eq:NL measure}. Then we have: 

\smallskip

\noindent(\textbf{1}) \underline{Convergence of the relative free-energy:}
\begin{equation}\label{eq:CV free ener trap}
\lim_{\lambda\to0^+}\log\frac{\cZ(\lambda)}{\cZ_0(\lambda)}= \log z = \log \left(\int e^{-\cD[u]} d \mu_0(u)\right).
\end{equation}

\smallskip

\noindent(\textbf{2}) \underline{Hilbert-Schmidt convergence of all density matrices:} for every $k\ge 1$, 
\begin{equation}\label{eq:CV DM Sp trap-HS}
\lim_{\lambda\to0^+}\tr\left|\lambda^k\, k!\,  \Gamma_\lambda^{(k)} -  \int |u^{\otimes k} \rangle \langle u^{\otimes k} | d \mu(u)\right|^2 =0.
\end{equation}

\smallskip

\noindent(\textbf{3}) \underline{Trace class convergence of the relative one-body density matrix:} 
\begin{equation}\label{eq:CV DM S1 trap}
\lim_{\lambda\to0^+}\Tr \left| \lambda \left( \Gamma_\lambda^{(1)} -  \Gamma_0^{(1)} \right) - \int |u \rangle \langle u | \left( d \mu(u) - d \mu_0(u) \right)  \right| =0.
\end{equation}
\end{theorem}

Here are our comments on Theorem \ref{thm:main-3}. 

\medskip

\noindent\textbf{1.} Let us go back to the case $h=-\Delta+V(x)$. When $V$ is not a constant, the renormalized Hamiltonian~\eqref{eq:dGh+W-I} is different from the physical Hamiltonian~\eqref{eq:many body hamil} because we have replaced the bare potential $V$ by the new potential  $V - \lambda w  \ast \rhoO + \lambda w(0)/2$, instead of simply shifting a chemical potential. Thus to obtain the Gibbs measure associated with $V$, we have started with an ad-hoc, different external potential. In this regard the above is an inverse statement. 

\medskip

\noindent\textbf{2.} The homogeneous case in Theorem~\ref{thm:main-1} is included in Theorem~\ref{thm:main-3} because we have seen in~\eqref{eq:def_H_lambda_homogeneous} that the Hamiltonian can be rewritten in the form~\eqref{eq:dGh+W-I} if $\nu(\lambda)$ and $E_0(\lambda)$ are chosen appropriately. But the inhomogeneous case in Theorem~\ref{thm:main-2} does not immediately follow from Theorem~\ref{thm:main-3}. In the latter situation, we can write the main Hamiltonian as in~\eqref{eq:dGh+W-I} but only if we use the mean-field potential $V_\lambda$ solving the counter term problem~\eqref{eq:counter-term}, and thus depending on $\lambda$. Recall that $V_\lambda$ converges to $V_0$ by Theorem~\ref{thm:counter}, however. Our proof of Theorem~\ref{thm:main-3} will indeed apply to this case too because all our estimates are quantitative and depend only on $\tr[h^{-2}]$. This stays bounded when $V_\lambda\to V_0$ in the sense of Theorem~\ref{thm:counter}. Further details on how to
  deduce Theorem~\ref{thm:main-2} from Theorem~\ref{thm:main-3} and its proof are given in Section~\ref{sec:conclusion}. 

\smallskip 

\noindent\textbf{3.} Here besides the natural condition \eqref{eq:Schatten h-strong-1}, we also require the bound on the commutator of $h$ with the multiplication operator by  $e^{ik\cdot x}$ in \eqref{eq:Schatten h-strong-2}. 
The precise meaning of~\eqref{eq:Schatten h-strong-2} is that we assume $e^{ik\cdot x}$ stabilizes the domain $D(h)$ of $h$ for all $k$, and that 
$$\norm{[h,e^{ik\cdot x}]u}^2\leq C(1+|k|^2)^2 \pscal{u,hu}.$$
for all $u\in D(h)$. This condition is needed for the key variance estimate in Section~\ref{sec:correl}. For $h=-\Delta+V$ in $\Omega=\R^d$, it follows immediately from the computation 
\begin{equation}\label{eq:commutator h ek}
[e^{ik\cdot x} , h] =[e^{ik\cdot x} , -\Delta] = e^{ik\cdot x}   \Big( -|k|^2 + 2 k \cdot (i\nabla) \Big)
\end{equation}
and the fact that $e^{ik\cdot x}$ stabilizes the domain of the Friedrichs realization of $-\Delta+V$. The assumption is also satisfied for the periodic Laplacian on the torus, or the Dirichlet Laplacian in a bounded domain $\Omega$. 

\smallskip 

\noindent\textbf{4.} Our other assumption~\eqref{eq:Schatten h-strong-3} is well-known to hold for $h=-\Delta +V$ on $L^2(\R^d)$; in this case the positivity of the heat kernel is a consequence of the Feynman-Kac formula (see for example~\cite{Simon-05}). But other models are covered, including fractional Laplacians, and localized versions on bounded domains with various boundary conditions. This condition is only needed for the convergence of density matrices in \eqref{eq:CV DM Sp trap-HS}--\eqref{eq:CV DM S1 trap}. Without \eqref{eq:Schatten h-strong-3}, the free energy convergence \eqref{eq:CV free ener trap} remains valid.  

\medskip

In the following we sketch an outline of the proof of Theorem~\ref{thm:main-3}.

\subsection{Variational method} 
Our method is variational, in the same spirit as our previous works~\cite{LewNamRou-15,LewNamRou-18a}. We shall however rely much more on the structure of the Gibbs state, that is, on the fact that it is the \emph{exact} minimizer (and not just an approximate one) of the free-energy
\begin{align} \label{eq:Gibbs-free-energy} 
\lambda\,F_\lambda=- \log \cZ(\lambda)  = \min_{\substack{\Gamma\ge 0\\ \Tr \Gamma=1}} \left\{ \tr\left[ \lambda\bH_\lambda \Gamma \right] + \tr\left[\Gamma \log \Gamma \right] \right\}.
\end{align}
From~\eqref{eq:Gibbs-free-energy} and a similar formula for the Gaussian quantum state $\Gamma_0$, we deduce that $\Gamma_{\lambda}$ is also the unique minimizer for the {\em relative free energy}:
\begin{equation} \label{eq:rel-energy}
-\log \frac{\cZ(\lambda)}{\cZ_0(\lambda)} = \min_{\substack{\Gamma\ge 0\\ \Tr \Gamma=1}} \Big\{ \cH(\Gamma,\Gamma_{0}) + \lambda \Tr[ (\bH_\lambda-\bH_0) \Gamma]  \Big\}= \min_{\substack{\Gamma\ge 0\\ \Tr \Gamma=1}} \Big\{ \cH(\Gamma,\Gamma_{0}) + \lambda^2 \Tr\left[\bWren \Gamma\right] \Big\}.
\end{equation}
Here  
$$\cH(\Gamma,\Gamma'):= \tr_{\gF}\left(\Gamma(\log\Gamma-\log\Gamma')\right) \ge 0$$
is the von Neumann \emph{relative entropy} of two quantum states $\Gamma$ and $\Gamma'$. The simple rewriting~\eqref{eq:rel-energy} is particularly useful, for the left-hand side is nothing but the free-energy difference, multiplied by $\lambda$. This is the quantity we show converges when $\lambda\to0$ in~\eqref{eq:CV free ener} and~\eqref{eq:CV free ener trap}. Characterizing the difference directly as an infimum is much more convenient than working on both terms seen as infima separately.  

Similarly, the classical Gibbs measure $\mu$ defined in Section~\ref{sec:defs class} is the unique minimizer for the variational problem
\begin{equation} \label{eq:zr-rel}
-\log z = \min_{\substack{\nu \text{~proba. meas.}\\ \nu\ll\mu_0}} \left\{ \cH_{\rm cl}(\nu, \mu_0) + \int \cD[u]\,d\nu(u)\right\}
\end{equation}
where
$$ \cH_{\rm cl} (\nu,\nu'):= \int_{\gH^{s}}\frac{d\nu}{d\nu'}(u)\log\left(\frac{d\nu}{d\nu'}(u)\right)\,d\nu'(u) \ge 0 $$
is the classical relative entropy of two probability measures $\nu$ and $\nu'$. 

The variational problems \eqref{eq:zr-rel}, \eqref{eq:rel-energy} and their basic properties will be 
discussed in Section~\ref{sec:class meas}. For now, observe that~\eqref{eq:zr-rel} begs for being interpreted as a semi-classical version of~\eqref{eq:rel-energy}. This is the route we follow, using semi-classical-type measures associated with general states on the Fock space. To our sequences of quantum states $\Gamma_\lambda$ and $\Gamma_0$ we can associate in the limit $\lambda\to0^+$ two semi-classical measures $\mu$ and $\mu_0$. 
This is a general theory, not particularly linked to the fact that we consider Gibbs states. The measures are also called \emph{de Finetti} or \emph{Wigner} measure and they can be constructed for very general states. The precise sense in which they approximate quantum states is fairly explicit~\cite{AmmNie-08,AmmNie-09,LewNamRou-14,LewNamRou-15b,LewNamRou-15,Rougerie-LMU,Rougerie-cdf}. Everything will be recalled in Section~\ref{sec:deF} below. Although $\mu$ is \emph{a priori} unknown, we have already proved in~\cite{LewNamRou-15} that the de Finetti/Wigner measure $\mu_0$ of the Gaussian state $\Gamma_0$ is simply the Gaussian measure with covariance $h^{-1}$, as it should be. The goal is therefore to prove that $\mu$ is what we want by showing that it solves the minimization problem~\eqref{eq:zr-rel}. This requires to pass to the limit in~\eqref{eq:rel-energy}, the difficult part being the lower bound
\begin{equation}
\liminf_{\lambda\to0^+}\left(-\log \frac{\cZ(\lambda)}{\cZ_0(\lambda)}\right) \geq \cH_{\rm cl}(\mu, \mu_0) +\int \cD[u]\,d\mu(u)\geq -\log z.
\label{eq:ineq_to_be_proven0}
\end{equation}
This can be split into two separate lower bounds
\begin{equation}
\liminf_{\lambda\to0^+}\cH(\Gamma,\Gamma_{0}) \geq \cH_{\rm cl}(\mu, \mu_0) 
\label{eq:ineq_to_be_proven1}
\end{equation}
and
\begin{equation}
\liminf_{\lambda\to0^+}\lambda^2 \Tr\left[\bWren \Gamma\right] \geq \int \cD[u]\,d\mu(u).
\label{eq:ineq_to_be_proven2}
\end{equation}
The first bound~\eqref{eq:ineq_to_be_proven1} has already been shown in~\cite{LewNamRou-15}, based on a Berezin-Lieb-type inequality~\cite{Berezin-72,Lieb-73b,Simon-80,Rougerie-LMU,Rougerie-cdf} for the relative entropy which we recall in Theorem~\ref{thm:rel-entropy} below. The bound~\eqref{eq:ineq_to_be_proven2} is much more difficult due to the renormalization of $\cD$, which is seen in the fact that $\bWren$ is essentially an average of squares of the difference of two divergent quantities, depending on $\lambda$. This makes compactness arguments difficult and all the known estimates in the literature seem insufficient to pass to the weak limit and obtain~\eqref{eq:ineq_to_be_proven2}. In the next section we explain our main new idea for proving~\eqref{eq:ineq_to_be_proven2}.

\subsection{Idea of the proof of~\eqref{eq:ineq_to_be_proven2}} 

As in several of our works on the subject~\cite{LewNamRou-15,LewNamRou-15b,LewNamRou-16c,LewNamRou-18a}, the method is to project the two quantum states $\Gamma_\lambda$ and $\Gamma_0$ to a finite dimensional space using the Fock-space/geometric localization method~\cite{Lewin-11}, where semi-classical analysis is better controlled. The main novelty of this paper is a new way of controlling the projection error, that we briefly describe here. 

We write the renormalized interaction operator in~\eqref{eq:discu renorm quant} as
\begin{align*}
  \lambda^2\bWren &= \frac{\lambda^2}{2} \int \hat{w} (k) \left|\dGamma (e^{ik\cdot x}) - \left\langle \dGamma (e^{ik\cdot x}) \right\rangle_{\Gamma_{0}} \right| ^2 dk \\
  & = \frac{\lambda^2}{2} \int \hat{w} (k) \left|\dGamma (\cos(k\cdot x)) - \left\langle \dGamma (\cos(k\cdot x)) \right\rangle_{\Gamma_{0}} \right| ^2 dk \\
  &\quad + \frac{\lambda^2}{2} \int \hat{w} (k) \left|\dGamma (\sin(k\cdot x)) - \left\langle \dGamma (\sin(k\cdot x)) \right\rangle_{\Gamma_{0}} \right| ^2 dk.
\end{align*}
The main technical step is to replace for each Fourier mode the multiplication operator $e_k\in \{ \cos(k\cdot x), \sin(k\cdot x) \}$ by the projected one $P e_k P$, where $P$ is the projection onto the chosen finite-dimensional space. Indeed, in this space we can rely on quantitative versions of the quantum de Finetti theorem~\cite{ChrKonMitRen-07,LewNamRou-15}. The errors thus made depend heavily on the rank of $P$, whence the need to keep that under control by precisely estimating the contribution of $e_k^+ := e_k - P e_k P$.

A natural choice for $P$ is 
$$ P = \one(h \leq \Lambda_e)$$
for some finite but large energy cut-off $\Lambda_e = \Lambda_e(\lambda)$ to be optimized over. The main challenge is to show that 
\begin{equation}\label{eq:term to bound}
 \lambda^2 \left\langle \left|\dGamma (e_k ^+) - \left\langle \dGamma ( e_k ^+) \right\rangle_{\Gamma_{0}} \right| ^2 \right\rangle_{\Gamma_\lambda} \underset{\lambda\to0^+}{\longrightarrow} 0, \qquad\text{with}\quad  e_k^+:=  e_k - P e_k P.
\end{equation}
This is the crucial place where the renormalization is taken very carefully into account since the counter term $\lambda \left\langle \dGamma ( e_k ^+) \right\rangle_{\Gamma_{0}}$ may diverge fast (like $\log (\lambda^{-1})$ in 2D and $\lambda^{-1/2}$ in 3D in the homogeneous case). The limit~\eqref{eq:term to bound} is the object of Theorem~\ref{thm:correlation} and it is proved using the new correlation inequality~\eqref{eq:2body_by_1body} mentioned in the introduction. The proof of this inequality occupies the whole Section~\ref{sec:variance-by-first-moment}.

The correlation inequality~\eqref{eq:2body_by_1body} allows us to replace the difficult limit of the two-body term in~\eqref{eq:term to bound} by the following one-body problem averaged over a small window $\eps\in[-a,a]$ 
\begin{equation}
\lambda\max_{\eps\in[-a,a]}\left| \left\langle \dGamma (e_k ^+) - \left\langle \dGamma ( e_k ^+) \right\rangle_{\Gamma_{0}} \right\rangle_{\Gamma_{\lambda,\eps}}\right|=\lambda\max_{\eps\in[-a,a]}\left| \tr\left[e_k^+\big(\Gamma_{\lambda,\eps}^{(1)}-\Gamma_{0,0}^{(1)}\big)\right]\right| \to 0
\label{eq:simpler_1_body}
\end{equation}
with the perturbed state
$$\Gamma_{\lambda,\eps}=\cZ(\lambda,\eps)^{-1}\exp\bigg(-\lambda\bH_\lambda+\eps\lambda\left(\dGamma (e_k^+) - \left\langle \dGamma ( e_k ^+) \right\rangle_{\Gamma_{0}}\right)\bigg)$$
where $\cZ(\lambda,\eps)^{-1}$ as usual normalizes the trace of $\Gamma_{\lambda,\eps}$. Note that the effect of the perturbation is just to replace $h$ by $h-\eps e_k^+$ since the other term is a constant and can be removed.
At this stage the new inequality~\eqref{eq:estim_HS_relative_entropy_intro} on one-body density matrices also mentioned in the introduction plays a key role. This inequality is a careful elaboration on a Feynman-Hellmann type argument, see Section~\ref{sec:relative entropy bound}.

\subsection{Organization of the proofs.} The above is our sketch of the proofs' main ideas. Here is how they shall be articulated in the sequel:

\smallskip

\noindent$\bullet$ In Section~\ref{sec:class meas} we collect some basic facts. First we go into more details regarding the construction of the Gibbs measure $\mu$, then we discuss de Finetti measures. Next, we provide some preliminary estimates on the quantum states $\Gamma_0$ and $\Gamma_\lambda$. 

\smallskip

\noindent$\bullet$ The novel part of our paper starts from Section \ref{sec:relative entropy bound}, where we prove the entropy estimate \eqref{eq:estim_HS_relative_entropy_intro} as well as other useful estimates on one-body density matrices.

\smallskip

\noindent$\bullet$ The second new ingredient of our paper is Section \ref{sec:variance-by-first-moment}, where we discuss a general strategy of controlling quantum variance by the first moments of a family of perturbed states. We prove there an inequality more general than~\eqref{eq:2body_by_1body} mentioned in the introduction. 

\smallskip

\noindent$\bullet$ The technical core of the paper is Section~\ref{sec:correl} where we give a quantitative estimate on the two body term on the left side of~\eqref{eq:term to bound}. This is done by carefully carrying out the method in Section \ref{sec:variance-by-first-moment}, using a-priori estimates from Section \ref{sec:relative entropy bound} and Section \ref{sec:class meas}. 

\smallskip

\noindent$\bullet$ All this allows us to prove the desired lower bound~\eqref{eq:ineq_to_be_proven0} in Section \ref{sec:low bound}, using de Finetti measures and controlling the errors as sketched above. 

\smallskip

\noindent$\bullet$ A free-energy upper bound matching~\eqref{eq:ineq_to_be_proven0} is derived in Section~\ref{sec:up bound}, by a trial state argument and some finite dimensional semiclassical analysis. The argument is much easier than for~\eqref{eq:ineq_to_be_proven0}. Our trial state is given by the free quantum state in the ultraviolet (high kinetic energy modes) and by the projected classical interacting state in the infrared (low kinetic energy modes). 

\smallskip

\noindent$\bullet$ In Section~\ref{sec:DM} the convergence of reduced density matrices is deduced from various estimates developed to prove the free-energy convergence, plus Pinsker inequalities. This concludes the proof of Theorems \ref{thm:main-3} and~\ref{thm:main-1}. 

\smallskip

\noindent$\bullet$ The proof of Theorem \ref{thm:main-2} is finally explained in Section~\ref{sec:conclusion}. 

\smallskip

\noindent$\bullet$ Appendix~\ref{app:counter} contains some material on the counter-term problem introduced in Section \ref{sec:inhom dir}. Most of this is taken from~\cite{FroKnoSchSoh-17} and reproduced for the convenience of the reader. 

\smallskip

\noindent$\bullet$ Appendix~\ref{app:interpretation} discusses the physical  interpretation of our result in light of the phase transition of the Bose gas. We start with the non-interacting case $w\equiv0$ in all dimensions and explain the emergence of Gaussian measures, since we have found this nowhere in the physical or mathematical literature (although this is certainly implicit e.g. in~\cite{ArnMoo-01,BayBlaiHolLalVau-99,BayBlaiHolLalVau-01,HolBay-03,KasProSvi-01}). Then we reformulate our result with interactions in microscopic variables.

\section{Classical measures and a priori bounds}\label{sec:class meas}

In this section, we collect some useful facts on the classical Gibbs measures we derive from the quantum problem, and on the semiclassical de Finetti measures that serve as one of our main tools. We will also recall some basic properties of the many-body quantum Gibbs state and prove a collection of a priori bounds to be used throughout the paper.  

\subsection{Gibbs measures} \label{sec:renormalized-measure}

We do not claim originality for the material below, the methods having been well-known to constructive quantum field theory experts for a long time. A related discussion can be found in~\cite[Section~3]{FroKnoSchSoh-17} but for pedagogical purposes we follow a somewhat more pedestrian route. 

In this section, we always assume that $h$ satisfies $\tr[h^{-p}]<\ii$ for some $p\geq1$. Let $\{\lambda_i\}_{i=1}^\infty$, $\{u_i\}_{i=1}^\infty$ be the eigenvalues and the corresponding eigenfunctions of $h$. Let us start by recalling the definition of the Gaussian measure: 

\begin{lemma}[\textbf{Free Gibbs measure}]\label{lem:free-meas}\mbox{}\\
Let $h>0$ on $\gH$ satisfy 
$$ 
\tr [h ^{-p}] < \infty\qquad\text{for some $p\geq1$,}
$$
with eigenvalues $\{\lambda_i\}_{i=1}^\infty$ and eigenfunctions $\{u_i\}_{i=1}^\infty$. The Gaussian measure $\mu_0$ of covariance $h^{-1}$ is the unique probability measure over the space $\gH^{1-p}$ such that for every $K\ge 1$ its cylindrical projection on $V_K={\rm Span}(u_1,...,u_K)$ is 
\begin{equation} \label{eq:def-mu0K}
d\mu_{0,K}(u)=\prod_{i = 1}^K  \left( \frac{\lambda_i}{\pi}e^{ - {\lambda} _i|\alpha_i|^2}\,d\alpha_i \right)
\end{equation}
where $\alpha_i=\langle u_i, u\rangle$ and $d\alpha_i=d\Re(\alpha_i)\,d\Im(\alpha_i)$ is the Lebesgue measure on $\C\approx\R^d$. Moreover, the corresponding $k$-particle density matrix
\begin{equation}\label{eq:DM free meas}
\gamma_{\mu_0}^{(k)}:=\int |u^{\otimes k}\rangle\langle u^{\otimes k}|\;d\mu_0(u) = k!\,(h^{-1})^{\otimes k}
\end{equation}
belongs to the Schatten space $\gS ^p \left(\gH^{\otimes_s k} \right)$. 
\end{lemma}

Our convention in~\eqref{eq:DM free meas} is to consider the action of $\gamma_{\mu_0}^{(k)}$ only on the symmetric subspace $\gH^{\otimes_s k}$. On the full space with no symmetry we have 
\begin{equation}\label{eq:notation merdique}
\gamma_{\mu_0}^{(k)} = k! P_s ^k (h^{-1})^{\otimes k} P_s^k
\end{equation}
with $P_s ^k$ the orthogonal projector on the symmetric subspace.

 \begin{proof} 
See~\cite[Section~3.1]{LewNamRou-15}
\end{proof}

The measure just defined on the space $\gH^{1-p}$ does {\em not} live on any better behaved subspace if $\tr[h^{-p'}]=+\ii$ for $p'<p$. This is called Fernique's theorem and is recalled e.g. in~\cite[Equation~(3.4)]{LewNamRou-15}. The need for renormalization arises from this fact. 

In particular, when $\tr[h^{-1}]=+\ii$, $\mu_0$ is supported on a negative Sobolev space and thus the mass $\int_{\R^d}|u|^2$ is equal to infinity $\mu_0$-a.e. However, it turns out that when $\tr[h^{-2}]<\ii$, this infinity ``is the same'' for $\mu_0$-almost every $u$. This  allows us to define a notion of renormalized mass. The idea goes back to Nelson~\cite{Nelson-66} and has been thoroughly studied in constructive quantum field theory~\cite{GliJaf-87,Simon-74}.

\begin{lemma}[\textbf{Mass renormalization}]\label{lem:re-mass}\mbox{}\\ 
Let $h>0$ satisfy $\tr [h ^{-2}] < \infty$. For every $K\ge 1$, define  the truncated renormalized mass 
\begin{equation}\label{eq:renorm mass}
\cM_K[u]:=\norm{P_Ku}^2 - \int \norm{P_Ku}^2 \,d\mu_0(u)
\end{equation}
where $P_K$ is the orthogonal projection onto $V_K = \mathrm{Span}(u_1,\ldots,u_K)$. Then the sequence $\cM_K$ converges strongly to a limit $\cM$ in $L^2(d\mu_0)$.

More generally, for every operator $A$ with $D(A)\subset D(h)$ and such that $Ah^{-1}$ is Hilbert-Schmidt, the renormalized expectation value
\begin{equation}\label{eq:renorm mass A}
\cM_K^A [u]:= \langle P_K u ,A P_K u \rangle -\int \langle P_K u ,A P_K u \rangle \,d\mu_0(u)
\end{equation}
converges  strongly in $L^2(d\mu_0)$ to a limit $\cM ^A$. The limit is uniform in $A$ on sets where the Hilbert-Schmidt norm $\|Ah^{-1}\|_{\gS^2(\gH)}$ is bounded by a constant. In fact 
\begin{equation}\label{eq:renorm mass A bis}
\int \left|\cM ^A [u] \right|^2\,d\mu_0(u)= \lim_{K\to \infty} \int\left|\cM_K^A [u]\right|^2\,d\mu_0(u) = \tr\left[ A h ^{-1} A^* h^{-1} \right]. 
\end{equation}
\end{lemma}

\begin{proof} 
Writing $P_Ku=\sum_{j=1}^K\alpha_j\,u_j$, we first recall the simple Gaussian integration formulae (Wick's theorem)
\begin{equation}\label{eq:Wick 2}
\av{\bar{\alpha_i} \alpha_j} = \frac{1}{\lambda_j} \delta_{i=j},\qquad \av{\alpha_i \bar{\alpha}_j \bar{\alpha}_k \alpha_\ell } = \frac{1}{\lambda_i \lambda_k} \delta_{i=j} \delta_{k=\ell} + \frac{1}{\lambda_i \lambda_\ell} \delta_{i=k} \delta_{j=\ell}.  
\end{equation}
Then we compute
$$\cM_K[u]=\sum_{j=1}^K|\alpha_j|^2-\av{\sum_{j=1}^K|\alpha_j|^2}=\sum_{j=1}^K|\alpha_j|^2-\sum_{j=1}^K\lambda_j^{-1}.$$
Therefore, for $L\geq K$,
$$\cM_L [u]-\cM_K[u]=\sum_{j=K+1}^L(|\alpha_j|^2-\lambda_j^{-1}).$$
From this and \eqref{eq:Wick 2} we find 
\begin{align*}
\av{(\cM_L [u]-\cM_K[u])^2}&= \sum_{j=K+1}^L\sum_{\ell=K+1}^L\av{(|\alpha_j|^2-\lambda_j^{-1})(|\alpha_\ell|^2-\lambda_\ell^{-1})}\\
&= \sum_{j=K+1}^L\sum_{\ell=K+1}^L\left(\av{|\alpha_j|^2|\alpha_\ell|^2}-\lambda_j^{-1}\lambda_\ell^{-1}\right)= \sum_{j=K+1}^L {\lambda_j^{-2}}.
\end{align*}
Since 
\begin{equation}\label{eq:Hilbert Schmidt}
\sum_{j=1}^\infty \lambda_j^{-2}=\tr(h^{-2})<\infty,
\end{equation}
we conclude that $\{\cM_K\}_{K=1}^\infty$ is a Cauchy sequence in $L^2(d\mu_0)$ and hence it converges strongly in $L^2(d\mu_0)$.

Now consider an operator $A$ such that $Ah^{-1}$ is Hilbert-Schmidt. We check that~\eqref{eq:renorm mass A} is in $L^2 (d\mu_0)$ uniformly in $K$ and leave to the reader the similar proof that it is a Cauchy sequence. 
Using~\eqref{eq:Wick 2} again we have 
\begin{align*}
\av{\cM_K^A [u] ^2} &= \sum_{1\leq i,j,k,\ell\leq K} \av{\alpha_i \bar{\alpha}_j \bar{\alpha}_k \alpha_\ell } \left\langle u_i, A u_j \right\rangle \overline{\left\langle u_k, A u_\ell \right\rangle} \\
&\quad \qquad- \sum_{1\leq i,j,k\leq K}\av{\alpha_i \bar{\alpha}_j}\frac{1}{\lambda_k} \left\langle u_i, A u_j \right\rangle \overline{\left\langle u_k, A u_k \right\rangle} \\
&\quad \qquad -\sum_{1\leq i,j,k\leq K}\av{\bar{\alpha_i} \alpha_j}\frac{1}{\lambda_k} \overline{\left\langle u_i, A u_j \right\rangle} \left\langle u_k, A u_k \right\rangle+ \sum_{i,j} \frac{1}{\lambda_i \lambda_j} \left\langle u_j, A u_j \right\rangle \overline{\left\langle u_i, A u_i \right\rangle} \\
&= \sum_{1\leq i,k\leq K} \frac{1}{\lambda_i \lambda_k} \left\langle u_i, A u_k \right\rangle \overline{\left\langle u_i, A u_k \right\rangle}= \tr\left[A(P_Kh^{-1})A^*(P_Kh^{-1})\right].
\end{align*}
which is bounded by assumption, uniformly in $K$. 
\end{proof}

From the previous result we may choose $A$ to be the multiplication operator by some bounded function $f$ localized around some point $x$. Thus~\eqref{eq:renorm mass A bis} implies that not only the global mass, but also the smeared local mass density $|u(x)| ^2$ around $x$, is renormalizable. In fact, in this case~\eqref{eq:renorm mass A bis} reduces to 
\begin{equation}\label{eq:renorm mass f}
\av{\left|\cM ^f [u]\right|^2} =  \tr\left[ f (x) h ^{-1} \bar{f}(x) h^{-1} \right] = \iint_{\Omega \times \Omega} f(x) \overline{f (y)} |G(x,y)|^2 dx dy  
\end{equation}
with $G$ the Green function of $h$ (i.e. the integral kernel of $h^{-1}$). This function is certainly square-integrable on $\Omega \times \Omega$ when $\Tr[h^{-2}]<\infty$. 

Consider now an interaction potential  $w$ as in \eqref{eq:interaction}. For the same reasons as the mass, the interaction energy is not well-defined on the support of $\mu_0$ and it is necessary to introduce a renormalization. In our context, it is in fact sufficient to insert the local mass renormalization in the interaction's expression. This leads to

\begin{lemma}[\textbf{Renormalized  interaction and nonlinear measure}]\label{lem:re-interaction}\mbox{}\\ 
Let $h>0$ satisfy $\tr [h ^{-2}] < \infty$. Let $w\in L^\ii(\Omega)$ be such that its Fourier transform satisfies $0\le \widehat w \in L^1(\Omega^*)$. For every $K\ge 1$, define the truncated renormalized interaction as in~\eqref{eq:def int}:
$$
 \cD_K[u]:=\frac{1}{2}\iint_{\Omega\times \Omega}  w(x-y)\left(|P_K u(x)|^2-\av{|P_K u(x)|^2}\right)  \left(|P_Ku(y)|^2-\av{|P_Ku(y)|^2}\right)dx\,dy.
$$
Then $\cD_K[u]\ge 0$ and $\cD_K[u]$ converges strongly to a limit $\cD[u]\ge 0$ in $L^1(d\mu_0)$.
Consequently, the probability measure 
\begin{equation}\label{eq:def mu}
d\mu (u) := \frac{1}{z} e^{-\cD [u ]} d\mu_0 (u),\qquad z=\int e^{-\cD[u]}\,d\mu_0(u) 
\end{equation}
is well-defined. Moreover, the reduced density matrices 
\begin{equation}\label{eq:DM clas}
\gamma_\mu ^{(k)} := \int |u^{\otimes k} \rangle \langle u ^{ \otimes k}| d\mu (u) 
\end{equation}
belong to $\gS^p (\gH_k)$ for every $p$ such that $\tr[h^{-p}]<\ii$. Finally, the relative one-particle density matrix 
\begin{equation}\label{eq:DM rel clas}
\gamma_{\mu} ^{(1)} - \gamma_{\mu_0} ^{(1)} := \lim_{K\to \infty} \int |u \rangle \langle u | \left(d\mu \left(P_K u\right) - d\mu_0 \left(P_K u\right)\right) 
\end{equation}
is a trace-class operator, with
\begin{equation}\label{eq:DM rel clas norm}
\tr\left| \gamma_{\mu} ^{(1)} - \gamma_{\mu_0} ^{(1)}\right|\leq  z^{-1/2} \sqrt{\tr[h^{-2}]}. 
\end{equation}
\end{lemma}

\begin{proof} 
We use the Fourier transform 
$$
w(x-y)=\int \widehat w(k) e^{ik\cdot x} e^{-ik\cdot y} d k 
$$
and denote $e_k$ the multiplication operator by $e^{ik \cdot x}$. Then, using the notation of Lemma~\ref{lem:re-mass}, 
\begin{equation}\label{eq:add int}
 \cD_K[u]=\frac{1}{2}\int \hat{w} (k) \left|\cM_K ^{e_k} [u]\right|^2 \, dk 
\end{equation}
Since $\hat{w}\ge 0$ by assumption, we obtain immediately that $\cD_K[u]\ge 0$. 

In order to prove that $\cD_K$ is a Cauchy sequence in $L^1(d\mu_0)$, we use the Cauchy-Schwarz inequality 
\begin{align}
2\left| \cD_L [u]-\cD_K[u] \right|&=\left| \int \hat{w}(k) \left( \left|\cM_K ^{e_k} [u]\right|^2 - \left|\cM_L ^{e_k} [u]\right|^2 \right)  \, dk \right| \nn \\
&\leq  \left| \int \hat{w} (k)  \left|\cM_K ^{e_k} [u] -\cM_L ^{e_k} [u]\right|^2  \, dk \right|^{\frac12}\left| \int \hat{w}(k)  \left( \left|\cM_K ^{e_k} [u]\right| +  \left|\cM_L ^{e_k} [u]\right|\right)^2  \, dk \right|^{\frac12}.\label{eq:Cauchy int}
\end{align}
Averaging over $\mu_0$, using Lemma~\ref{lem:re-mass} and recalling that $\hat{w} \in L^1$, this goes to zero when $L,K\to \infty$. Thus $\cD_K[u]$ is a Cauchy sequence in $L^1(d\mu_0)$ and it converges strongly to a limit $\cD[u]$. Since $\cD_K[u]\ge 0$, we have $\cD[u] \ge 0$. It follows from Jensen's inequality that 
\begin{equation}\label{eq:def zr}
 z := \int e^{-\cD [u]} d\mu_0 (u) > 0,  
\end{equation}
which ensures that~\eqref{eq:def mu} is well-defined. 

That the density matrices~\eqref{eq:DM clas} are in $\gS ^p$ directly follows from the positivity of the renormalized interaction and the corresponding statement for the free density matrices~\eqref{eq:DM free meas}. To see that the relative one-particle density matrix~\eqref{eq:DM rel clas} is trace-class, note that for any finite-rank operator $A$
$$ 
\tr\left[ A \left(\gamma_{\mu} ^{(1)} - \gamma_{\mu_0} ^{(1)}\right) \right] = \int \cM^A [u] d\mu(u). 
$$
Using Cauchy-Schwarz and the fact that $\mu\leq z^{-1}\mu_0$ we find
\begin{align*}
\left|\tr\left[ A \left(\gamma_{\mu} ^{(1)} - \gamma_{\mu_0} ^{(1)}\right) \right]\right| &\leq  \int \left|\cM^A [u]\right| d\mu(u)\\
&\leq  \left(z^{-1} \int  \cM^A [u]^2 d\mu_0(u)\right)^{1/2}\leq z^{-1/2}\|A\| \sqrt{\tr[h^{-2}]}.
\end{align*}
By duality, the relative one-particle density matrix is thus trace-class, with
$$\tr\left| \gamma_{\mu} ^{(1)} - \gamma_{\mu_0} ^{(1)}\right|\leq  z^{-1/2} \sqrt{\tr[h^{-2}]}.$$
\end{proof}

\begin{remark}[The interaction as an exchange term]\label{rem:exchange}\mbox{}\\
Note that, by Gaussian integrations similar to those appearing in the proof of Lemma~\ref{lem:re-mass}, we obtain that, formally 
\begin{multline*}
\av{\iint_{\Omega \times \Omega} |u(x)| ^2 w(x-y) |u(y)| ^2 dxdy  } = \text{``}\;\frac{1}{2}\iint_{\Omega\times \Omega} w(x-y)G (x,x) G(y,y) \,dx\,dy \\+ \frac{1}{2}\iint_{\Omega\times \Omega} w(x-y)|G (x,y)|^2\,dx\,dy\;\text{''}
\end{multline*}
where the first term is called the \emph{direct term} and the second the \emph{exchange term}, see Section~\ref{sec:app-H-vs-rH}. Here the direct term is infinite because $\lim_{y\to x}G(x,y)=+\ii$. For instance for the Laplacian on a bounded  domain $\Omega$ we have that $G(x,y)$ behaves like  (see e.g.~\cite[Lemma~5.4]{RouSer-16}). 
$$
G(x,y) \underset{x- y\to0}{\sim} 
\begin{cases}
- \frac{1}{2\pi}\log |x-y|&\text{if $d=2$,}\\
\frac{1}{4\pi|x-y|}&\text{if $d=3$.}
\end{cases}
$$
On the other hand, averaging~\eqref{eq:add int} with respect to $\mu_0$ and using~\eqref{eq:renorm mass f} we find, after passing to the limit $K\to\ii$,
\begin{align}\label{eq:exchange}
0\leq \av{\cD [u]}&=\frac{1}{2} \int \hat{w} (k) \av{|\cM ^{e_k} [u]|^2} \, dk\nn\\
&= \frac{1}{2}\int \hat{w} (k) \iint_{\Omega \times \Omega} e ^{ik \cdot(x-y)} |G(x,y)|^2 dxdy \, dk\nn \\
&=\frac{1}{2} \iint_{\Omega\times\Omega} w(x-y)|G(x,y)|^2\,dx\,dy.
\end{align}
We thus see that renormalizing the mass density to define the interaction is equivalent to dropping the direct term from the bare interaction. Actually,~\eqref{eq:exchange} proves that the renormalized interaction is well defined under the sole condition that $w$ satisfies
$$\iint_{\Omega\times\Omega} |w(x-y)|\,|G(x,y)|^2\,dx\,dy<\ii$$
with $G(x,y)=h^{-1}(x,y)$ the Green function of $h$. \hfill$\diamond$
\end{remark}

\subsection{De Finetti measure} \label{sec:deF}

Here we review how to associate a semiclassical measure (that we call de Finetti measure) on the one-body Hilbert space to a given sequence of many-particles bosonic states. This idea has a long history, for it is related to the de Finetti-Hewitt-Savage theorem used in classical statistical mechanics to approximate a many-particle state by a statistical mixture of i.i.d. laws. See~\cite{Rougerie-cdf,Rougerie-LMU} for review. For more on the connection with usual semi-classics, see~\cite{Ammari-hdr} and references therein.

The approach we use in this paper is a blend of ideas originating from semi-classical analysis~\cite{AmmNie-08,AmmNie-09,AmmNie-11,Berezin-72,Lieb-73b,Simon-80} and quantum information theory~\cite{BraHar-12,ChrKonMitRen-07,Chiribella-11,Harrow-13} with many-body localization methods~\cite{Ammari-04,DerGer-99,Lewin-11,LewNamRou-14,LewNamRou-16c,LewNamRou-15}. 

It will be crucial for us that, once the one-body state-space is projected to finitely many dimensions, quantitative estimates on the error made by approximating a many-body state using a classical measure are available~\cite{ChrKonMitRen-07,LewNamRou-15b}. We thus begin by recalling what Fock-space localization is. Then we continue with the quantitative version of the quantum de Finetti theorem available after finite dimensional localization. To deal with the entropy term it is crucial that the de Finetti measure we use is in fact a lower symbol (associated with a coherent states basis). This allows us to use a Berezin-Lieb-type inequality from~\cite{LewNamRou-15}.  

\medskip

\noindent\textbf{Fock-space localization.} We will localize the problem to low kinetic energy modes. For this purpose, let us recall the standard {\em localization method} in Fock space.  Let $P$ be an orthogonal projection on $\gH$ and let $Q=\1-P$. Since $\gH=(P\gH)\oplus (Q\gH)$, we have the corresponding factorization of Fock spaces 
\begin{equation}
\gF\approx \gF(P\gH)\otimes\gF(Q\gH)
\label{eq:factorization_Fock_space}
\end{equation}
in the sense of a unitary equivalence. That is, there is a unitary
\begin{equation}\label{eq:Fock factor}
 \cU : \gF ( P\gH \oplus Q\gH) \mapsto \gF(P\gH) \otimes \gF (Q\gH) 
\end{equation}
satisfying $\cU\cU^* = \1$.  Its action on creation operators is 
\begin{equation}\label{eq:loc creation}
 \cU a ^*( f ) \cU ^*  = a ^*(Pf) \otimes \1 + \1 \otimes a^*(Q f) 
\end{equation}
and a similar formula for annihilation operators. We refer to~\cite[Appendix~A]{HaiLewSol_2-09} and references therein for precise definitions and properties.

\begin{definition}[\textbf{Fock-space localization}]\label{def:localization}\mbox{}\\
For any state $\Gamma$ on $\gF$ and any orthogonal projector $P$, we define its \emph{localization} $\Gamma_{P}$ as a state on $\gF$ obtained by taking the partial trace over $\gF(Q\gH)$:
$$\Gamma_{P}:= \tr_{\gF(Q\gH)}\left[ \cU \Gamma \cU^* \right].$$
The density matrices of $\Gamma_P$ can be shown to be equal to
\begin{equation}
(\Gamma_P)^{(k)}=P^{\otimes k}\Gamma^{(k)}P^{\otimes k},\qquad \forall k\geq1.
\label{eq:GammaV-k}
\end{equation} 
\hfill$\diamond$
\end{definition}

The crucial property~\eqref{eq:GammaV-k} follows immediately from~\eqref{eq:DM ann cre} and~\eqref{eq:loc creation}, see again~\cite[Appendix A]{HaiLewSol_2-09} and references therein for detailed discussions. An equivalent definition leading to~\eqref{eq:GammaV-k} originates in~\cite{Lewin-11} and is reviewed in~\cite[Chapter~5]{Rougerie-LMU}.

\medskip

\noindent\textbf{Coherent states and lower symbols.} The de Finetti measure is in fact a lower symbol with respect to the over-complete basis of $\gF (P\gH)$ given by coherent states when $P$ is a finite-dimensional orthogonal projector. Below, the notation $|0\rangle = 1 \oplus 0 \oplus 0 \oplus \ldots$ stands for the vacuum vector of Fock space.

\begin{definition}[\textbf{Coherent states}]\label{def:coherent state}\mbox{}\\
A {\em coherent state} is a Weyl-rotation of the vacuum $|0\rangle$ in the Fock space $\gF$
\begin{equation}\label{eq:coherent state}
\xi(u):=W(u) |0\rangle := \exp(a^\dagger(u)-a(u))  |0\rangle = e^{-\norm{u}^2/2} \exp\left(a^\dagger(u)\right)  |0\rangle = e^{-\norm{u}^2/2} \bigoplus_{n \ge 0} \frac{1}{\sqrt{n!}} u^{\otimes n},
\end{equation}
for $u\in\gH$. \hfill$\diamond$
\end{definition}

The Weyl operator $W(u)$ is a unitary operator translating creation and annihilation operators
\begin{equation}
 \label{eq:Weyl-action}
W(f)^* a^\dagger(g) W(f)=a^\dagger(g) + \langle f,g \rangle, \quad W(f)^* a(g) W(f)=a(g) + \langle g,f \rangle.
\end{equation}
The $k$-particle density matrix of the coherent state $\xi(u)$ is
\begin{equation}
|\xi(u)\>\<\xi(u)|^{(k)}=\frac{1}{k!}|u^{\otimes k}\>\<u^{\otimes k}|.
\label{eq:DM_coherent}
\end{equation}

\begin{definition}[\textbf{Lower symbol}]\label{def:lower symbol}\mbox{}\\
For any state $\Gamma$ on $\gF$ and any scale $\eps>0$, we define the \emph{lower symbol} (or {\em Husimi function}) of $\Gamma$ on $P\gH$ at scale $\eps$ by
\begin{equation} \label{eq:Husimi}
d\mu_{P,\Gamma}^{\eps}(u):=(\eps\pi)^{-\Tr(P)}\pscal{\xi(u/\sqrt{\eps}),\Gamma_P \xi(u/\sqrt{\eps})}_{\gF(P\gH)} d u.
\end{equation}
Here $du$ is the usual Lebesgue measure on $P\gH \simeq \C^{\Tr(P)}$.  \hfill$\diamond$
\end{definition}

Thanks to the resolution of the identity/closure relation (see e.g.~\cite{LewNamRou-15})
\begin{equation}
\pi^{-\Tr(P)} \int_{P\gH} |\xi(u)\rangle \langle \xi (u)| du = \pi^{-\Tr(P)} \left( \int_{P\gH} e^{-|u|^2} du \right) \1_{\F(V)} = \1_{\F(P\gH)},
\label{eq:resolution_coherent2}
\end{equation}
the lower symbol $\mu_{P,\Gamma}^{\eps}(u)$ is a probability measure on $P\gH$. Moreover, it provides a good approximation for the density matrices $\Gamma_P^{(k)}$, as per the following version of the \emph{quantum de Finetti theorem}: 

\begin{theorem}[\textbf{Lower symbols as de Finetti measures}]\label{thm:quant deF}\mbox{}\\
We have, for all $k\in\N$,
\begin{equation}\label{eq:Chiribella}
\int_{P\gH}|u^{\otimes k}\>\<u^{\otimes k}|\;d\mu^\eps_{P,\Gamma}(u) = k!\eps^k\Gamma^{(k)}_P + k! \eps ^k \sum_{\ell = 0} ^{k-1} {k \choose \ell} \Gamma^{(\ell)} \otimes_s \one_{\otimes_s ^{k-\ell} P\gH}.
\end{equation}
Thus, with $d=\Tr [P]$, 
\begin{align}
\Tr \left| k!\eps^k\Gamma^{(k)}_P-\int_{P\gH}|u^{\otimes k}\>\<u^{\otimes k}|\;d\mu^\eps_{P,\Gamma}(u) \right| \leq \eps^k \sum_{\ell=0}^{k-1}{k\choose \ell}^2  \frac{(k-\ell +d-1)!}{(d-1)!}\tr \left[ \cN^{\ell}\Gamma_P\right].
\label{eq:quantitative}
\end{align}
\end{theorem}

The result is taken from \cite[Lemma 6.2 and Remark 6.4]{LewNamRou-15}. It is an elaboration on a theorem originating on~\cite{ChrKonMitRen-07} and a proof thereof later provided in~\cite{LewNamRou-15b}. If $k$ is fixed and $\tr\left[(\eps \cN)^{\ell}\Gamma_P\right]= O(1)$, then the upper bound in~\eqref{eq:quantitative} behaves as $Cd\eps$ in the limit $\eps \to0$. This is similar to the bound $4kd/N$ obtained for $N$-particle states in the references just mentioned. 

Finally, we recall a Berezin-Lieb type inequality, which links the von Neumann relative entropy of two quantum states to the classical Boltzmann entropy of their lower symbols.

\begin{theorem}[\textbf{Relative entropy: quantum to classical}] \label{thm:rel-entropy}\mbox{}\\
Let $\Gamma$ and $\Gamma'$ be two states on $\gF$. Let $\mu_{P,\Gamma}^\eps$ and $\mu_{P,\Gamma'}^\eps$ be the lower symbols defined in \eqref{eq:Husimi}.  Then we have
\begin{equation}
\cH(\Gamma,\Gamma')\geq \cH(\Gamma_P,\Gamma'_P)\geq \cHcl(\mu_{P,\Gamma}^\eps,\mu_{P,\Gamma'}^\eps).
\label{eq:Berezin-Lieb}
\end{equation}
\end{theorem}

The result is taken from \cite[Theorem 7.1]{LewNamRou-15}, whose proof goes back to the techniques in~\cite{Berezin-72,Lieb-73b,Simon-80}. Note that, to obtain an approximation of density matrices, other constructions than that based on the lower symbol we just discussed are available~\cite{BraHar-12,CavFucSch-02,HudMoo-75,Stormer-69,Rougerie-19}. Some of those give good quantitative estimates and the main reason for us to rely on lower symbols (a.k.a. Husimi functions, covariant symbols, anti-Wick symbols) is that~\eqref{eq:Berezin-Lieb} is heavily based on their properties.

\subsection{Gaussian quantum states in the limit $\lambda\to0^+$}\label{sec:free-Gibbs-state} 
Now we introduce the parameter $\lambda$ and study the limit $\lambda\to0$. We start by stating some simple properties of the Gaussian quantum states. Let $h>0$ satisfy $\tr[h^{-p}]<\ii$ for some $p\geq1$. Then, $\tr[e^{-\beta h}]<\ii$ for all $\beta>0$ and we may define the associated Gaussian quantum state by
\begin{equation}
\boxed{\Gamma_0 = \frac{e^{-\lambda\dG(h)}}{\cZ_0(\lambda)} , \qquad \cZ_0(\lambda) = \tr_{\gF}\left[e^{-\lambda\dG(h)}\right].}
\label{eq:reminder_free}
\end{equation}

We recall (see e.g.~\cite[Appendix~A]{LewNamRou-15}) that the partition function of the Gaussian quantum state satisfies
\begin{equation}
-\log \cZ_0(\lambda) = -\log \tr_{\gF}(e^{-\lambda\dG(h)}) = \tr_{\gH}\left(\log(1-e^{-\lambda h})\right) :=\lambda F_0
\label{eq:partition_quasi_free}
\end{equation}
where the free-energy 
$$ F_0 = \cF_0 [\Gamma_0] = \min \cF_0 [\Gamma] = \min \left\{ \tr\left( \dG (h) \Gamma \right) + \lambda^{-1} \tr\left( \Gamma \log \Gamma \right)\right\}$$
is the infimum over all states of the free-energy functional associated with $\bH_0$.
We collect some of its first properties in the

\begin{lemma}[\textbf{Gaussian quantum state}]\label{lem:quasi-free}\mbox{}\\
Let $h>0$ satisfy 
\begin{equation}\label{eq:Schatten h again}
\tr(h^{-p})<\ii\qquad\text{for some $p\geq1$.} 
\end{equation}
The $k$-particle density matrix of the Gaussian quantum state is given by
\begin{equation}\label{eq:DM_quasi_free}
\Gamma_0^{(k)}=\left(\frac{1}{e^{\lambda h}-1}\right)^{\otimes k} \le \lambda^{-k} (h^{-1})^{\otimes k}.
\end{equation}
Consequently, for every $k\ge 1$,
\begin{equation}\label{eq:CV-DM-G0}
\lambda^k k!\;\Gamma_0^{(k)} \underset{\lambda\to0}\longrightarrow k! (h^{-1})^{\otimes k} = \int |u^{\otimes k} \rangle \langle u^{\otimes k} | d\mu_0(u)
\end{equation}
strongly in the Schatten space $\gS^p(\gH^k)$. Moreover,
\begin{equation} \label{eq:cNk-G0}
\Tr_\gH \left[ \Gamma_0^{(k)}\right] = \left \langle {\cN \choose k}\right\rangle_0  \le C_k \Tr(h^{-p})^k \lambda^{-pk}.
\end{equation}
and, if in addition $p>1$,
\begin{equation} \label{eq:cNk-G0-limit}
\lim_{\lambda \to0}\lambda^{pk}\Tr \left[ \Gamma_0^{(k)}\right] =0.
\end{equation}
Here $\cN$ is the particle number operator and 
$$\langle \cdot \rangle_0 := \tr_{\gF} \left( \cdot \, \Gamma_0 \right)$$
is the expectation against $\Gamma_0$ in the Fock space $\gF$. 
\end{lemma}

Regarding~\eqref{eq:DM_quasi_free} we recall the notational convention discussed around Equation~\eqref{eq:notation merdique}. To illustrate the bounds on the particle number~\eqref{eq:cNk-G0}-\eqref{eq:cNk-G0-limit}, recall that in the homogeneous case where $h = -\Delta+ \kappa$ on a box with periodic boundary conditions we have that 
\begin{align} \label{eq:p-N0}
\begin{cases}
 \mbox{in 1D~\eqref{eq:Schatten h again} holds with } p=1 \mbox{ and } \left\langle \cN \right\rangle_0 \sim \lambda ^{-1} \\
 \mbox{in 2D~\eqref{eq:Schatten h again} holds with any } p>1 \mbox{ and } \left\langle \cN \right\rangle_0 \sim -\lambda^{-1}\log \lambda  \\
\mbox{in 3D~\eqref{eq:Schatten h again} holds with any } p>3/2 \mbox{ and } \left\langle \cN \right\rangle_0 \sim \lambda^{-3/2} 
 \end{cases}
\end{align}

\begin{proof} 
Formula \eqref{eq:DM_quasi_free} is taken from \cite[Lemma 2.1]{LewNamRou-15}. Formula \eqref{eq:CV-DM-G0} follows from the monotone convergence of operators and the fact that $h^{-1}$ belongs to the Schatten space $\gS ^p(\gH)$ by the assumption \eqref{eq:Schatten h}. Finally, \eqref{eq:cNk-G0} holds true because 
\begin{align*}
\Tr\left[\Gamma_0^{(k)}\right] = \left \langle {\cN \choose k}\right\rangle_0  \le C_k \left \langle \cN \right\rangle_0^k = C_k \left( \Tr \left[ \frac{1}{e^{\lambda h}-1} \right]\right)^k  \le C_k  \left( \Tr \left[ (\lambda h)^{-p} \right]\right)^k. 
\end{align*}
We have used here that
\begin{equation}
 \frac{1}{e^x-1}\leq \frac{C_p}{x^p},\qquad \forall p\geq1.
 \label{eq:exp_bound}
\end{equation}
Here in the first estimate we have used Wick's formula for the quasi-free state $\Gamma_0$. 
We then remark that 
$$ \lambda^p\Tr \left[ \frac{1}{e^{\lambda h}-1} \right]=\sum_{j\geq1}\frac{\lambda^p}{e^{\lambda \lambda_j}-1}$$
so the limit~\eqref{eq:cNk-G0-limit} follows from the dominated convergence theorem, since 
$$\frac{\lambda^p}{e^{\lambda\lambda_j}-1}\leq \min\left\{\frac{\lambda^{p-1}}{\lambda_j},\frac{C_p}{(\lambda_j)^p}\right\}$$
by~\eqref{eq:exp_bound} and $\lambda^{p-1}/\lambda_j\to0$ when $p>1$.
\end{proof}

The following result is the counterpart of Lemma~\ref{lem:re-mass}. It shows in particular that the renormalized mass is also bounded independently of $\lambda$ for the Gaussian quantum state (take $A=1$ in the statement). 

\begin{lemma} [\textbf{Variance estimate for the Gaussian quantum state}]\label{lem:quasi-free mass}\mbox{}\\ 
Let $h>0$ satisfy $\tr[h^{-2}]<\ii$.
For every bounded self-adjoint operator $A$, we have
\begin{equation}\label{eq:variance free gibbs}
\lim_{\lambda\to0} \lambda^2 \left\langle \Big| \dGamma(A)- \langle \dGamma(A) \rangle_{0}\Big|^2\right\rangle_{0} = \Tr [Ah^{-1}Ah^{-1}]
\end{equation}
where $\langle \cdot \rangle_0$ is the expectation against $\Gamma_0=e^{-\lambda\dG(h)}/\tr[e^{-\lambda\dG(h)}]$ in the Fock space $\gF$. 
\end{lemma}

\begin{proof} 
Pick an orthonormal basis $(u_i)$ of $\gH$ and denote $a_i ^*,a_i$ the associated creation and annihilation operators. Since $\Gamma_0$ is a quasi-free state and it commutes with $\cN$, we can compute explicitly, using the CCR and Wick's theorem:
\begin{align*}
\Big\langle |\dGamma(A)|^2 \Big\rangle_0 &= \sum_{m,n,p,q=1}^\infty  \langle u_m, Au_n\rangle\langle u_p, Au_q\rangle \langle a_m^* a_n a_p^* a_q \rangle_0 \\
&= \sum_{m,n,p,q=1}^\infty  \langle u_m, Au_n\rangle\langle u_p, Au_q\rangle \left( \langle a_m^* a_n \rangle_0 \langle a_p^* a_q \rangle_0 + \langle a_m^* a_q \rangle_0 \delta_{np} + \langle a_m^* a_q \rangle_0 \langle a_p^* a_n \rangle_0 \right)\\
&= \Big\langle \dGamma(A) \Big\rangle_0^2+ \tr\left[A^2\Gamma_0^{(1)}\right]+ \Tr \left[A\Gamma_0^{(1)}A\Gamma_0^{(1)}\right].
\end{align*}
Then using \eqref{eq:CV-DM-G0} and \eqref{eq:cNk-G0}, we conclude that  
\begin{align}\label{eq:proof variance free}
\lambda^{2}\left\langle \Big| \dGamma(A)- \langle \dGamma(A) \rangle_{0}\Big|^2\right\rangle_{0} &= \lambda^{2} \left( \left\langle |\dGamma(A)|^2 \right\rangle_0 - \left\langle \dGamma(A) \right\rangle_0^2\right)\nonumber\\
&= \lambda^{2} \tr\left[A^2\Gamma_0^{(1)}\right] + \lambda^2 \Tr\left[ A\Gamma_0^{(1)}A\Gamma_0^{(1)}\right] \nn\\
&\underset{\lambda\to0}{\longrightarrow} \Tr [Ah^{-1}Ah^{-1}].
\end{align}
In the last line we have used that
$$\lambda^2\tr\left[A^2\Gamma_0^{(1)}\right]\leq \lambda^2\|A\|^2\tr\Gamma_0^{(1)}\underset{\lambda\to0}\longrightarrow0$$
by~\eqref{eq:cNk-G0-limit} for $p=2$.
\end{proof}

The term obtained in~\eqref{eq:variance free gibbs} is the limit of the exchange term. An important consequence of the previous lemma is that the Gaussian quantum state $\Gamma_0$ has a bounded renormalized interaction energy:

\begin{lemma} [\textbf{Interaction energy of the Gaussian quantum state}]\label{lem:quasi-free ener}\mbox{}\\ 
Let $h>0$ satisfy $\tr[h^{-2}]<\ii$. Let $w:\Omega \to \R$ such that $0\le \widehat w \in L^1(\Omega^*)$. For the interaction operator defined as in~\eqref{eq:renorm int trap}, namely 
\begin{align} \label{eq:renorm int homo-trap}
\bWren &=  \frac 12 \int  \hat{w} (k) \left|\dG (e ^{ik \cdot x}) - \left\langle \dG (e ^{ik \cdot x}) \right\rangle_{0} \right| ^2 dk, 
\end{align}
we have
\begin{equation} 
\label{eq:W-G0} \lambda^2\langle \bWren\rangle_{0} \le  C \tr ( h ^{-2} ).
\end{equation} 
where we recall that $\Gamma_0=e^{-\lambda\dG(h)}/\tr[e^{-\lambda\dG(h)}]$. 
\end{lemma}

\begin{proof} Let us write 
\begin{align*}
\bWren &=\frac12 \int  \hat{w} (k)  \left|\dG (\cos (k \cdot x)) - \left\langle \dG (\cos (k \cdot x)) \right\rangle_0\right| ^2 dk \nonumber \\
&+  \frac12 \int  \hat{w} (k)\left|\dG (\sin (k \cdot x)) - \left\langle \dG (\sin (k \cdot x)) \right\rangle_0\right| ^2 dk.
\end{align*}
Next, we take the expectation against $\Gamma_0$ and use ~\eqref{eq:variance free gibbs} with $A=\cos (k \cdot x)$ or $\sin (k \cdot x)$. Since  $\|A\|\le 1$,~\eqref{eq:W-G0} follows immediately: 
\begin{align*}
\lambda^2 \langle \bWren\rangle_0 &\le \int \hat{w} (k) \tr[h^{-2}] dk  + o_{\lambda} (1) \le  C  \tr ( h ^{-2} ).
\end{align*}
\end{proof}

Finally, we have the following bounds on the fluctuations of the particle number operator $\cN$ around its mean value $\langle \cN \rangle_0$. 

\begin{lemma}[\textbf{Fluctuations of the particle number}]\label{lem:number fluctuation}\mbox{}\\ 
Let $h>0$ satisfy $\tr(h^{-2})<\ii$. Then we have
\begin{align} 
 \langle (\cN-\langle \cN \rangle_0 )^2\rangle_{0} &\le C \tr[h^{-2}] \lambda^{-2}, \label{eq:N2-N20-free} \\  
 \langle  (\cN-\langle \cN \rangle_0 )^4\rangle_{0} &\le C \tr[h^{-2}]^3 \lambda^{-6}.  \label{eq:N4-N40-free}
\end{align}
\end{lemma}

\begin{proof} The first bound \eqref{eq:N2-N20-free}  is a consequence of \eqref{eq:variance free gibbs} (with $A=1$). To prove the second bound \eqref{eq:N4-N40-free}, let $(u_i)_{i\ge 1}$ be the orthonormal basis of $\gH$ consisting of eigenfunctions of $h$ and denote $a_i ^*=a^*(u_i),a_i=a(u_i)$ the associated creation and annihilation operators.  By Wick's theorem we can write
\begin{align*}
\langle \cN(\cN-1)(\cN-2)(\cN-3)\rangle_0 &= \sum_{n_1,n_2,n_3,n_4\ge 1} \langle a_{n_1}^* a_{n_2}^* a_{n_3}^* a_{n_4}^*a_{n_1} a_{n_2} a_{n_3} a_{n_4} \rangle_0\\
&=\sum_{n_1,n_2,n_3,n_4\ge 1} \sum_{\sigma \in S_4}\langle a_{n_1}^* a_{n_{\sigma(1)}} \rangle_0 \langle a_{n_2}^* a_{n_{\sigma(2)}} \rangle_0 \langle a_{n_3}^* a_{n_{\sigma(3)}} \rangle_0 \langle a_{n_4}^* a_{n_{\sigma(4)}} \rangle_0
\end{align*}
where the sum is taken over all permutations $\sigma$ of $\{1,2,3,4\}$. Given the choice of $(u_i)$ and the explicit formula of $\Gamma_0^{(1)}$ in \eqref{eq:DM_quasi_free}, for any $n_1,n_2,n_3,n_4\ge 1$,  if the product 
$$\langle a_{n_1}^* a_{n_{\sigma(1)}} \rangle_0 \langle a_{n_2}^* a_{n_{\sigma(2)}} \rangle_0 \langle a_{n_3}^* a_{n_{\sigma(3)}} \rangle_0 \langle a_{n_4}^* a_{n_{\sigma(4)}} \rangle_0$$ 
is  nonzero, then $n_{\sigma(i)}=n_i$ for all $i\in \{1,2,3,4\}$. The latter condition implies that either $\sigma$ is the identity (i.e. $\sigma(i)=i$ for all $i$) or $n_i=n_j$ for some $i\ne j$. Therefore, we have the upper bound 
\begin{align*}
\langle \cN(\cN-1)(\cN-2)(\cN-3)\rangle_0 &\le \sum_{n_1,n_2,n_3,n_4\ge 1} \langle a_{n_1}^* a_{n_1} \rangle_0 \langle a_{n_2}^* a_{n_2} \rangle_0 \langle a_{n_3}^* a_{n_3} \rangle_0 \langle a_{n_4}^* a_{n_4} \rangle_0\\
& \quad + 6 \sum_{n_1,n_2,n_3\ge 1} \sum_{\sigma \in S_2}\langle a_{n_1}^* a_{n_{\sigma(1)}} \rangle_0 \langle a_{n_2}^* a_{n_{\sigma(2)}} \rangle_0 |\langle a_{n_3}^* a_{n_3} \rangle_0|^2 \\
&= \langle \cN \rangle_0^4 + 6 \langle \cN(\cN-1)\rangle_0 \Tr((\Gamma_0^{(1)})^2)\\
&\le  \langle \cN \rangle_0^4 + C \tr[h^{-2}] \lambda^{-6}. 
\end{align*}
In the last estimate, we have used \eqref{eq:cNk-G0} and the bound \eqref{eq:DM_quasi_free}. Using again \eqref{eq:cNk-G0} to bound $\langle \cN^k \rangle_0$ with $k=1,2,3$ we conclude that
\begin{align} \label{eq:N4-N40-bare}
\langle \cN^4 \rangle_0  - \langle \cN \rangle_0^4 \le C \tr[h^{-2}] \lambda^{-6}. 
\end{align}
Finally, using the convexity of $x\mapsto x^3$ when $x\ge 0$ we have $\langle \cN^3 \rangle_0 \ge \langle \cN \rangle_0^3$, and hence
\begin{align*}
\langle (\cN^4 -\langle \cN \rangle_0)^4\rangle_0 &=  \langle \cN^4 \rangle_0 - 4 \langle \cN^3 \rangle_0 \langle \cN \rangle_0 + 6 \langle \cN^2 \rangle_0 \langle \cN\rangle_0^2 - 3 \langle \cN \rangle_0^4  \\
&\le \langle \cN^4 \rangle_0 - \langle \cN \rangle_0^4 + 6 \langle (\cN -\langle \cN \rangle_0)^2 \rangle_0 \langle \cN\rangle_0^2.
\end{align*}
Therefore, \eqref{eq:N4-N40-free} follows from \eqref{eq:N4-N40-bare}, \eqref{eq:N2-N20-free} and \eqref{eq:cNk-G0}. 
\end{proof}

\subsection{Interacting Gibbs state: first bounds} 

In this section let us consider the interacting Hamiltonian
\begin{align} \label{eq:def-bH-W}
\bH_\lambda= \dGamma(h)+ \lambda \bWren
\end{align}
with the interaction $\bWren$ defined in \eqref{eq:renorm int homo-trap}. Recall that the interacting Gibbs state
\begin{equation} \label{eq:int-Gibbs-W}
\boxed{\Gamma_\lambda := \frac{e^{-\lambda\bH_\lambda}}{\cZ(\lambda)},  \qquad \cZ(\lambda) = \tr_\gF\left[e^{-\lambda\bH_\lambda}\right],}
\end{equation}
is the unique minimizer for the variational problem \eqref{eq:rel-energy}:
$$
-\log \frac{\cZ(\lambda)}{\cZ_0(\lambda)} = \min_{\substack{\Gamma\geq 0\\ \tr_{\gF}\Gamma=1}} \left\{ \cH(\Gamma,\Gamma_{0}) + \lambda^2 \Tr(\bWren \Gamma_{\lambda}) \right\}
$$  
where $\cZ_0(\lambda)$ is given by~\eqref{eq:reminder_free}. We also recall the notation $\bH_0:=\dG(h)$. 
We can first control the relative free energy, or equivalently the ratio of the free and interacting partition functions:

\begin{lemma}[\textbf{Bound on relative partition function}]\label{lem:partition}\mbox{}\\
Let $h>0$ satisfy $\tr[h^{-2}]<\ii$. Let $w:\Omega \to \R$ such that $0\le \widehat w \in L^1(\Omega^*)$. Then
\begin{equation}\label{eq:part bound}
0 \le -\log \frac{\cZ(\lambda)}{\cZ_{0}(\lambda)} \le C \tr ( h ^{-2} ).
\end{equation}
In particular, we deduce that 
\begin{equation}
\cH(\Gamma_\lambda,\Gamma_{0})\leq C\tr ( h ^{-2} )
\label{eq:apriori_entropy}
\end{equation}
and
\begin{equation}
 \Tr[\bWren \Gamma_{\lambda}]\leq C\lambda^{-2}\tr ( h ^{-2} ),
 \label{eq:apriori_W}
\end{equation}
uniformly in $\lambda$.
\end{lemma}

\begin{proof} 
For the upper bound in~\eqref{eq:part bound} we take  the trial state $\Gamma=\Gamma_0$ in \eqref{eq:rel-energy} and use \eqref{eq:W-G0}. Then~\eqref{eq:apriori_entropy} and~\eqref{eq:apriori_W} follow immediately, and since both these quantities are positive, we also get the lower bound in~\eqref{eq:part bound}.
\end{proof}

Using Lemma~\ref{lem:partition} and a simple monotonicity argument we can control the expectations of some specific observables against the interacting Gibbs state $\Gamma_\lambda$ by those against the Gaussian quantum state. 

\begin{lemma}[\textbf{Moments of particle number and kinetic operators}]\label{lem:number interacting}\mbox{}\\ 
Let $h>0$ satisfy $\tr[h^{-2}]<\ii$. Let $w:\Omega \to \R$ such that $0\le \widehat w \in L^1(\Omega^*)$. Then 
\begin{align} 
\langle \cN^k \rangle_\lambda \le  \,e^{C \tr[h^{-2}]}\langle\cN^k\rangle_0&\leq C_k \,e^{C \tr[h^{-2}]}\lambda^{-2k}, \quad \forall k\ge 1,  \label{eq:Nk} \\
\langle (\cN-\langle \cN \rangle_0 )^2\rangle_{\lambda} &\le C e^{C \tr[h^{-2}]} \lambda^{-2}, \label{eq:N2-N20} \\  
\langle (\cN-\langle \cN \rangle_0 )^4\rangle_{\lambda} &\le C e^{C \tr[h^{-2}]} \lambda^{-6},  \label{eq:N4-N40}\\
\langle (\dGamma(h) )^2\rangle_{\lambda} &\le C  e^{C \tr[h^{-2}]} \lambda^{-6}, \label{eq:dGh2}\\
\Tr(h^{\alpha} \Gamma_\lambda^{(1)}) &\le C e^{C \tr[h^{-2}]} \lambda^{-2-\alpha}, \quad \forall \alpha \in [0,1]. \label{eq:dGhalpha}
\end{align}
Here $\langle \cdot \rangle_\lambda$ is the expectation against $\Gamma_\lambda$. 
\end{lemma}

Note that from~\eqref{eq:cNk-G0-limit} we actually obtain
\begin{equation}
 \lim_{\lambda\to0^+}\lambda^{2k}\langle \cN^k \rangle_\lambda=0.
 \label{eq:Nk_limit_lambda}
\end{equation}

\begin{proof} We will use the elementary fact (see e.g.~\cite[Section 2.2]{Carlen-10}) that for an increasing function $f:\R \to \R$ and self-adjoint operators $A,B$
$$ A\leq B \Rightarrow \tr \left[f (A) \right] \leq \tr \left[ f (B) \right]. $$

In particular, using $\bH_\lambda\ge \bH_0$ and the fact that $\cN$ commutes with both $\bH_0$ and $\bH_\lambda$, we have
$$
\Tr\left[\cN^k e^{-\lambda\bH_\lambda}\right] = \Tr\left[e^{k\log \cN -\lambda\bH_\lambda}\right] \leq \Tr\left[e^{k\log \cN -\lambda\bH_0}\right] = \Tr\left[\cN^k e^{-\lambda\bH_0}\right].
$$
Therefore, 
\begin{equation}\label{eq:number free int}
 \langle \cN^k \rangle_\lambda \le \frac{\cZ_0(\lambda)}{\cZ(\lambda)} \langle \cN^k \rangle_0
\end{equation}
and hence \eqref{eq:Nk} follows from~\eqref{eq:part bound} and ~\eqref{eq:cNk-G0}. Similarly, from \eqref{eq:part bound}, \eqref{eq:N2-N20-free} and \eqref{eq:N4-N40-free} we have
$$
 \langle ( \cN -\langle \cN\rangle_0)^2 \rangle_\lambda \le \frac{\cZ_0(\lambda)}{\cZ(\lambda)} \langle ( \cN -\langle \cN\rangle_0)^2 \rangle_0 \le C e^{C \tr[h^{-2}]}  \tr[h^{-2}] \lambda^{-2} 
 $$
 and 
 $$
 \langle ( \cN -\langle \cN\rangle_0)^4 \rangle_\lambda \le \frac{\cZ_0(\lambda)}{\cZ(\lambda)} \langle ( \cN -\langle \cN\rangle_0)^4 \rangle_0 \le C e^{C \tr[h^{-2}]}  \tr[h^{-2}]^3 \lambda^{-6}.$$
 Thus the bounds \eqref{eq:N2-N20} and \eqref{eq:N4-N40} follow. 

To prove the kinetic energy estimate \eqref{eq:dGh2} we use 
$$
\Tr[f(\lambda\bH_\lambda)] \le \Tr[f(\lambda\bH_0)]
$$
where $f(x)=(x^2+2x+2)e^{-x}$ is decreasing (as $f'(x)=-x^2 e^{-x} \le 0$). Combining with \eqref{eq:part bound}, this leads to 
\begin{align} \label{eq:second-H-H0}
\left\langle \lambda^2(\bH_\lambda)^2 + 2 \lambda\bH_\lambda +2 \right\rangle_\lambda \le  e^{C \tr[h^{-2}]} \left\langle \lambda^2(\bH_0)^2 + 2 \lambda\bH_0 +2 \right\rangle_0. 
\end{align}
The right side of \eqref{eq:second-H-H0} can be bounded easily using \eqref{eq:DM_quasi_free}:
\begin{align*}
\left\langle \lambda^2(\bH_0)^2  \right\rangle_0 &= \lambda^2\Tr[h^2 \Gamma_0^{(1)}] +  \lambda^2\Tr[h \otimes h \Gamma_0^{(2)}]\\
&\le 2\lambda^2 \Tr\left[ \frac{h^2}{e^{\lambda h}-1} \right] +   2\lambda^2\left(\Tr\left[ \frac{h}{e^{\lambda h}-1} \right] \right)^2\\
&\le 2 \lambda^2  \Tr\left[ \frac{h^2}{(\lambda h)^{p+2}} \right] +  2 \lambda^2  \left(\Tr\left[ \frac{h}{(\lambda h)^{p+1}} \right] \right)^2\\
&\le C\Tr[h^{-2}] \Big( 1+ \Tr[h^{-2}] \Big) \lambda^{-4}. 
\end{align*}
The left side of \eqref{eq:second-H-H0} can be estimated using the Cauchy-Schwarz inequality
\begin{align*}
\lambda^2(\bH_\lambda)^2 = \lambda^2\left(\bH_0 + \lambda \bW^{\rm ren}  \right)^2 \ge \frac{\lambda^2}{2} (\bH_0)^2 - C\lambda^4(\bW^{\rm ren})^2 \ge \frac{\lambda^2}{2} (\bH_0)^2 -  C\lambda^4 (\cN^2 + \lambda^{-2} \cN)^2
\end{align*}
and \eqref{eq:Nk}. Thus \eqref{eq:second-H-H0} leads to
\begin{align*}
\lambda^2\left\langle ( \bH_0)^2 \right\rangle_\lambda &\le C e^{C \tr[h^{-2}]} \left\langle \lambda^2(\bH_0)^2 + 2 \lambda\bH_0 +2 \right\rangle_0 +  C\lambda^{4} \left\langle (\cN^2 + \lambda^{-2} \cN)^2  \right\rangle_\lambda \\
&\le C e^{C \tr[h^{-2}]}  \lambda^{-4} 
\end{align*}
and \eqref{eq:dGh2} holds true. By the Cauchy-Schwarz inequality, we deduce from \eqref{eq:dGh2}  that 
$$
\langle \dG(h) \rangle_\lambda \le  C e^{C \tr[h^{-2}]} \lambda^{-3}
$$
which is equivalent to \eqref{eq:dGhalpha} with $\alpha=1$. Moreover, \eqref{eq:Nk} gives \eqref{eq:dGhalpha} with $\alpha=0$. The full bound \eqref{eq:dGhalpha} with $\alpha\in [0,1]$ then follows by interpolation. 
\end{proof}

\section{From relative entropy to density matrices} \label{sec:relative entropy bound}

In this section we prove a new inequality of general interest on the entropy relative  to Gaussian (quasi-free) quantum states. As a consequence, this gives a strong bound for  the relative one-particle density matrix of the interacting Gibbs state, which is the seed for the analysis  in Section~\ref{sec:correl}.

Let $\mathfrak{K}$ be an arbitrary separable Hilbert space. We consider the relative entropy
$$\cH(\Gamma,\Gamma_0)=\tr_{\gF(\mathfrak{K})} \Big( \Gamma(\log\Gamma- \log \Gamma_0) \Big)$$
of two states over the Fock space $\gF(\mathfrak{K})$, under the assumption that 
$$\Gamma_0=\frac{e^{-\lambda\dG(h)}}{\tr_\gF[e^{-\lambda\dG(h)}]}$$
is a Gaussian (quasi-free) state. 
It is well known that $\cH(\Gamma,\Gamma_0)$ vanishes if and only if $\Gamma=\Gamma_0$. Indeed Pinsker's inequality (see~\cite{CarLie-14} and~\cite[Section~5.4]{Hayashi-06})  states that 
\begin{equation}
 \cH(\Gamma,\Gamma_0)\geq \frac12 \Big(\tr_{\gF(\mathfrak{K})}|\Gamma-\Gamma_0|\Big)^2.
\label{eq:Pinsker}
 \end{equation}
Our goal here is to deduce a  bound on the difference $\Gamma^{(1)}-\Gamma^{(1)}_0$ of the one-particle density matrices, instead of the difference $\Gamma-\Gamma_0$ of the states in Fock space. Such a bound does not follow from~\eqref{eq:Pinsker}. The main result of this section is

\begin{theorem}[\textbf{From relative entropy to reduced density matrices}]\label{thm:estim_relative_entropy}\mbox{}\\
Let $h>0$ be a positive self-adjoint operator on a separable Hilbert space $\mathfrak{K}$ such that $\tr[e^{-h}]<\ii$. Consider the associated Gaussian (quasi-free) quantum state 
$$\Gamma_0=\frac{e^{-\dG(h)}}{\tr_{\gF(\mathfrak{K})}[e^{-\dG(h)}]}$$
on the Fock space~$\gF(\mathfrak{K})$. Then for any other state $\Gamma$ on $\gF(\mathfrak{K})$, we have 
\begin{equation}
\tr_{\mathfrak{K}}\left|\sqrt{h}\big(\Gamma^{(1)}-\Gamma^{(1)}_0\big)\sqrt{h}\right|^2\leq 4\,\cH(\Gamma,\Gamma_0)\left(\sqrt2 +\sqrt{\cH(\Gamma,\Gamma_0)}\right)^2
 \label{eq:estim_HS_relative_entropy}
\end{equation}
where $\Gamma^{(1)}$ and $\Gamma^{(1)}_0$ are the associated one-body density matrices.
\end{theorem}

The constants in the above inequality are not optimal and are displayed only for concreteness. The bound~\eqref{eq:estim_HS_relative_entropy} is one of our most crucial estimates, exploiting the fact that quantities calculated relative to the free state are much better behaved than bare ones. In particular, note that when $\tr(h^{-2})<\ii$, the left side of \eqref{eq:estim_HS_relative_entropy} dominates the trace norm of the relative one-body density matrix: 
$$
\tr_{\mathfrak{K}}\left|\Gamma^{(1)}-\Gamma^{(1)}_0\right|\leq \sqrt{\tr_{\mathfrak{K}}[h^{-2}]} \norm{ h^{1/2}(\Gamma^{(1)}-\Gamma^{(1)}_0)h^{1/2}}_{\gS^2}.  
$$
More precisely, if $\tr(h^{-p})<\ii$ for some $1\leq p\leq2$, we have (see~\cite{LewNamRou-18_2D})
\begin{equation}
\norm{h^\alpha \big(\Gamma^{(1)}-\Gamma^{(1)}_0\big)h^\alpha}_{\gS^1}\leq 2\sqrt2 \sqrt{\tr(h^{-2+4\alpha})}\sqrt{\cH(\Gamma,\Gamma_0)}+2\|h^{-1}\|^{1-2\alpha}\cH(\Gamma,\Gamma_0)
 \label{eq:estim_trace_relative_entropy}
\end{equation}
for all 
$$0\leq \alpha\leq \frac{2-p}{4}.$$

\begin{remark}[Bosonic relative entropy of reduced density matrices]\label{rem:entropy quasi free}\mbox{}\\ 
If $\Gamma$ is also a quasi-free state, then the relative entropy of $\Gamma$ and $\Gamma_0$ equals $\cH(\Gamma,\Gamma_0)=\cH_{\text{B-E}}(\Gamma^{(1)},\Gamma_0^{(1)})$ where
\begin{equation}
\cH_{\text{B-E}}(\gamma,\gamma_0):=\tr_{\mathfrak{K}}\Big(\gamma\big(\log\gamma-\log\gamma_0\big)-(1+\gamma)\big(\log(1+\gamma)-\log(1+\gamma_0)\big)\Big).
\label{eq:H_B-E}
\end{equation}
In particular, Theorem~\ref{thm:estim_relative_entropy} provides a lower bound on this quantity. Related estimates have been derived in ~\cite[Lemma~4.1]{DeuSeiYng-19} and \cite[Lemma~4.1]{DeuSei-19_ppt}; the difference here is that we are able to include the operator $h$ explicitly in Theorem ~\ref{thm:estim_relative_entropy}. \hfill$\diamond$
\end{remark}

The proof of Theorem~\ref{thm:estim_relative_entropy} is a Feynman-Hellmann-type argument, i.e. a perturbation of the variational principle defining $\Gamma_0$. We shall use the explicit expression of the one-particle density matrix of $\Gamma_0$ (c.f.   \eqref{eq:DM_quasi_free})
$$\Gamma_0^{(1)}=\frac{1}{e^h-1}.$$
The following lemma, a consequence of Klein's inequality, allows us to estimate the effect of the perturbation:

\begin{lemma}[\textbf{Perturbed Gaussian quantum state}]\label{lem:pert free}\mbox{}\\
Let $A$ be a self-adjoint operator with $A \le c h$ for a constant $0<c < 1$. We have 
\begin{equation}\label{eq:pert free}
0\leq \tr \left( A \left( \frac{1}{e ^{h-A} - 1}- \frac{1}{e ^{h} - 1}\right) \right) \leq \frac{1}{1-c} \tr\left( \frac{1}{h} A \frac{1}{h} A \right).
\end{equation}
\end{lemma}

\begin{proof}
The function 
$$ x\mapsto \frac{1}{e^x - 1} - \frac{1}{x}$$
is increasing on $\R^+$, whereas $x\mapsto \frac{1}{e^x - 1}$ is decreasing. Thus, for all $x, y >0$ we have 
\begin{equation}\label{eq:Klein}
 0\leq (x-y) \left( \frac{1}{e^y - 1} - \frac{1}{e^x - 1}\right) \leq (x-y) \left( \frac{1}{y} - \frac{1}{x}\right). 
\end{equation}
Klein's matrix inequality~\cite[Proposition~3.16]{OhyPet-93} implies that if $f_k,g_k$ are real functions on $\R^+$ and $c_k$ real numbers, then 
$$ \sum_k c_k f_k(x) g_k (y) \geq 0 \mbox{ for all } x,y\in\R^+$$
implies that for any pair $C,D$ of positive self-adjoint operators 
$$ \tr \left(\sum_k c_k f_k (C) g_k (D) \right) \geq 0.$$
Hence it follows from~\eqref{eq:Klein} that for any positive self-adjoint operators  $C,D$
$$ 0\leq \tr \left[ (C-D) \left( \frac{1}{e ^{D} - 1}- \frac{1}{e ^{C} - 1}\right) \right] \leq \tr \left[ (C-D)  \left( \frac{1}{D} - \frac{1}{C}\right) \right].$$
Applying this with $D = (h-A)/T$ and $C = h/T$ yields 
$$ 0\leq \tr \left[ \frac{A}{T} \left( \frac{1}{e ^{(h-A)/T} - 1}- \frac{1}{e ^{h/T} - 1}\right) \right] \leq \tr \left[ A \left( \frac{1}{h-A}- \frac{1}{h}\right) \right].$$
There remains to use the resolvent expansion 
$$ \frac{1}{h-A} = \frac{1}{h} + \frac{1}{h} A \frac{1}{h-A}$$
and observe that by the assumption $A \le c h$,
$$ \tr\left[ \frac{1}{h} A \frac{1}{h-A} A \right] = \tr \left[ \frac{1}{h^{1/2}} A \frac{1}{h-A} A \frac{1}{h^{1/2}} \right] \leq \frac{1}{1-c} \tr \left[ \frac{1}{h^{1/2}} A \frac{1}{h} A \frac{1}{h^{1/2}} \right]
$$
to conclude the proof.
\end{proof}

\begin{proof}[Proof of Theorem~\ref{thm:estim_relative_entropy}]
Let $A$ be a finite rank self-adjoint operator on $\mathfrak{K}$, such that $A< h$ and let 
$$\Gamma_A=\frac{e^{-\dG(h-A)}}{\tr_{\gF(\mathfrak{K})}[e^{-\dG(h-A)}]}$$
be the associated quasi-free state, with one-particle density matrix
$$\gamma_A:=\frac{1}{e^{h-A}-1}.$$
Recall that $\Gamma_A$ minimizes the free-energy
$$ \tr \left(  \dG (h-A) \Gamma \right) - S (\Gamma)$$
with the entropy denoted by $S(\Gamma)=-\tr\Gamma\log\Gamma$. Hence, we find 
\begin{align*}
\cH(\Gamma,\Gamma_0)-\tr\big( A\Gamma^{(1)}\big)&=\tr\big(\dG(h-A)\Gamma\big)-S(\Gamma)-\tr\big(\dG(h)\Gamma_0\big)+S(\Gamma_0)\\
&\geq \tr\big(\dG(h-A)\Gamma_A\big)-S(\Gamma_A)-\tr\big(\dG(h)\Gamma_0\big)+S(\Gamma_0)\\
&\geq -\tr\big(\dG(A)\Gamma_A\big)\\
&=-\tr\left(A\frac{1}{e^{h-A}-1}\right).
\end{align*}
Therefore we have shown that 
\begin{equation}
\tr\big( A\Gamma^{(1)}\big)\leq \cH(\Gamma,\Gamma_0)+\tr\left(A\frac{1}{e^{h-A}-1}\right)
\label{eq:preliminary bound relative entropy}
\end{equation}
for any $A<h$. From this we deduce in particular that 
\begin{equation}
\tr\Big( A\big(\Gamma^{(1)}-\Gamma^{(1)}_0\big)\Big)\leq \cH(\Gamma,\Gamma_0)+\tr \left(A\left(\frac{1}{e^{h-A}-1}-\frac{1}{e^{h}-1}\right)\right).
\label{eq:first bound relative entropy}
\end{equation}
Inserting Lemma~\ref{lem:pert free} gives 
\begin{equation}
\boxed{\tr\Big( A\big(\Gamma^{(1)}-\Gamma^{(1)}_0\big)\Big)\leq \cH(\Gamma,\Gamma_0)+\frac1{1-c}\norm{\frac{1}{\sqrt{h}} A \frac{1}{\sqrt{h}}}_{\gS^2}^2.}
\label{eq:main bound relative entropy}
\end{equation}
for any $A\leq ch$ with $0<c<1$. 

Let us now take
$$A=\pm\frac{\eps}2  h^{1/2} B h^{1/2}$$
where $0<\eps\le 1$ and $B$ is a bounded finite rank self-adjoint operator with $\|B\|\leq1$ and range in $D(A)$. Our choice ensures that $A\leq h/2$. Then we obtain 
\begin{equation}
\left|\tr\Big( Bh^{1/2}\big(\Gamma^{(1)}-\Gamma^{(1)}_0\big)h^{1/2}\Big)\right|\leq 2\eps^{-1} \cH(\Gamma,\Gamma_0)+\eps \norm{B}_{\gS^2}^2.
\label{eq:bound_relative_entropy_with_B}
\end{equation}
\medskip
Optimizing over $B$ under the assumption $\|B\|_{\gS^2}\leq 1$ (which implies as required that $\|B\|\leq 1$) we find
\begin{align*}
\norm{h^{1/2}\big(\Gamma^{(1)}-\Gamma^{(1)}_0\big)h^{1/2}}_{\gS^2}&=\max_{\|B\|_{\gS^2}\leq1}\tr\Big( Bh^\alpha\big(\Gamma^{(1)}-\Gamma^{(1)}_0\big)h^\alpha\Big)\\
&\leq 2 \eps^{-1} \cH(\Gamma,\Gamma_0) +\eps
\end{align*}
for all $0< \eps\le 1$. The bound \eqref{eq:estim_HS_relative_entropy} then follows from the elementary inequality
\begin{align} \label{eq:opt-eps}
\inf_{0<\eps\le 1} \left( \eps^{-1} a + \eps b\right) \le \max\{ 2\sqrt{ab}, 2a\} \le 2(\sqrt{ab}+a), \quad \forall a,b\in [0,\infty).
\end{align}
This concludes the proof of Theorem~\ref{thm:estim_relative_entropy}.  
\end{proof}

\section{Controlling variance by first moments}\label{sec:variance-by-first-moment}

In this section we introduce a general strategy to control the quantum variance of a Gibbs state in terms of the first moments of a family of perturbed states. This will be the key ingredient to derive the correlation estimate in Section \ref{sec:correl}. Our approach is captured in the following abstract result. 

\begin{theorem}[\textbf{Controlling variance by first moments and commutators}]\label{thm:general-variance-by-linear-response}\mbox{}\\
Let $H$ be a self-adjoint operator on a separable Hilbert space $\mathfrak{K}$ such that $\tr[e^{-\beta H}]<\ii$ for any $\beta>0$. Let $A$ be a bounded self-adjoint  operator on $\mathfrak{K}$  such that 
$$X=[[H,A],A]$$ 
is bounded. Consider the perturbed Gibbs states
\begin{equation} \Gamma_\eps:=\frac{e^{-H+\eps A}}{ \tr(e^{-H+\eps A})}.\label{eq:Gamma-eps}
\end{equation}
Let $a>0$ and introduce
\begin{align} \label{eq:E}
\eta = \eta (a):= \sup_{\eps\in [-a,a]} \Big( |\Tr(A\Gamma_\eps)| + a\sqrt{|\Tr(X\Gamma_\eps)|} \sqrt{\|X\| + \|[[X,A],A]\|} \Big).
\end{align}
Then 
\begin{equation}
\Tr(A^2 \Gamma_{0}) \le  \frac{2\left(1+a^2+\eta^2\right)}{a}\eta\,e^{a\eta}
 \label{eq:estim_variance}
\end{equation}
\end{theorem}

Here the boundedness of $X=[[H,A],A]$  includes the assumption that all the commutators are properly defined, namely $AD(H)\subset D(H)$. 

The significance of Theorem \ref{thm:general-variance-by-linear-response} is that the second moment $\Tr( A^2 \Gamma_0)$ can be bounded by only the first moment $|\Tr( A\Gamma_\eps)|$ for $\eps$ in a window around the origin, plus some error terms involving the double and quadruple commutators of $H$ with $A$. We will use the bound for $A'=A-\tr(A\Gamma_0)$, with $A$ living on high kinetic energy modes. In this case  $\tr(A'\Gamma_\eps)=\tr(A\Gamma_\eps)-\tr(A\Gamma_0)$ will stay small uniformly in $\eps\in(-a,a)$ for $a$ small enough, as a consequence of Theorem~\ref{thm:estim_relative_entropy}. Assuming that the commutators are also suitably small, the bound~\eqref{eq:estim_variance} will then tell us that the variance $\Tr(A^2 \Gamma_0)-\tr(A\Gamma_0)^2$ stays small as well. 

Note that in our application we will have $a,\eta\leq1$, so that the bound~\eqref{eq:estim_variance} implies 
\begin{equation}
 \Tr(A^2 \Gamma_{0}) \le  \frac{17\eta}{a}.
 \label{eq:simplified_variance}
\end{equation}
Notice also that 
$$\eta\leq \sup_{\eps\in [-a,a]} |\Tr(A\Gamma_\eps)| + a\norm{X}\sqrt{1+\norm{A}^2}$$
so that the estimate~\eqref{eq:estim_variance} can be stated only in terms of the supremum on the right together with $\|A\|$ and $\|X\|=\|[[H,A],A]\|$, as we did in the introduction. In our case it will be important to use $\Tr(X\Gamma_\eps)$ since it will be much smaller than $\|X\|$.

\subsubsection*{\bf Case of commuting operators}
To make the idea transparent, we first explain the proof of Theorem~\ref{thm:general-variance-by-linear-response} when $A$ and $H$ commute (the argument is the same in the classical case). The function 
$$
f(\eps)= \Tr(A^2e^{-H+\eps A})
$$
is convex since
$$
f''(\eps)= \Tr(A^4 e^{-H+\eps A})\ge 0.
$$
Hence,
$$
f(0)\le \frac{1}{2a}\int_{-a}^a f(\eps) d\eps.
$$
On the other hand, we have 
$$f(\eps)=\partial_\eps \Tr(A e^{-H+\eps A})$$
and therefore we obtain
\begin{equation}
\Tr(A^2e^{-H}) \le \frac{\Tr(Ae^{-H+a A}) - \Tr(Ae^{-H-aA}) }{2a}. 
\label{eq:simple_classical} 
\end{equation}
It only remains to divide by the appropriate partition function. Its variations are controlled by
$$
\partial_\eps \log( \Tr(e^{-H+\eps A}))= \Tr(A\Gamma_\eps).
$$
Letting
$$\eta:=\sup_{\eps\in[-a,a]}|\Tr(A\Gamma_\eps)|,$$
this implies
$$e^{-a\eta} \leq\frac{ \Tr(e^{-H\pm a A})}{ \Tr(e^{-H})}\leq e^{a\eta}.$$
Dividing~\eqref{eq:simple_classical} by $\tr(e^{-H})$ and using the previous estimate, we conclude that 
$$
\frac{\Tr(A^2e^{-H})}{\tr(e^{-H})} \le \frac{\eta}{a}\, e^{a\eta},\qquad \text{when $[H,A]=0$.} 
$$

\medskip

\noindent
{\bf General case.} The challenge for proving Theorem~\ref{thm:general-variance-by-linear-response} is to handle the case when the relevant operators do not commute. Our goal is to prove that 
\begin{itemize}

\item[(i)] $\Tr(A^2e^{-H+\eps A})$ is close to $\partial_\eps \Tr(Ae^{-H+\eps A})$.

\item[(ii)] $\eps\mapsto \Tr(A^2e^{-H+\eps A})$ is ``almost convex''. 
 
 \end{itemize}
The first point has been used many times in the literature in different forms, within linear response theory \`a la Kubo~\cite{Kubo-66,KubYokNak-57,KubTodHas-91}. We will rely here on known ideas but provide different proofs.

We are not aware of any previous use of (ii). The quotation marks around ``almost convex'' are in order: we do not directly use the second derivative in $\eps$, which would lead to a Duhamel \emph{three-point} function. Instead, we compute the first derivative and prove that it is given by $\Tr(A^3e^{-H+\eps A})$ modulo errors. Then we compute the first derivative of $\Tr(A^3e^{-H+\eps A})$ and prove it is positive, modulo errors. Here ``modulo errors'' means up to the commutators that appear in the main statement and are small in applications, and up to multiples of $\Tr(A^2e^{-H+\eps A})$ itself. This rather weak form of convexity is sufficient to make an averaging argument work and bound $\Tr(A^2e^{-H + \eps A})$ pointwise by its  mean value over a small interval.

We will bound the errors in (i) and (ii) using quantitative estimates between the normal thermal expectation $\Tr(AB\Gamma_\eps)$ and the {\em Duhamel two-point function} 
$$ \int_0^1 \Tr(A\Gamma_\eps^s B \Gamma_\eps^{1-s}) ds$$
which naturally occurs when differentiating $\tr \left(A e ^{-H + B \eps}\right)$ in $\epsilon$. This link will be explained in Section \ref{sec:variances}. Then in Section \ref{sec:linear-response} we will justify the (approximate) convexity of $\eps\mapsto \Tr(A^2e^{-H+\eps A})$. The proof of Theorem \ref{thm:general-variance-by-linear-response} is concluded in Section \ref{sec:proof-thm:general-variance-by-linear-response}.

\subsection{Discrepancy of quantum variances}\label{sec:variances} Here we discuss some general properties of  quantum variances. Let $\Gamma$ be a quantum (mixed) state on a Hilbert space $\mathfrak{K}$. For a self-adjoint operator $A$ on $\mathfrak{K}$, the {\em quantum variance} is usually defined by (see ~\cite{Farina-82})
\begin{equation}
\tr\left(\big(A-\tr(A\Gamma)\big)^2\Gamma\right)=\tr\left(A^2\Gamma\right)-\big(\tr(A\Gamma)\big)^2.
\label{eq:def quantum variance}
\end{equation}
For the formula to make sense it is only required that $\sqrt\Gamma A\in\gS^2$, in which case the first term on the right side is understood as $\tr(A^2\Gamma)=\tr(\sqrt\Gamma A^2\sqrt\Gamma)=\|\sqrt\Gamma A\|_{\gS^2}^2$. For simplicity of exposition, we will most of the time assume that $A$ is bounded.

When $A$ does not commute with $\Gamma$, one might be interested in the {\em averaged quantum variance} (or the {\em canonical correlation}~\cite{KubTodHas-91})
\begin{equation}
(A,A)_\Gamma:=\int_0^1\tr\left(A\Gamma^s A\Gamma^{1-s}\right)\,ds-\left(\tr\left(A\Gamma\right)\right)^2.
\label{eq:def averaged quantum variance}
\end{equation}
The expression \eqref{eq:def averaged quantum variance} appears naturally from the study of the perturbed Gibbs states in~\eqref{eq:Gamma-eps} via Kubo's formula (c.f. Lemma  \ref{lem:dZ} below)
$$
\partial_\eps \tr(A\Gamma_\eps)=(A,A)_{\Gamma_\eps}= \int_0^1\tr\left(A\Gamma^s_\eps A\Gamma_\eps^{1-s}\right)\,ds-\left(\tr\left(A\Gamma_\eps\right)\right)^2.
$$
This formula is well-known in linear response theory, where the averaged quantum variance is interpreted as a (static) response function measuring the fluctuation/dissipation against the perturbation (see \cite{Kubo-66},  \cite[Chapter~4]{KubTodHas-91} and ~\cite[Section~2.10]{Feynman-98}).

More generally, for two self-adjoint operators $A,B$ on $\mathfrak{K}$, one might relate the normal covariance $\tr\left(A B \Gamma \right)-\Tr(A\Gamma)\Tr(B\Gamma)$ to the  {\em Duhamel two-point function} 
\begin{equation}
(A,B)_{\Gamma} := \int_0^1 \tr(A \Gamma^s B \Gamma^{1-s}) ds - \Tr(A\Gamma)\Tr(B\Gamma). 
\label{eq:def quantum covariance Duhamel}
\end{equation}
The expression \eqref{eq:def quantum covariance Duhamel} goes back to Kubo \cite{KubYokNak-57}, Bogoliubov (Jr) \cite{Bogoliubov-62,Bogoliubov-66} and has been used by many authors, including Dyson-Lieb-Simon~\cite{DysLieSim-78}. 

An important tool of our analysis is the following result, which sets bounds on the possible discrepancy of $\tr(A \Gamma^s B \Gamma^{1-s})$ at different values of $s$.

\begin{theorem}[\textbf{Discrepancy of quantum variances}]\label{thm:s-variance}\mbox{}\\
Let $\Gamma\geq0$ be a trace class operator on a separable Hilbert space $\mathfrak{K}$. Let $A,B$ be bounded self-adjoint operators on $\mathfrak{K}$. Then the following holds. 

\begin{itemize}

\item [(i)] The function $s\mapsto \Tr(A\Gamma^s A \Gamma^{1-s})$ is convex on $[0,1]$ and attains its minimum at $s=1/2$. In particular, 
\begin{equation}
 0\leq  \Tr(A \sqrt{\Gamma} A \sqrt{\Gamma}) \leq \Tr(A\Gamma^s A \Gamma^{1-s}) \leq  \Tr(A^2 \Gamma), \quad \forall s\in [0,1]. 
 \label{eq:relation_variances}
\end{equation}

\item [(ii)] If $\Gamma=e^{-H}/\tr(e^{-H})$  for a self-adjoint operator $H$ such that $A\,D(H)\subset D(H)$ and if $\left[A,H\right]\sqrt\Gamma\in\gS^2$, then
\begin{equation}
0 \le  \Tr(A^2 \Gamma) -  \Tr(A\Gamma^s A \Gamma^{1-s})  \le \frac14  \tr\Big(\Gamma\big[[A,H],A\big]\Big), \quad \forall s\in [0,1]. 
\label{eq:upper lower bound averaged variance}
\end{equation}

\item [(iii)] If the conditions in {\rm (ii)} hold for both $A$ and $B$, then for all $s\in [0,1]$ we have
\begin{align} \label{eq:var-AB}
\left| \Re \Tr(A \Gamma^s B \Gamma^{1-s}) - \Re \Tr(AB \Gamma) \right| \le \frac{1}{4} \sqrt{\Tr \Big( \Gamma [[A,H],A] \Big)}  \sqrt{\Tr \Big( \Gamma [[B,H],B] \Big)}.
\end{align}
\end{itemize}
\end{theorem}

Several similar estimates have appeared in the literature. The bound~\eqref{eq:upper lower bound averaged variance} is a refined version of Bogoliubov's inequality in \cite{Bogoliubov-62,Bogoliubov-66}, see also Roepstorff \cite{Roepstorff-76}. A famous bound of the same kind is the Falk-Bruch inequality~\cite{FalBru-69} which was later rediscovered by Dyson-Lieb-Simon in~\cite{DysLieSim-78}. Our bound~\eqref{eq:upper lower bound averaged variance}, valid for any $s$, does not seem to have been noticed before. Its average over $s$, which gives access to the discrepancy between variance and Duhamel two-point function, actually follows from~\cite[Theorem~3.1]{DysLieSim-78}. We provide a simpler proof below, by working pointwise in $s$. The covariance estimate \eqref{eq:var-AB} is an immediate consequence of \eqref{eq:upper lower bound averaged variance} by the Cauchy-Schwarz inequality and it is useful to control higher/nonlinear correlations.  

Theorem \ref{thm:s-variance} is  crucial for the sequel: it allows us to make rigorous the intuition that different versions of the quantum variances coincide in a semi-classical limit, where commutators ought to disappear.

\begin{proof} (i) We assume that $\Gamma>0$ and that $\tr(\Gamma^{1-\eps})<\ii$ for some $0<\eps<1$ throughout the proof. The general case is obtained by an approximation argument. Under these assumptions we have that $H^k \Gamma$ is trace class for all $k\geq0$, where $H:=-\log \Gamma$. For $A$ a bounded self-adjoint operator,  the function 
$$f(s):=\tr(\Gamma^sA\Gamma^{1-s} A),$$
is $C^\ii$ on $(0,1)$ and we have 
\begin{align} \label{eq:s-variance-der}
f'(s)=-\tr(\Gamma^s[H,A]\Gamma^{1-s}A),
\end{align}
$$f''(s)=-\tr(\Gamma^s[H,A]\Gamma^{1-s}[H,A])=\norm{\Gamma^{\frac{s}2}[H,A]\Gamma^{\frac{1-s}2}}_{\gS^2}^2\geq0.$$
Note that for $s\in(0,1)$, $H$ is always multiplied by some $\Gamma^t$ with $t>0$ and that $\Gamma^tH\in\gS^{1/t}$ due to our assumption that $\tr(\Gamma^{1-\eps})<\ii$. This gives a clear meaning to all the above expressions (the derivatives can blow up only at $s=0$ and $s=1$). In particular, this proves that $f$ is convex on $[0,1]$. Due to the symmetry $f(s)=f(1-s)$, we conclude that $f$ achieves its minimum at $s=1/2$. 

\smallskip

\noindent (ii) From the convexity of $f$ we have  
$$f'(0)\le f'(s)\le f'(1/2)=0, \quad \forall s \in [0,1/2].$$
Hence, we deduce  \eqref{eq:upper lower bound averaged variance} as
$$f(1/2)-f(0)=\int_0^{1/2}f'(s)\,ds\geq \frac{f'(0)}2=-\frac12\tr\big([H,A]\Gamma A\big)=-\frac14\tr\big([[A,H],A]\,\Gamma\big).$$
Here in the last identity we have used the cyclicity of the trace and the fact that $H$ commutes with $\Gamma$. Note that $f'(0)=\lim_{s\to0^+}f'(s)$ exists due to the assumption $\left[A,H\right]\sqrt\Gamma\in\gS^2$. 

\smallskip

\noindent (iii) For any $\eps>0$ by the variance estimate \eqref{eq:upper lower bound averaged variance} we can write
\begin{align*}
&\pm 2 \Re \Tr \Big( A \Gamma^s B \Gamma^{1-s} \Big) \\
&\qquad= \Tr \Big( (\eps A \pm \eps^{-1} B) \Gamma^s (\eps A \pm \eps^{-1} B) \Gamma^{1-s} \Big) - \eps^2 \Tr \Big( A \Gamma^s A \Gamma^{1-s} \Big)  - \eps^{-2}  \Tr \Big( B \Gamma^s B\Gamma^{1-s} \Big)  \\
&\qquad\le \Tr \Big( (\eps A \pm \eps^{-1} B)^2 \Gamma) \Big) -  \eps^2 \Big( \Tr(A^2 \Gamma) -\frac{1}{4} \Tr \Big( \Gamma [[A,H],A] \Big) \Big) \\
&\qquad\qquad -\eps^{-2} \Big( \Tr(B^2 \Gamma) +\frac{1}{4} \Tr \Big( \Gamma [[H,B],B] \Big) \Big) \\
&\qquad= \pm 2 \Re \Tr(AB \Gamma) + \frac{\eps^2}{4} \Tr \Big( \Gamma [[A,H],A] \Big) + \frac{\eps^{-2}}{4} \Tr \Big( \Gamma [[B,H],B] \Big).
\end{align*}
Therefore,
\begin{align*}
 2 \left| \Re \Tr \Big( A \Gamma^s B \Gamma^{1-s} \Big) - \Re \Tr(AB \Gamma)  \right| \le \frac{\eps^2}{4} \Tr \Big( \Gamma [[A,H],A] \Big) + \frac{\eps^{-2}}{4} \Tr \Big( \Gamma [[B,H],B] \Big). 
\end{align*}
Optimizing the latter bound over $\eps>0$ leads to the desired estimate \eqref{eq:var-AB}. 
\end{proof}

\subsection{Derivatives of perturbed partition functions} \label{sec:linear-response} 
In this section we prove quantitative estimates on the derivatives up to order 4 of $\eps\mapsto \tr(e^{-H+\eps A})$, using the bound~\eqref{eq:upper lower bound averaged variance}.

Let $H$ be a self-adjoint operator on a separable Hilbert space $\mathfrak{K}$ such that $\tr[e^{-sH}]<\ii$ for any $s>0$ and let $A$ be a bounded self-adjoint  operator on $\mathfrak{K}$. We consider the perturbed Gibbs states \eqref{eq:Gamma-eps}:
$$
\Gamma_\eps:=Z_\eps^{-1}{e^{-H_\eps}}, \quad Z_\eps:=\tr(e^{-H_\eps}), \quad H_\eps:= H-\eps A.  
$$

Theorem \ref{thm:s-variance} allows us to derive effective bounds for the derivatives of $Z_\eps$.  

\begin{lemma}[\bf {Approximate derivatives of the perturbed partition function}] \label{lem:dZ} \mbox{}\\
Assume that $A$ and $X=[[H,A],A]$ are bounded and let $Y=[[[[H,A],A],A],A]$. We have  
\begin{align}
\partial_\eps \Tr(e^{-H_\eps}) &= \Tr(A e^{-H_\eps}) , \label{eq:d1Z} \\
\left| \partial_\eps \Tr(Ae^{-H_\eps} ) - \Tr(A^2 e^{-H_\eps} ) \right| &\le  \frac{1}{4}  |\Tr ( X e^{-H_\eps} )| , \label{eq:d2Z} \\
\left| \partial_\eps \Tr(A^2 e^{-H_\eps}  ) - \Tr(A^3 e^{-H_\eps} ) \right| &\le  \Tr(A^2 e^{-H_\eps} ) + \frac{1}{4} \|X\|\, |\Tr (X e^{-H_\eps} ) |\nn\\
&\quad + \frac{1}{4} \sqrt{ \|Y\|\, \Tr(e^{-H_\eps}) \,|\Tr (Xe^{-H_\eps}  ) |} , \label{eq:d3Z} \\
- \partial_\eps \Tr(A^3 e^{-H_\eps} ) &\le  \Tr(A^2 e^{-H_\eps} ) + \frac{9}{64}  (\|X\|+\|Y\|) |\Tr (X e^{-H_\eps} ) | \nn\\
&\quad + \frac{1}{4}  \sqrt{ \|Y\|\, \Tr(e^{-H_\eps})\, |\Tr (X e^{-H_\eps} ) |}.
\label{eq:d4Z} 
\end{align}
\end{lemma}

The result says that $\partial_\eps \tr(A^k\Gamma_\eps)$ is close to $\tr(A^{k+1}\Gamma_\eps)$ with errors involving lower order terms. Note that in~\eqref{eq:d4Z}, $\partial_\eps \tr(A^3\Gamma_\eps)\approx\partial^2_\eps\tr(A^2\Gamma_\eps)$ is bounded from below by lower order terms only. This is the almost convexity of $\eps\mapsto \tr(A^2\Gamma_\eps)$ which we announced at the beginning of this section. 

\begin{proof} We use Duhamel's formula
$$
\partial_\eps e^{-H_\eps} = \int_0^1 e^{-sH_\eps} (-\partial_\eps H_\eps ) e^{-(1-s)H_\eps} ds = \int_0^1 e^{-sH_\eps} A e^{-(1-s)H_\eps} ds.
$$
In particular, \eqref{eq:d1Z} follows from the linearity and the cyclicity of the trace:
$$
\partial_\eps \tr(e^{-H_\eps})=\int_0^1 \tr\left(e^{-sH_\eps } A e^{-(1-s)H_\eps}\right) ds=\tr(Ae^{-H_\eps}).
$$
(Alternatively,~\eqref{eq:d1Z} also follows from the Feynman-Hellmann principle.) 

Next, using Theorem \ref{thm:s-variance} (ii)  we obtain  
\begin{align*} 
\left| \partial_\eps \tr(Ae^{-H_\eps}) - \Tr(A^2 e^{-H_\eps})\right| &= \left| \int_0^1 \tr\left(Ae^{-s H_\eps } A e^{-(1-s)H_\eps }\right) ds - \Tr(A^2 e^{-H_\eps}) \right| \nn\\
&\le \int_0^1  \left| \tr\left(Ae^{-s H_\eps } A e^{-(1-s)H_\eps }\right) - \Tr(A^2 e^{-H_\eps}) \right| ds \nn\\
&\le  \frac{1}{4} \left| \Tr ( [[H_\eps,A],A] e^{-H_\eps}) \right| = \frac{1}{4} | \Tr ( X e^{-H_\eps})|.
\end{align*}
This gives \eqref{eq:d2Z}. Now using Theorem \ref{thm:s-variance} (iii) we get 
\begin{align*}
\left| \partial_\eps \tr(A^2 e^{-H_\eps}) - \Tr(A^3 e^{-H_\eps})\right| &= \left| \int_0^1 \Re \tr\left(A e^{-s H_\eps } A^2 e^{-(1-s)H_\eps }\right) ds -  \Tr(A^3 e^{-H_\eps}) \right| \\
&\le \int_0^1  \left| \Re \tr\left(A e^{-s H_\eps } A^2 e^{-(1-s)H_\eps }\right) -  \Tr(A^3 e^{-H_\eps}) \right| ds \\
&\le  \frac{1}{4} \sqrt{ |\Tr (X e^{-H_\eps})| } \sqrt{ | \Tr ([[H,A^2],A^2]e^{-H_\eps})| }  . 
\end{align*}
Here we have used the fact that $\tr(A^k e^{-H_\eps})$ is real for all $k\in \mathbb{N}$. Since
\begin{align*}
 [[H,A^2],A^2] &= [[H,A]A + A[H,A],A^2] = A^2 X + 2 A X A + XA^2 = 4 AXA + Y
\end{align*}
we can bound
$$
|\Tr ( [[H,A^2],A^2]e^{-H_\eps})| \le 4 \|X\| \Tr(A^2 e^{-H_\eps}) +  \|Y\| \Tr(e^{-H_\eps}). 
$$
Combining with the Cauchy-Schwarz inequality we obtain \eqref{eq:d3Z}: 
\begin{align*}
&\left| \partial_\eps \tr(A^2 e^{-H_\eps}) - \Tr(A^3 e^{-H_\eps})\right| \\
&\qquad\le  \frac{1}{4} \sqrt{  |\Tr (Xe^{-H_\eps}) |} \sqrt{ 4 \|X\| \Tr(A^2 e^{-H_\eps}) +  \|Y\| \Tr(e^{-H_\eps}) }  \\
&\qquad\le \frac{1}{4} \sqrt{  |\Tr (Xe^{-H_\eps}) |} \sqrt{ 4 \|X\| \Tr(A^2 e^{-H_\eps}) } +   \frac{1}{4} \sqrt{  \Tr (Xe^{-H_\eps})}   \sqrt{\|Y\| \Tr(e^{-H_\eps}) }  \\
&\qquad\le \Tr(A^2 e^{-H_\eps}) + \frac{1}{4} \|X\| |\Tr (Xe^{-H_\eps}) |+ \frac{1}{4} \sqrt{ \|Y\| |\Tr (Xe^{-H_\eps}) |\Tr(e^{-H_\eps}) }. 
\end{align*}
Finally, by Theorem \ref{thm:s-variance} (iii)  again, 
\begin{align*}
\left| \partial_\eps \tr(A^3 e^{-H_\eps}) - \Tr(A^4 e^{-H_\eps})\right| &= \left| \int_0^1 \Re \tr\left(A e^{-s H_\eps } A^3 e^{-(1-s)H_\eps }\right) ds -  \Tr(A^4 e^{-H_\eps}) \right| \\
&\le  \frac{1}{4} \sqrt{  |\Tr (Xe^{-H_\eps}) |} \sqrt{ | \Tr ([[H,A^3],A^3]e^{-H_\eps}) | }  .
\end{align*}
From the decomposition 
\begin{align*}
[[H,A^3],A^3]  &= [[H,A]A^2 + A[H,A]A+ A^2[H,A],A^3]\\
&=X A^4+ 2 A X A^3 + 3 A^2 X A^2 + 2 A^3 X A + A^4 X \\
&= 9 A^2 X A^2 + [[X,A^2],A^2] + 4 A[[X,A],A] A \\
&= 9 A^2 X A^2 + 6 A Y A + A^2 Y + Y A^2
\end{align*}
we can bound  
\begin{align*}
&| \Tr( [[H,A^3],A^3]  e^{-H_\eps}) |\\
&\qquad \le 9 \|X\| \Tr(A^4 e^{-H_\eps}) + 6 \|Y\| \Tr(A^2 e^{-H_\eps}) + 2 \|Y\| \sqrt{ \Tr(e^{-H_\eps}) \Tr(A^4 e^{-H_\eps})} \\
&\qquad\le (9 \|X\| + \|Y\|) \Tr(A^4 e^{-H_\eps}) + 6 \|Y\| \Tr(A^2 e^{-H_\eps})  + \|Y\| \Tr(e^{-H_\eps}). 
\end{align*}
Using the Cauchy-Schwarz inequality we obtain 
\begin{align*}
&\left| \partial_\eps \tr(A^3 e^{-H_\eps}) - \Tr(A^4 e^{-H_\eps})\right| \\
&\qquad\le  \frac{1}{4}  \sqrt{ |\Tr (Xe^{-H_\eps}) |} \sqrt{ (9 \|X\|+\|Y\|) \Tr(A^4 e^{-H_\eps}) + 6 \|Y\|  \Tr(A^2 e^{-H_\eps}) + \|Y\| \Tr(e^{-H_\eps})} \\
&\qquad\le  \frac{1}{4}  \sqrt{ |\Tr (Xe^{-H_\eps}) |} \sqrt{ (9 \|X\|+\|Y\|) \Tr(A^4 e^{-H_\eps})} \\
&\qquad\qquad + \frac{1}{4}  \sqrt{ |\Tr (Xe^{-H_\eps}) |} \sqrt{ 6 \|Y\| \Tr(A^2 e^{-H_\eps})}  + \frac{1}{4}  \sqrt{ |\Tr (Xe^{-H_\eps}) |}  \sqrt{ \|Y\| \Tr(e^{-H_\eps}))} \\
&\qquad\le \Tr(A^4 e^{-H_\eps}) + \frac{1}{64}  (9\|X\|+\|Y\|) |\Tr (Xe^{-H_\eps}) |\\
&\qquad\qquad + \Tr(A^2 e^{-H_\eps}) + \frac{6}{64}  \|Y\|  |\Tr (Xe^{-H_\eps}) |+ \frac{1}{4}  \sqrt{ \|Y\| |\Tr (Xe^{-H_\eps})| \Tr(e^{-H_\eps})}. 
\end{align*}
The lower bound \eqref{eq:d4Z} then follows from the triangle inequality. 
\end{proof}

\subsection{Proof of Theorem \ref{thm:general-variance-by-linear-response}} \label{sec:proof-thm:general-variance-by-linear-response}
Let 
$$H_\eps=H-\eps A, \quad Z_\eps=\Tr(e^{-H_\eps})$$
and
$$X=[[H,A],A], \quad Y=[[[[H,A],A],A],A].$$
Note that the definition of $\eta$ in Theorem~\ref{thm:general-variance-by-linear-response} implies that 
\begin{align} \label{eq:E-app}
|\Tr(A\Gamma_\eps)| \le \eta, \quad a^2 |\Tr(X\Gamma_\eps)|^2\leq  a^2 (\|X\|+\|Y\|)|\Tr(X\Gamma_\eps)| \le \eta^2 
\end{align}
for all $\eps\in [-a,a]$. In particular, from \eqref{eq:d1Z} we have
$$
|\partial_\eps (\log Z_\eps)| = |Z_\eps^{-1} \partial_\eps Z_\eps|= |\Tr(A\Gamma_\eps)| \le \eta,   
$$
which implies that the partition function $Z_\eps$ does not vary much in $\eps$, namely
\begin{align} \label{eq:Z-vary}
e^{-a\eta} Z_0 \le Z_\eps \le e^{a\eta} Z_0,\quad \forall \eps\in [-a,a]. 
\end{align}
Thanks to \eqref{eq:E-app} and \eqref{eq:Z-vary}, we can simplify \eqref{eq:d2Z}, \eqref{eq:d3Z} and \eqref{eq:d4Z} into 
\begin{align}
\left| \partial_\eps \Tr(Ae^{-H_\eps}) - \Tr(A^2 e^{-H_\eps}) \right| &\le  \frac{Z_\eps}{4a}\eta \le \frac{Z_0 }{4a} \eta e^{a\eta}, \label{eq:direct-1}   \\
\left| \partial_\eps \Tr(A^2e^{-H_\eps}) - \Tr(A^3 e^{-H_\eps}) \right| &\le  \Tr(A^2 e^{-H_\eps}) + \frac{Z_0}{4a}\eta e^{a\eta}\left(1+\frac{\eta}{a}\right),\label{eq:direct-2}\\
- \partial_\eps \Tr(A^3e^{-H_\eps})  &\le  \Tr(A^2 e^{-H_\eps}) + \frac{Z_0}{4a}\eta e^{a\eta}\left(1+\frac{\eta}{a}\right). \label{eq:direct-3}
\end{align}
Now we use these bounds. First, integrating \eqref{eq:direct-1} we get 
\begin{align}\label{eq:direct-4}
F:=\int_{-a}^a \Tr(A^2 e^{-H_\eps}) d\eps \le \int_{-a}^a \left( \partial_\eps \Tr(Ae^{-H_\eps}) + \frac{Z_0 }{4a} \eta e^{a\eta}\right) \le  \frac{5Z_0}{2} \eta e^{a\eta},  
\end{align}
which is an averaged version of~\eqref{eq:estim_variance}. We next get rid of the averaging by using the estimates on derivatives. Using \eqref{eq:direct-2} and \eqref{eq:direct-3} we find that 
\begin{align*}
& 2a\Tr(A^2 e^{-H}) = \int_{-a}^0 \int_{\eps}^0 \partial_s \Tr(A^2 e^{-H_s}) ds d\eps -  \int_0^a \int_0^\eps \partial_s \Tr(A^2 e^{-H_s}) ds d\eps +F \\
&\le \int_{-a}^0 \int_{\eps}^0 \Big( \Tr(A^3 e^{-H_s}) + \Tr(A^2 e^{-H_s}) + \frac{Z_0}{4a}\eta e^{a\eta}\left(1+\frac{\eta}{a}\right) \Big) ds d\eps \\
&\quad -  \int_0^a \int_0^\eps \Big( \Tr(A^3 e^{-H_s}) - \Tr(A^2 e^{-H_s})-  \frac{Z_0}{4a}\eta e^{a\eta}\left(1+\frac{\eta}{a}\right) \Big) ds d\eps +F\\
&\le \int_{-a}^0 \int_{\eps}^0  \Tr(A^3 e^{-H_s}) ds d\eps - \int_0^a \int_0^\eps \Tr(A^3 e^{-H_s}) ds d\eps + (a+1)F + \frac{Z_0}{4}\eta e^{a\eta}\left(a+\eta\right) \\
&= - \int_{-a}^0 \int_{\eps}^0 \int_s^0  \partial_r \Tr(A^3 e^{-H_r}) dr  ds d\eps - \int_0^a \int_0^\eps \int_0^s \partial_r  \Tr(A^3 e^{-H_r}) dr ds d\eps \\
&\quad  + (a+1)F + \frac{Z_0}{4}\eta e^{a\eta}\left(a+\eta\right) \\
&\le  \int_{-a}^0 \int_{\eps}^0 \int_s^0 \Big(\Tr(A^2 e^{-H_r}) + \frac{Z_0}{4a}\eta e^{a\eta}\left(1+\frac{\eta}{a}\right) \Big) dr  ds d\eps \\
&\quad + \int_0^a \int_0^\eps \int_0^s \Big(\Tr(A^2 e^{-H_r}) + \frac{Z_0}{4a}\eta e^{a\eta}\left(1+\frac{\eta}{a}\right)\Big) dr ds d\eps  + (a+1)F + \frac{Z_0}{4}\eta e^{a\eta}\left(a+\eta\right)\\
&\le \Big(\frac{a^2}{2}+a+1\Big)F+ \frac{Z_0}{4}\eta e^{a\eta}\left(a+\eta\right)\left(1+\frac{a}{3}\right) .
\end{align*} 
Inserting \eqref{eq:direct-4} in the last estimate, we obtain 
\begin{align*}
2a\Tr(A^2 e^{-H}) &\le \left(\frac52\Big(\frac{a^2}{2}+a+1\Big)+\frac14\left(a+\eta\right)\left(1+\frac{a}{3}\right)\right)Z_0\eta e^{a\eta}\\
&=\left(\frac52+\frac{11}{4}a+\frac43 a^2+\frac\eta4\left(1+\frac{a}{3}\right)\right)Z_0\eta e^{a\eta}\\
&\leq 4\left(1+a^2+\eta^2\right)Z_0\eta e^{a\eta}.
\end{align*}
We used the last bound for aesthetic reasons, it is far from optimal. This concludes the proof of Theorem~\ref{thm:general-variance-by-linear-response}.
\qed

\section{Correlation estimate for high momenta}\label{sec:correl}

Now we are ready to prove the key estimate allowing us to control errors when localizing to low-momentum modes. The main result of this section is 

\begin{theorem}[\textbf{Correlation estimate for high momenta}]\label{thm:correlation}\mbox{}\\
Let $h>0$ on $L^2(\R^d)$ satisfy \eqref{eq:Schatten h-strong-1}-\eqref{eq:Schatten h-strong-2} and let  $w:\R^d \to \R$ satisfy \eqref{eq:interaction}. Denote  
 $$e_k^+=e_k-Pe_kP$$
where $e_k$ is the multiplication operator by $\cos(k\cdot x)$ or $\sin(k\cdot x)$ and  $P= \1(h\le \Lambda_e)$ with $1 \ll \Lambda_e \le \lambda^{-2}$. Then we have
\begin{equation}
\label{eq:variance bound}
\lambda^2  \left\langle \left| \dG (e_k ^+) - \left\langle \dG (e_k ^+) \right\rangle_{0} \right| ^2 \right\rangle_\lambda
\leq  C (1+|k|^2) \left( \lambda + \|(1-P)h^{-1}\|_{\gS^2} \right)^{\frac17}
\end{equation}
for all $k\in \Omega^*$, where we recall that $\pscal{\cdot}_0$ and $\pscal{\cdot}_\lambda$ denote the expectation against the Gaussian state $\Gamma_0$ and the interacting Gibbs state $\Gamma_\lambda$, respectively.  The constant $C>0$ depends on $h$ only via $\tr[h^{-2}]$.  
\end{theorem}

We will prove \eqref{eq:variance bound} using the method described in Theorem \ref{thm:general-variance-by-linear-response}, namely we use first moment estimates of a family of perturbed Gibbs states. Actually we will derive Theorem~\ref{thm:correlation} from the following. 

\begin{lemma}[\textbf{Intermediate variance estimate}]\label{lem:variance-inter}\mbox{}\\
Let $h>0$ on $L^2(\R^d)$ satisfy \eqref{eq:Schatten h-strong-1}-\eqref{eq:Schatten h-strong-2} and let  $w:\R^d \to \R$ satisfy \eqref{eq:interaction}. Pick two orthogonal projectors commuting with $h$ such that 
$$\|Qh^{-1}\|_{\gS^2} \underset{\lambda \to 0}{\to} 0 \mbox{ and } \norm{Ph} \leq L \lambda ^{-2}$$
for some $L$ satisfying 
\begin{equation}\label{eq:new cut off}
 1\ll L \ll \left(\lambda+\|Qh^{-1}\|_{\gS^2}\right)^{-\frac16}. 
\end{equation}
Let 
$$f_k^+ = P e_k Q + Q e_k P$$ 
where $e_k$ is the multiplication operator by $\cos(k\cdot x)$ or $\sin(k\cdot x)$. 
Then 
$$
\lambda^2  \Big\langle \Big(\dG(f_k^+)-\langle \dG(f_k^+)\rangle_0\Big)^2 \Big\rangle_{\lambda} \leq C (1+|k|^2) L^{\frac32} (\lambda + \|Qh^{-1}\|_{\gS^2})^{\frac14} +\frac{C}{L^2} .
$$
The constant $C$ depends on $h$ only via $\tr[h^{-2}]$. 
\end{lemma}

We now focus on the proof of Lemma \ref{lem:variance-inter}. Throughout this section, we denote by
\begin{align} \label{eq:bA-Ge}
\bA = \lambda\Big(\dG(f_k^+) - \langle \dG(f_k^+)\rangle_0\Big),\qquad \Gamma_{\lambda,\eps}=  \frac{e^{-\lambda\bH_\lambda + \eps \bA}}{\Tr[e^{-\lambda\bH_\lambda + \eps \bA}]}. 
\end{align}
for $\eps\in [-a,a]$ with some $a$ which will always be assumed to satisfy
\begin{equation}\label{eq:condition a}
a\leq \min\left(1,\frac{\min\sigma(h)}{2}\right). 
\end{equation}
Then $h-\eps f_k^+ \ge h/2$ and hence the Gibbs state $\Gamma_{\lambda,\eps}$ is well-defined under the assumption $\tr[h^{-2}]<\infty$. We will denote by  $\langle \cdot\rangle_{\lambda,\eps}$ the expectation against the perturbed state $\Gamma_{\lambda,\eps}$.

In order to apply Theorem \ref{thm:general-variance-by-linear-response}, we will derive a first moment estimate for $\bA$ against $\Gamma_{\lambda,\eps}$ from Theorem~\ref{thm:estim_relative_entropy} and then control the commutators between $\bA$ and $\bH_\lambda$. Since Theorem~\ref{thm:general-variance-by-linear-response} requires bounded perturbations, we will also have to introduce  a cut-off on the particle number operator in Fock space
\begin{align}  \label{eq:def-cP}
\cP= \1(\cN \le L\lambda^{-2}),
\end{align}
which ensures that $\cP\bA$ and $\cP[[\bH_\lambda,\bA],\bA]$ are bounded. The parameter $L$ in \eqref{eq:def-cP} is chosen as in Lemma \ref{lem:variance-inter} for simplicity, although this might not be the optimal choice. 

\subsection{First moment estimates}

The starting point of our analysis is the following crucial input from the relative entropy estimate in   Theorem~\ref{thm:estim_relative_entropy}. 

\begin{lemma}[\textbf{First moment estimate via relative entropy}]\label{lem:first-moment}\mbox{}\\
Under the conditions of Lemma \ref{lem:variance-inter}, and with $\bA$ and $\Gamma_{\lambda,\eps}$ as in \eqref{eq:bA-Ge}, we have 
$$
\left| \langle \bA \rangle_{\lambda,\eps}  \right| \le C \| Q h^{-1}\|_{\gS^2}^{1/2}. 
$$
The constant $C$ depends on $h$ only via $\tr[h^{-2}]$.  
\end{lemma}

\begin{proof}  We use the triangle inequality 
\begin{align} \label{eq:first-moment-0}
\left| \langle \bA \rangle_{\lambda,\eps}   \right| &\le \left| \langle \bA \rangle_{\lambda,\eps} - \langle \bA \rangle_{0,\eps}  \right| + \left| \langle \bA \rangle_{0,\eps} -\langle \bA \rangle_{0,0} \right| \nn\\ 
& =  \lambda \left| \Tr\left(f_k^+ (\Gamma_{\lambda,\eps}^{(1)}-\Gamma_{0,\eps}^{(1)})\right) \right| + \lambda  \left| \tr\left(  f_k^+ (  \Gamma_{0,\eps}^{(1)} - \Gamma_{0,0}^{(1)} ) \right) \right| . 
\end{align}
Let us estimate the first term on the right side of \eqref{eq:first-moment-0}. By following the proof of   \eqref{eq:apriori_entropy} (with $h$ replaced by $h-\eps e_k$) we obtain the a-priori bound on the relative entropy  
$$
\cH(\Gamma_{\lambda,\eps}, \Gamma_{0,\eps}) \le C.    
$$
Since $\Gamma_{0,\eps}$ is a quasi-free state, we can use Theorem~\ref{thm:estim_relative_entropy} to deduce that  
\begin{equation}
\lambda \norm{h^{1/2}\big(\Gamma^{(1)}_{\lambda,\eps}-\Gamma^{(1)}_{0,\eps}\big)h^{1/2}}_{\gS^2}  \leq C. \label{eq:rel-1PDM_bounded}
\end{equation}
Consequently, 
\begin{align*}
\lambda  \left| \Tr\left(f_k^+ (\Gamma_{\lambda,\eps}^{(1)}-\Gamma_{0,\eps}^{(1)})\right) \right| &\le  \lambda\,\| h^{-1/2} f_k^+ h^{-1/2}\|_{\gS^2}  \Big\| h^{1/2}(\Gamma_\lambda^{(1)}-\Gamma_0^{(1)})h^{1/2}\Big\|_{\gS^2}\nn\\
&\le C \| h^{-1/2} f_k^+ h^{-1/2}\|_{\gS^2}.
\end{align*}
Note that, since $\norm{e_k}\leq1$ and $\|h^{-1}\|_{\gS^2} \le C$ we have
\begin{align*}
\| h^{-1/2} f_k^+ h^{-1/2}\|_{\gS^2}^2&=\tr(h^{-1}f_k^+h^{-1}f_k^+)\\
&= \Tr(h^{-1} P e_k Q h^{-1} Pe_k Q) + \Tr(h^{-1} Q e_k P h^{-1} Qe_k P) \\
&\quad +  \Tr(h^{-1} Q e_k P h^{-1} Pe_k Q) + \Tr(h^{-1} P e_k Q h^{-1} Qe_k P)\\
&\le C\norm{Qh^{-1}}_{\gS^2}.
\end{align*}
Therefore
$$\lambda  \left| \Tr\left(f_k^+ (\Gamma_{\lambda,\eps}^{(1)}-\Gamma_{0,\eps}^{(1)})\right) \right|\leq C\norm{Qh^{-1}}_{\gS^2}^{1/2}.$$
Now we turn to the second term on the right side of \eqref{eq:first-moment-0}. Since $\Gamma_{0,\eps}$ is a quasi-free state, its one-body density matrix can be computed explicitly as in \eqref{eq:DM_quasi_free}. This allows us to use Lemma ~\ref{lem:pert free} to bound
\begin{align*}
\lambda  \left| \Tr\left(f_k^+ (\Gamma_{0,\eps}^{(1)}-\Gamma_{0,0}^{(1)})\right) \right| &=  \lambda\left| \tr\left( f_k^+  \left( \frac{1}{e ^{\lambda(h-\eps f_k^+)} - 1} -\frac{1}{e ^{\lambda h} - 1} \right) \right) \right| \\
&\leq C |\eps| \tr\left( h^{-1} f_k^+  h^{-1}  f_k^+ \right)  \le C \| Qh^{-1} \|_{\gS^2}.
\end{align*}
The conclusion follows by inserting the above estimates in \eqref{eq:first-moment-0}. 
\end{proof}

Our proof requires us to introduce the particle number cut-off $\cP$ in \eqref{eq:def-cP}. The error due to this insertion is controlled by the following lemma. 

\begin{lemma}[\textbf{First moment estimate in truncated Fock space}]\label{lem:truncation-N}\mbox{}\\
Under the conditions of Lemma \ref{lem:variance-inter}, and with $\bA,\Gamma_{\lambda,\eps},\cP$ as in \eqref{eq:bA-Ge}-\eqref{eq:def-cP}, we have
$$
|\langle \cP \bA \rangle_{\lambda,\eps}| \le C \left(\lambda L^{-1}+\|Qh^{-1}\|_{\gS^2}^{1/2}\right). 
$$
\end{lemma}

\begin{proof} Using $\langle \cN \rangle_0\leq L_0\lambda^{-2}$ by Lemma~\ref{lem:quasi-free}, we have $|\bA|\le C\lambda(\cN+\lambda^{-2})$. For $L\geq 2L_0$, we also have 
$$\frac{\cN}{2}\1(\cN>L\lambda^{-2})\geq \frac{L}{2\lambda^2}\1(\cN>L\lambda^{-2})\geq  \frac{L}{2L_0}\1(\cN>L\lambda^{-2})\langle \cN \rangle_0\geq \1(\cN>L\lambda^{-2}) \langle \cN \rangle_0$$
hence
$$\1(\cN>L\lambda^{-2})(\cN -\langle \cN\rangle_0)\geq \frac{\cN}{2}\1(\cN>L\lambda^{-2}).$$
This gives
$$
(1-\cP) |\bA| \le C\lambda(\cN+\lambda^{-2}) \1(\cN>L\lambda^{-2}) \le  \frac{C\lambda^3(\cN -\langle \cN\rangle_0)^2}{L}.
$$
Therefore, from \eqref{eq:N2-N20} it follows that 
$
|\langle (1-\cP) \bA \rangle_{\lambda,\eps}| \le C\lambda L^{-1}.
$ 
The desired bound then follows from the triangle inequality and Lemma \ref{lem:first-moment}. 
\end{proof}

As a consequence of Lemmas  \ref{lem:first-moment} and \ref{lem:truncation-N} we have

\begin{lemma}[\textbf{Discrepancy of partition functions}]\label{lem:dis-partition}\mbox{}\\ 
Under the conditions of Lemma \ref{lem:variance-inter}, and with $\bA,\Gamma_{\lambda,\eps},\cP$ as in \eqref{eq:bA-Ge}-\eqref{eq:def-cP}, we have
$$
C^{-1} \Tr(e^{-\lambda \bH_\lambda}) \le \Tr(e^{-\lambda \bH_\lambda +\eps  \bA}) \le C \Tr(e^{-\lambda \bH_\lambda}) 
$$
and 
$$
C^{-1} \Tr(e^{-\lambda \bH_\lambda}) \le \Tr(e^{-\lambda \bH_\lambda +\eps \cP \bA}) \le C \Tr(e^{-\lambda \bH_\lambda}), 
$$
for all $\eps\in [-a,a]$ with a small constant $a>0$ (depending on $h$ only via $\tr[h^{-2}]$). 
\end{lemma}

\begin{proof} From Lemma~\ref{lem:first-moment} it follows that
$$
|\partial_\eps \log( \Tr(e^{-\lambda \bH_\lambda+\eps \bA}))| = |\langle \bA \rangle_{\lambda,\eps}| \le C,
$$
and hence
\begin{align} \label{eq:Ze-Z0}
C^{-1} \Tr(e^{-\lambda \bH_\lambda}) \le \Tr(e^{-\lambda \bH_\lambda+\eps \bA}) \le C \Tr(e^{-\lambda \bH_\lambda}).
\end{align}
From  Lemma \ref{lem:truncation-N} and \eqref{eq:Ze-Z0} we have
\begin{align*}
|\partial_\eps \Tr(e^{-\lambda \bH_\lambda+\eps \cP\bA})| &=  |\Tr(\cP \bA e^{-\lambda \bH_\lambda+\eps \cP\bA})| = |\Tr(\cP \bA e^{-\lambda \bH_\lambda+\eps \bA})|\\
& = |\langle \cP \bA \rangle_{\lambda,\eps}| \Tr(e^{-\lambda \bH_\lambda+\eps \bA}) \le C \Tr(e^{-\lambda \bH_\lambda}).
\end{align*}
Integrating the latter bound over $\eps$, we find that 
\begin{align} \label{eq:ZPe-Z0}
C^{-1} \Tr(e^{-\lambda \bH_\lambda}) \le \Tr(e^{-\lambda \bH_\lambda+\eps \cP\bA}) \le C \Tr(e^{-\lambda \bH_\lambda})
\end{align}
if $|\eps|$ is sufficiently small. 
\end{proof} 

\subsection{Commutator estimates}

In this section we control the commutator $[[\bH_\lambda,\bA],\bA]$, projected over the particle number sectors with $\cN\leq L/\lambda^2$. 

\begin{lemma}[\textbf{Commutator estimates}]\label{lem:commutator-X}\mbox{}\\
Under the conditions of Lemma \ref{lem:variance-inter}, with $\bA,\Gamma_{\lambda,\eps},\cP$ as in \eqref{eq:bA-Ge}-\eqref{eq:def-cP}, if we denote $\bX= \lambda  [[\bH_{\lambda}, \bA],\bA]$, then 
\begin{align}
\| \cP\bX \| &\le C (1+|k|^2)L^2, \label{eq:X-norm}\\
\| \cP [[\bX,\bA],\bA]\| &\le C (1+|k|^2) L^2,\label{eq:X-A-A-norm}\\
|\langle \cP \bX \rangle_{\lambda,\eps} | &\le C (1+|k|^2) L \left(\sqrt{\lambda} +  \|Qh^{-1}\|_{\gS^2}\right) . \label{eq:X-tr}   
\end{align}
\end{lemma}

\begin{proof} We will use repeatedly that 
$ \left[\dG (a),\dG (b)\right] = \dG \left( [a,b]\right)$
for any pair of one-body operators $a,b$.  

\bigskip

\noindent
{\bf Kinetic estimates.} Consider 
$$
\bX_1:=\lambda [[\dG(h),\bA],\bA]= \lambda^3 \dG([[h,f_k^+],f_k^+]). 
$$
From $f_k^+=Pe_kQ+Qe_kP$ we have $\|f_k^+\|\le 1$. Moreover, using the fact that $h$ commutes with $P$ and $Q$, together with Assumption  \eqref{eq:Schatten h-strong-2} we can bound 
\begin{align} \label{eq:ek+h}
\| [h,f_k^+] \| &\le 2 \| Q [h,e_k] P\| \le 2 \|Q [h,e_k] h^{-1/2} \|\; \|h^{1/2}P\| \le C (1+|k|^2) L^{1/2}/\lambda
\end{align}
Consequently,
$$
\| [[h,f_k^+],f_k^+] \| \le 2 \|f_k^+ \| \|[f_k^+,h]\| \le  C (1+|k|^2) L^{1/2}/\lambda
$$
and hence
\begin{align} \label{eq:X1-norm}
\|\cP \bX_1 \| = \lambda^{3} |\cP \dG([[h,f_k^+],f_k^+])| \le  C (1+|k|^2)  L^{3/2}. 
\end{align}
Similarly, we also have 
\begin{align} \label{eq:X1-A-A-norm}
\|\cP[[\bX_1,\bA],\bA]\| \le C (1+|k|^2) \|Ph\|^{1/2}  L^{3/2}. 
\end{align}
We can actually gain a factor $\lambda^2$ in \eqref{eq:X1-A-A-norm}, but this will not be needed in the following. 

Next, we derive improved bounds for the expectation against the Gibbs state $\langle \cdot \rangle_{\lambda,\eps}$.  Using Assumption \eqref{eq:Schatten h-strong-2} and the a-priori estimate \eqref{eq:dGhalpha} we have
\begin{multline*}
\left| \langle \bX_1 \rangle_{\lambda,\eps} \right| = \lambda^3 \left| \Tr \Big( [[h,f_k^+],f_k^+] \Gamma_{\lambda,\eps}^{(1)} \Big) \right| = 2\lambda^{3} \left|\Re \Tr \Big( f_k^+ [h,f_k^+] \Gamma_{\lambda,\eps}^{(1)} \Big)\right|\\
\le 2\lambda^{3} \|f_k^+\| \| [h,f_k^+] h^{-1/2}\|  \|  h^{1/2} \Gamma_{\lambda,\eps}^{(1)}\|_{\gS^1} \le C(1+|k|^2) \sqrt{\lambda}.
\end{multline*}
Moreover, using \eqref{eq:X1-norm} and \eqref{eq:N2-N20} we can argue as in the proof of Lemma \ref{lem:truncation-N}:
\begin{align*}
\left| \langle (1-\cP)\bX_1 \rangle_{\lambda,\eps} \right| &\le \|\bX_1\|  \langle (\cN-\langle \cN\rangle_0)^2 \lambda^4/L^2 \rangle_{\lambda,\eps} \\
&\le C (1+|k|^2)  \|Ph\|^{1/2} L\lambda  \times \lambda^{-2} \lambda^4/L^2 \\
&= C (1+|k|^2)  L^{-1/2}\lambda^{3}. 
\end{align*}
By the triangle inequality, we conclude that 
\begin{align}\label{eq:X1-tr}
\left| \langle \cP\bX_1 \rangle_{\lambda,\eps} \right| \le C (1+|k|^2) \sqrt{\lambda} .
\end{align}

\medskip

\noindent
{\bf Interaction estimates.}
Next, we consider the renormalized interaction in  \eqref{eq:renorm int trap}
$$
\lambda^2\bW^{\rm ren}=\sum_{e_\xi } \frac{1}{2} \int_{\Omega^*} d\xi \hat w(\xi) \bB_\xi^2, \quad \bB_\xi =\lambda\left(\dGamma(e_\xi) - \langle \dGamma(e_\xi) \rangle_0\right)
$$
where the sum is taken over $e_\xi\in \{\cos(\xi\cdot x), \sin(\xi\cdot x)\}$. Denote
\begin{align*}
\bX_2 &:= \lambda^2[[\bW^{\rm ren}],\bA],\bA] =  \frac{1}{2} \sum_{e_\xi}  \int_{\R^d} d\xi \hat w(\xi) \left[[\bB_\xi^2, \bA],\bA \right] \\
&=  \frac{1}{2} \sum_{e_\xi}  \int_{\Omega^*} d\xi \hat w(\xi) \Big( 2[ \bB_\xi, \bA]^2 + \bB_\xi [[ \bB_\xi, \bA],\bA] + [[ \bB_\xi, \bA],\bA] \bB_\xi \Big). 
\end{align*}
From the operator inequalities
\begin{align*}
|\cP \bB_\xi| &\le \lambda \cP(\cN + T^2)  \le C L /\lambda,  \\
| \cP [ \bB_\xi, \bA] |  &= \lambda^2 |\cP\dG([e_\xi,f_k^+])| \le \lambda^2 \|[e_\xi,f_k^+]\|  \cP \cN \le C L\\
| \cP [[ \bB_\xi, \bA],\bA] |  &= \lambda^{3} |\cP\dG([[e_\xi,f_k^+],f_k^+])| \le \lambda^{3} \|[[e_\xi,f_k^+],f_k^+] \|  \cP \cN \le C L \lambda 
\end{align*}
and the assumption $\hat w\in L^1(\Omega^*)$, it follows that 
\begin{align} \label{eq:X2-norm}
\| \cP \bX_2 \| \le  C L^2.
\end{align}
Similarly, we also have
\begin{align} \label{eq:X2-A-A-norm}
\| \cP [[\bX_2,\bA],\bA] \| \le  C L^2.
\end{align}

Now we derive improved bounds  for the expectation against the Gibbs state. Since $0\le \hat w\in L^1(\Omega^*)$, we have the Cauchy-Schwarz inequality 
\begin{align}
&\left|  \int_{\Omega^*} d\xi \hat w(\xi)  \left\langle \cP \Big( \bB_\xi [[ \bB_\xi, \bA],\bA] + [[ \bB_\xi, \bA],\bA] \bB_\xi \Big) \right\rangle_{\lambda,\eps} \right| \nn\\
&\qquad\qquad\le \sqrt{\int_{\Omega^*} d\xi \hat w(\xi) \left\langle \bB_\xi^2 \right\rangle_{\lambda,\eps} }  \sqrt{ \int_{\Omega^*} d\xi \hat w(\xi) \left\langle \cP|[[ \bB_\xi, \bA],\bA]|^2 \right\rangle_{\lambda,\eps} } \nn\\
&\qquad\qquad \le C \sqrt{\lambda^2 \langle \bW^{\rm ren} \rangle_{\lambda,\eps}} \sqrt{ \lambda^{3} \langle \cN \rangle_{\lambda,\eps}} \le C \sqrt{\lambda}. \label{eq:B-A-A-tr}
\end{align}
In the last estimate we have used Lemmas \ref{lem:partition} and \ref{lem:number interacting} with $\Gamma_{\lambda}$ replaced by $\Gamma_{\lambda,\eps}$. 

Next, we have
$$
\cP |[ \bB_\xi, \bA]|^2 =  \lambda^{4} \cP |\dG([e_\xi,f_k^+])|^2  \le \lambda^{4} \cP \cN \dG (|[e_\xi,f_k^+]|^2) \le L\lambda^2 \dG (|[e_\xi,f_k^+]|^2)
$$
and hence
\begin{align} \label{eq:BA2a}
\langle \cP |[ \bB_\xi, \bA]|^2 \rangle_{\lambda,\eps} \le L\lambda^2 \Tr(|[e_\xi,f_k^+]|^2 \Gamma_{\lambda,\eps}^{(1)}).
\end{align}
By \eqref{eq:rel-1PDM_bounded} we can replace $\Gamma_{\lambda,\eps}^{(1)}$ by $\Gamma_{0,\eps}^{(1)}$ with a small  error:
\begin{align} \label{eq:BA2b}
L\lambda^2 \left| \Tr\Big(  |[e_\xi,f_k^+]|^2  (\Gamma_{\lambda,\eps}^{(1)} -\Gamma_{0,\eps}^{(1)} ) \Big) \right| \le  L\lambda^2 \|[e_\xi,f_k^+]\|^2 \| \Gamma_{\lambda,\eps}^{(1)} -\Gamma_{0,\eps}^{(1)} \|_{\gS^1}  \le CL\lambda . 
\end{align}
On the other hand, using the operator inequalities
$$
\Gamma_{0,\eps}^{(1)}  = \frac{1}{e^{\lambda(h - \eps f_k^+)}-1} \le C\lambda^{-2} (h - \eps f_k^+)^{-2} \le C\lambda^{-2} h^{-2}
$$
and 
\begin{align*}
|[e_\xi, f_k^+]|^2 &= | e_\xi P e_k Q + e_\xi Q e_k P - P e_kQ e_\xi - Q e_k P e_\xi|^2 \\
&\le C (e_\xi P e_k Q e_k P e_\xi + e_\xi Q e_k P e_k Q e_\xi + P e_k Q e_\xi^2 Q e_k P + Qe_k P e_\xi^2 P e_k Q)
\end{align*}
we get 
\begin{multline} \label{eq:BA2ab}
L\lambda^2 \left| \Tr\Big(  |[e_\xi,f_k^+]|^2  \Gamma_{0,\eps}^{(1)}  \Big) \right|
\le  C L \Big( \|Q e_k P e_\xi h^{-1}\|_{\gS^2}^2 + \|P e_k Q e_\xi h^{-1}\|_{\gS^2}^2\\ + \|e_\xi Q e_k P h^{-1}\|_{\gS^2}^2 + \|e_\xi P e_k Q h^{-1}\|_{\gS^2}^2 \Big).
\end{multline}
In all the above terms we commute $h^{-1}$ until it hits a $Q$, using 
$[h^{-1},e_\xi]=-h^{-1}[h,e_\xi]h^{-1}$
and Assumption \eqref{eq:Schatten h-strong-2} which, we recall, says that $\|[h,e_\xi]h^{-1/2}\|\leq C(1+|\xi|^2)$. For instance, we have for the first term
\begin{align*}
Q e_k P e_\xi h^{-1}&=Q e_k P h^{-1}[h,e_\xi] h^{-1}+Q e_k h^{-1}P e_\xi \\
&=Q h^{-1}[h,e_k] h^{-1}P[h,e_\xi] h^{-1}+Q h^{-1}e_k P [h,e_\xi] h^{-1}\\
&\qquad +Q h^{-1}[h,e_k] h^{-1}P e_\xi+Q h^{-1}e_k P e_\xi.
\end{align*}
Using that $\|e_k\|\leq1$ and~\eqref{eq:Schatten h-strong-2}, this gives
$$\norm{Q e_k P e_\xi h^{-1}}_{\gS^2}\leq C(1+|k|^2)(1+|\xi|^2)\norm{Qh^{-1}}_{\gS^2}.$$
All the other terms on the right side of~\eqref{eq:BA2ab} are estimated similarly. In order to avoid having to square $(1+|k|^2)$ and $(1+|\xi|^2)$ which would require more assumptions on $w$, we may also bound all the terms on the right side of~\eqref{eq:BA2ab} by a constant, hence remove the squares. This gives
\begin{align} \label{eq:BA2c}
 L\lambda^2 \left| \Tr\Big(  |[e_\xi,f_k^+]|^2  \Gamma_{0,\eps}^{(1)}  \Big) \right| \le C L (1+|k|^2)(1+|\xi|^2) \|Q h^{-1}\|_{\gS^2}.
\end{align}
Putting \eqref{eq:BA2c} and \eqref{eq:BA2b} together, we  conclude from \eqref{eq:BA2a} that 
$$
\langle \cP |[ \bB_\xi, \bA]|^2 \rangle_{\lambda,\eps} \le C L (1+|k|^2)(1+|\xi|^2) (\lambda +\|Q h^{-1}\|_{\gS^2}).  
$$
Consequently, thanks to Assumption  \eqref{eq:interaction}, 
\begin{align} \label{eq:BA2-tr}
 \int_{\Omega^*} d\xi \hat w(\xi)  \left\langle \cP |[ \bB_\xi, \bA]|^2 \right\rangle_{\lambda,\eps} \le C L (1+|k|^2) (\lambda +\|Q h^{-1}\|_{\gS^2})\end{align}

From \eqref{eq:B-A-A-tr} and \eqref{eq:BA2-tr} it follows that
\begin{align} \label{eq:X2-tr}
\langle \cP \bX_2\rangle_{\lambda,\eps} \le C  (1+|k|^2) L (\sqrt{\lambda}+\|Qh^{-1}\|_{\gS^2}). 
\end{align}

\medskip

\noindent
{\bf Conclusion.} The bound \eqref{eq:X-norm} follows from \eqref{eq:X1-norm} and \eqref{eq:X2-norm}. The bound \eqref{eq:X-A-A-norm} follows from \eqref{eq:X1-A-A-norm} and \eqref{eq:X2-A-A-norm}. The bound \eqref{eq:X-tr} follows from  \eqref{eq:X1-tr} and \eqref{eq:X2-tr}.
\end{proof}

\subsection{Variance estimates} 

We are now able to to provide the

\begin{proof}[Proof of Lemma \ref{lem:variance-inter}]
We apply Theorem \ref{thm:general-variance-by-linear-response} to the case 
$$
H = \lambda \bH_\lambda, \quad A=\cP \bA, \quad \bA = \lambda\Big(\dG(f_k^+) - \langle \dG(f_k^+)\rangle_0\Big). 
$$
Note that the corresponding perturbed state 
$$
\Gamma_{\lambda,\eps,\cP} = \frac{e^{-\lambda \bH_\lambda+\eps \cP \bA}}{\Tr(e^{-\lambda \bH_\lambda+\eps \cP \bA})}
$$
is different from $\Gamma_{\lambda,\eps}$ defined before. However, thanks to Lemma \ref{lem:dis-partition} the partition functions are comparable: 
\begin{align} \label{eq:partition-com}
C^{-1} \Tr(e^{-\lambda \bH_\lambda+\eps \bA}) \le \Tr(e^{-\lambda \bH_\lambda+\eps \cP \bA}) \le C \Tr(e^{-\lambda \bH_\lambda+\eps \bA}).
\end{align}
Therefore, from Lemma \ref{lem:truncation-N} we deduce the first moment estimate 
$$
|\Tr(A \Gamma_{\lambda,\eps,\cP})| = |\Tr(\cP \bA \Gamma_{\lambda,\eps})| \frac{\Tr(e^{-\lambda \bH_\lambda+\eps \bA})}{\Tr(e^{-\lambda \bH_\lambda+\eps \cP \bA})}  \le C \left(\lambda +\|Qh^{-1}\|_{\gS^2}^{1/2}\right). 
$$
Moreover, from Lemma \ref{lem:commutator-X} we find that the commutator 
$$
X= [[H,A],A]=\lambda \cP[[\bH_\lambda,\bA],\bA]
$$
satisfies
$$
\|X\|+\|[[X,A],A]\| \le C(1+|k|^2) L^2
$$
and 
$$
|\Tr(X\Gamma_{\lambda,\eps,\cP})| = |\Tr(X \Gamma_{\lambda,\eps})| \frac{\Tr(e^{-\lambda \bH_\lambda+\eps \bA})}{\Tr(e^{-\lambda \bH_\lambda+\eps \cP \bA})} \le C (1+|k|^2)L (\sqrt{\lambda}+ \|Qh^{-1}\|_{\gS^2}). 
$$
Thus for $a>0$ sufficiently small we have
\begin{align*}
\eta&:=\sup_{\eps\in [-a,a]} \Big( |\Tr(A \Gamma_{\lambda,\eps,\cP})| + a \sqrt{|\Tr(X\Gamma_{\lambda,\eps,\cP})|} \sqrt{\|X\|+\|[[X,A],A]\|} \Big)\nn\\
& \le C \left(\lambda +\|Qh^{-1}\|_{\gS^2}^{1/2}\right) +Ca(1+|k|^2)L^{\frac32}(\lambda+ \|Qh^{-1}\|_{\gS^2})^{\frac14}.
\end{align*}
We satisfy~\eqref{eq:condition a} by picking
$$a := (1+|k|^2)^{-1} \min\left(1,\frac{\min\sigma(h)}{2}\right)$$
and we find that $\eta\to0$ when $\lambda\to0$, uniformly in $k$, under the conditions in~\eqref{eq:new cut off}. Theorem~\ref{thm:general-variance-by-linear-response} (or, rather~\eqref{eq:simplified_variance}) then gives the variance bound 
$$
\langle \cP \bA^2\rangle_{\lambda}= \Tr(A^2 \Gamma_{\lambda,0,\cP}) \le C (1+|k|^2) L^{\frac32} (\lambda + \|Qh^{-1}\|_{\gS^2})^{\frac14} 
$$
where we used that since  both $\lambda, \|Qh^{-1}\|_{\gS^2}^{1/2} \ll 1$ we have 
$$\left(\lambda +\|Qh^{-1}\|_{\gS^2}^{1/2}\right)\leq C (\lambda+ \|Qh^{-1}\|_{\gS^2})^{\frac14}.$$
Finally,  arguing as in the proof of Lemma \ref{lem:truncation-N} we have 
\begin{align*}
\langle (1-\cP) \bA^2  \rangle_{\lambda} \le \frac{C\lambda^6}{L^2} \langle (\cN-\langle \cN\rangle_0)^4 \rangle_{\lambda,\eps} \le \frac{C}{L^2}. 
\end{align*}
Here we have used \eqref{eq:N4-N40} in the last estimate. Thus in summary,
$$
\langle \bA^2 \rangle_{\lambda}  \le C (1+|k|^2) L^{\frac32} (\lambda + \|Qh^{-1}\|_{\gS^2})^{\frac14}+\frac{C}{L^2}.
$$
as desired. 
\end{proof}

Finally, we are ready to provide the 

\begin{proof}[Proof of Theorem \ref{thm:correlation} ] Consider 
$$
e_k^+ = 1- P e_k P = f_k^++ Qe_kQ, \quad P=\1_{h\le \Lambda_e}, \quad Q=\1-P.   
$$
Take a large parameter $L$ (will be determined at the end). We have $\|Ph\| \le \Lambda_e \le L \lambda^{-2}$ and $\|Qh^{-1}\|_{\gS^2} \to 0$, because $\Lambda_e\gg 1$. Hence we may apply Lemma \ref{lem:variance-inter} to obtain
\begin{equation} \label{eq:PwQ-bootstrap-1}
\lambda^2  \Big\langle \Big(\dG(f_k^+)-\langle \dG(f_k^+)\rangle_0\Big)^2 \Big\rangle_{\lambda}  \\
\le C (1+|k|^2) L^{\frac32} (\lambda + \|Qh^{-1}\|_{\gS^2})^{\frac14} +\frac{C}{L^2}.
\end{equation}
For the term $Qe_kQ$ we decompose further
$$
Qe_kQ = P_1 e_k P_1 + P_1 e_k Q_1 + Q_1 e_k P_1 + Q_1 e_k Q_1, \quad Q_1= \1_{h> L/\lambda^2}, \quad P_1=Q-Q_1. 
$$
Since $\|P_1h^{-1}\|_{\gS^2}, \|Q_1h^{-1}\|_{\gS^2} \le \|Qh^{-1}\|_{\gS^2}$ and $\|P_1h\| \le L /\lambda^2$, we may apply Lemma~\ref{lem:variance-inter} again, first with $P=Q=P_1$, then with $P=P_1,Q=Q_1$ to obtain 
\begin{equation} \label{eq:PwQ-bootstrap-2-}
\lambda^2  \Big\langle \Big(\dG(P_1 e_k P_1)-\langle \dG(P_1 e_k P_1)\rangle_0\Big)^2 \Big\rangle_{\lambda} \le C (1+|k|^2) L^{\frac32} (\lambda + \|Qh^{-1}\|_{\gS^2})^{\frac14} +\frac{C}{L^2}
\end{equation}
and
\begin{multline} \label{eq:PwQ-bootstrap-2}
\lambda^2  \Big\langle \Big(\dG(P_1 e_k Q_1+Q_1 e_k P_1)-\langle \dG(P_1e_k Q_1 + Q_1 e_k P_1 )\rangle_0\Big)^2 \Big\rangle_{\lambda} \\
 \le C (1+|k|^2) L^{\frac32} (\lambda + \|Qh^{-1}\|_{\gS^2})^{\frac14} +\frac{C}{L^2}.
\end{multline}
On the other hand, using the operator inequality
$$|Q_1 e_k Q_1| \le Q_1 \le \frac{\lambda^2 }{L}\,h$$
and the kinetic estimate \eqref{eq:dGh2} in Lemma \ref{lem:number interacting} we deduce that 
\begin{align} \label{eq:PwQ-bootstrap-3}
\lambda^2 \langle (\dGamma(Q_1 e_k Q_1)- \langle  \dGamma(Q_1 e_k Q_1) \rangle_0)^2 \rangle_{\lambda} 
 \le \frac{C\lambda^6}{L^2} \left(  \left\langle \left(\dGamma(h)\right)^2 \right\rangle_{\lambda} + \left\langle \left(\dGamma(h)\right)^2 \right\rangle_{0} \right) 
\le \frac{C}{L^{2}}. 
\end{align}
Combining the bounds~\eqref{eq:PwQ-bootstrap-1}-\eqref{eq:PwQ-bootstrap-3} and using the Cauchy-Schwarz inequality we conclude that 
\begin{align} \label{eq:PwQ-bootstrap-4}
\lambda^2 \langle (\dGamma(e_k^+) -\langle \dG(e_k^+)\rangle_0 )^2\rangle_{\lambda} \le C (1+|k|^2)L^{\frac32} \left(\lambda + \|Qh^{-1}\|_{\gS^2}\right)^{\frac14} +\frac{C}{L^2}.
\end{align}
The desired bound \eqref{eq:variance bound} follows from the choice $L= (\lambda + \|Qh^{-1}\|_{\gS^2})^{-{1}/{14}}$, which is indeed compatible with~\eqref{eq:new cut off}.
\end{proof}

\section{Free energy lower bound}\label{sec:low bound}

We are now ready to prove the free energy lower bound announced in~\eqref{eq:ineq_to_be_proven0}. As usual in variational approaches, the lower bound on the free energy is the harder part. A matching upper bound will be obtained in the next section by a trial state argument. 

Consider $\bH_\lambda=\bH_0+\lam\bWren$. Recall that we can write the relative free-energy as an infimum:
\begin{equation*}
-\log \frac{\cZl}{\cZ_0(\lambda)} = \cH(\Gamma_\lambda,\Gamma_{0}) + \lambda^2 \Tr[\bWren \Gamma_{\lambda}]= \inf_{\substack{\Gamma\geq 0\\ \tr_{\gF}\Gamma=1}} \Big\{ \cH(\Gamma,\Gamma_{0}) + \lambda^2 \Tr[\bWren \Gamma_{\lambda}] \Big\}.
\end{equation*}
We shall relate this variational principle to its classical analogue (cf. Section~\ref{sec:renormalized-measure}) 
$$ 
-\log z = \min_{\substack{0\leq f \in L ^1 (d\mu_0)\\ \int f(u) d\mu_0 (u) = 1}} \left\{ \int \cD[u]\,f(u) d\mu_0 (u) + \int f(u) \log (f(u)) d\mu_0 (u) \right\}
$$
with the optimal $f=e^{-\cD[u]}/z$. 
This section is devoted to the proof of the following

\begin{proposition}[\textbf{Free-energy lower bound}]\label{pro:low bound}\mbox{}\\
Let $h>0$ satisfy \eqref{eq:Schatten h-strong-1}-\eqref{eq:Schatten h-strong-2} and let   $w:\R^d \to \R$ satisfy \eqref{eq:interaction}.  Let $z$ be the classical relative partition function defined in Lemma~\ref{lem:re-interaction}.
Then we have 
\begin{equation}\label{eq:low bound}
\liminf_{\lambda\to0^+} \left(-\log \frac{\cZl}{\cZ_0(\lambda)}\right)  \geq - \log z.
\end{equation}
\end{proposition}

We split the proof in two parts, occupying a subsection each. 
\begin{itemize}
 \item In Section \ref{sec:low ener} we project the energy (together with counter-terms) on low momentum modes $P=\1(h\le \Lambda_e)$ and estimate the error thus made. This is the core novelty with respect to our previous papers~\cite{LewNamRou-15,LewNamRou-18a}, where the new correlation estimate in Section~\ref{sec:correl} is crucial. The relative entropy is controlled by Theorem~\ref{thm:rel-entropy} as in our previous papers. Combining with the quantitative de Finetti Theorem~\ref{thm:quant deF} this leads to a quantitative energy lower bound in terms of the projected classical energy of a lower symbol/Husimi measure. 
 
\item In Section \ref{sec:low wrap} we compare further the lower symbol with the cylindrical projection of the Gaussian measure on $P\gH$. This allows to remove the localization in the classical problem and to conclude \eqref{eq:low bound}. 
\end{itemize}

\subsection{Localization and energy lower bound}\label{sec:low ener}

In this subsection we localize to the low kinetic energy modes. Our energy lower bound is as follows:

\begin{lemma}[\textbf{Renormalized energy lower bound}]\label{lem:local-W-P}\mbox{}\\
Let $h>0$ satisfy  \eqref{eq:Schatten h-strong-1}-\eqref{eq:Schatten h-strong-2} and let $w$ satisfy \eqref{eq:interaction}. Define
\begin{equation}\label{eq:projections bis}
P = \one(h\leq \Lambda_e),\quad Q=\1-P,  \quad 1 \ll \Lambda_e \ll \lambda^{-\frac{1}{3}}.
\end{equation}
Let $\mu_{P,0}$ be the lower symbol/Husimi function of $\Gamma_0$ associated with the projection $P$ and the scale $\eps=\lambda $ as in~\eqref{eq:Husimi}, and let $\cD_K$ be the truncated renormalized interaction from Lemma~\ref{lem:re-interaction}. Then we have
\begin{align} \label{eq:partition-lwb-0}
 -\log \frac{\cZ(\lambda)}{\cZ_0(\lambda)} \ge -\log \left( \int_{P\gH} e^{-\cD_K[u]} \; d\mu_{P,0}(u)\right) -C  \lambda  \Lambda_e ^{3} - C ( \lambda  + \|Qh^{-1}\|_{\gS^2} )^{\frac1{14}}.
\end{align}
The constant $C$ depends on $h$ only via $\tr[h^{-2}]$.  
\end{lemma}

\begin{proof} By the Berezin-Lieb inequality \eqref{eq:Berezin-Lieb}, we have
\begin{align} \label{eq:BL-G-G0}
\cH(\Gamma_\lambda,\Gamma_0) \ge \cH((\Gamma_\lambda)_P,(\Gamma_0)_P)    \ge \cHcl(\mu_{P,\lambda}, \mu_{P,0}).
\end{align}
We will prove that
\begin{equation}  \label{eq:local-W-P}
\lambda^2 \Tr [\bWren\Gamma_\lambda] \ge \frac{1}{2} \int_{P\gH} \cD_K [u] d\mu_{P,\lambda}(u) - C  \lambda  \Lambda_e ^{3} - C ( \lambda  + \|Qh^{-1}\|_{\gS^2} )^{\frac1{14}}
\end{equation}
where $\mu_{P,\lambda}$ is  the lower symbol of $\Gamma_\lambda$ associated with the projection $P$ at the scale $\eps=\lambda $ as in~\eqref{eq:Husimi}. 
Then putting \eqref{eq:BL-G-G0} and \eqref{eq:local-W-P} together, we conclude \eqref{eq:partition-lwb-0} by the classical variational principle \eqref{eq:zr-rel}: 
\begin{align*}
 -\log \frac{\cZ(\lambda)}{\cZ_0(\lambda)} &= \cH(\Gamma_\lambda,\Gamma_0) + \lambda^2 \Tr [\bWren\Gamma_\lambda] \nn\\
&\ge \cHcl(\mu_{P,\lambda}, \mu_{P,0})+\frac12\int_{P\gH} \cD_K [u] d\mu_{P,\lambda}(u) -C  \lambda  \Lambda_e ^{3} - C ( \lambda  + \|Qh^{-1}\|_{\gS^2} )^{\frac1{14}}
 \nn\\
&\ge -\log \left( \int_{P\gH} e^{-\cD_K[u]} \; d\mu_{P,0}(u)\right) -C  \lambda  \Lambda_e ^{3} - C ( \lambda  + \|Qh^{-1}\|_{\gS^2} )^{\frac1{14}}.
\end{align*}

It remains to prove \eqref{eq:local-W-P}. We will write the renormalized interaction as in~\eqref{eq:renorm int trap} and estimate each Fourier component separately.

\medskip

\noindent\textbf{Step 1: Localization.} Let $e_k$ be the multiplication operator by $\cos(k\cdot x)$ or $\sin(k\cdot x)$  and let 
$$e_k^{-}=Pe_kP = e_k- e_k^+, \quad P=\1_{h\le \Lambda_e}$$
By the Cauchy-Schwarz inequality and Theorem \ref{thm:correlation}  we have
\begin{align*}
&\lambda^2 \left\langle \big|\dG(e_k^-) - \langle \dG(e_k^-) \rangle_0\big|^2 \right\rangle_{\lambda} \\
&\le (1+\eps) \lambda^2 \left\langle \big|\dG(e_k) - \langle \dG(e_k) \rangle_0\big|^2 \right\rangle_{\lambda} + (1+\eps^{-1}) \lambda^2 \left\langle \big|\dG(e_k^+) - \langle \dG(e_k^+) \rangle_0\big|^2 \right\rangle_{\lambda} \\
&\le (1+\eps) \lambda^2 \left\langle \big|\dG(e_k) - \langle \dG(e_k) \rangle_0\big|^2 \right\rangle_{\lambda} + (1+\eps^{-1}) C (1+|k|^2) ( \lambda  + \|Qh^{-1}\|_{\gS^2} )^{1/7}
\end{align*}
for all $\eps>0$. Integrating against $\hat w(k)$ over $k\in \Omega^*$, then using \eqref{eq:interaction} and \eqref{eq:apriori_W} we get
\begin{align*}
&\lambda^2 \int_{\Omega^*} \hat w(k) \left\langle \big|\dG(e_k^-) - \langle \dG(e_k^-) \rangle_0\big|^2 \right\rangle_{\lambda} dk \\
&\le (1+\eps) \lambda^2 \left\langle \bW^{\rm ren}\right\rangle_{\lambda} + (1+\eps^{-1})  C ( \lambda  + \|Qh^{-1}\|_{\gS^2} )^{1/7} \\
&\le \lambda^2 \left\langle \bW^{\rm ren}\right\rangle_{\lambda}  + C\eps + (1+\eps^{-1}) C ( \lambda  + \|Qh^{-1}\|_{\gS^2} )^{1/7}.
\end{align*}
Optimizing over $\eps>0$ gives
\begin{align} \label{eq:part-local-W-0}
&\lambda^2 \left\langle \bW^{\rm ren}\right\rangle_{\lambda} \ge \lambda^2 \int_{\Omega^*} \hat w(k) \left\langle \big|\dG(e_k^-) - \langle \dG(e_k^-) \rangle_0\big|^2 \right\rangle_{\lambda} dk - C ( \lambda  + \|Qh^{-1}\|_{\gS^2} )^{1/14}.
\end{align}

\medskip

\noindent\textbf{Step 2: Use of the de Finetti theorem.} Now we turn to the low-momentum part of the interaction. For any self-adjoint one-body operator $A$ we have
$$ (\dG (A))^2 = 2 \bigoplus_{n\geq 2} \sum_{1\leq i<j \leq n} A_i\otimes A_j + \dG (A ^2). $$
Hence
\begin{align}\label{eq:part-local-W-1}
 \left\langle \left| \dG (e_k^-) - \left\langle \dG (e_k^-) \right\rangle_{0} \right| ^2 \right\rangle_\lambda  &=  2 \Tr \left( (e_k^-)^{\otimes 2}\Gamma_\lambda^{(2)} \right) -2 \Tr \left( e_k^- \Gamma_\lambda^{(1)} \right) \Tr \left( e_k^{-}\Gamma_0^{(1)} \right) \nn\\
 &\quad + \left(\Tr \left( e_k^-\Gamma_0^{(1)} \right) \right)^2
 + \Tr \Big((e_k^-)^2 \Gamma_\lambda^{(1)}\Big).
\end{align}
The last term $\Tr ((e_k^-)^2 \Gamma_\lambda^{(1)})\ge 0$ can be omitted for a lower bound. From the explicit formulas~\eqref{eq:DM_quasi_free} and \eqref{eq:DM free meas}, the operator bound 
$$
\left| \frac{1}{e^{\lambda h}-1} -  \frac{1}{\lambda h} \right| \le \frac{1}{2}
$$
and
\begin{equation}\label{eq:CLR} 
K:=\mathrm{dim} \left( P\gH \right) = \Tr P \le \Tr [(\Lambda_e/h)^2]\le C \Lambda_e^{2},
\end{equation}
it follows that 
\begin{align} \label{eq:ek-Gamma0}
\lambda  \Tr \left( e_k^- \Gamma_0^{(1)} \right) &= \lambda \Tr   \left( e_k^- \frac{1}{e^{\lambda h}-1} \right) \nn\\
 & = \Tr  \left(  e_k^- h^{-1} \right) +\lambda \Tr   \left( e_k^- \left(\frac{1}{e^{\lambda h}-1} -\frac{1}{\lambda h}\right) \right) \nn\\
 &= \int\left\langle u, e_k^-  u \right\rangle d\mu_0(u) + O(\lambda  \Lambda_e^2).
\end{align}
On the other hand, using $|e_k^-|\le P \le \Lambda_e h^{-1}$, \eqref{eq:rel-1PDM_bounded} and \eqref{eq:DM_quasi_free} we get
\begin{align} \label{eq:counter term bound}
\left| \lambda \Tr \left( e_k^- \Gamma_\lambda^{(1)} \right) \right| &\le \left| \lambda \Tr \left( e_k^- (\Gamma_\lambda^{(1)} -\Gamma_0^{(1)} \right) \right| +  \left| \lambda \Tr \left( e_k^- \Gamma_0^{(1)} \right) \right| \nn\\
&\le  \lambda \Tr \left| \Gamma_\lambda^{(1)} -\Gamma_0^{(1)} \right| + \Tr  \left(  P h^{-1} \right) \nn\\
&\le C \Lambda_e  \tr\left[h^{-2}\right]. 
\end{align}
Thus \eqref{eq:part-local-W-1} gives
\begin{align}\label{eq:part-local-W-1aaa}
 \lambda^2\left\langle \left| \dG (e_k^-) - \left\langle \dG (e_k^-) \right\rangle_{0} \right| ^2 \right\rangle_\lambda &\geq 
 2 \lambda^2 \Tr \left( (e_k^-)^{\otimes 2}\Gamma_\lambda^{(2)} \right) -2 \lambda  \Tr \left( e_k^- \Gamma_\lambda^{(1)} \right) \Big\langle \left\langle u, e_k^-  u \right\rangle \Big\rangle_{\mu_0} \nn\\
 & \quad +  \left(\int \left\langle u, e_k^-  u \right\rangle d\mu_0(u)\right)^2 -C  \left(\lambda  \Lambda_e ^{3} + \lambda^2 \Lambda_e ^4 \right).  
\end{align}
Note that, for $\Lambda_e \leq \lambda^{-1/3}$ the last error term $\lambda^2 \Lambda_e ^4$ is of lower order and can be absorbed in $\lambda  \Lambda_e ^{3}$. 

Next, let $\mu_{P,\lambda}$ be the lower symbol of $\Gamma_\lambda$ associated with the projection $P$ and the scale $\eps=\lambda $ as in~\eqref{eq:Husimi}. We apply~\eqref{eq:Chiribella} to obtain the density matrices of the $P$-projected state~$\Gamma_{\lambda,P}$:
\begin{align*} 
\lambda^2\Gamma_{\lambda,P}^{(2)} &= \frac{1}{2}\int_{P\gH}|u^{\otimes 2}\>\<u^{\otimes 2}|\;d\mu_{P,\lambda}(u) - 2 \lambda^2 \Gamma_{\lambda,P} ^{(1)}\otimes_s P - 2 \lambda^2 P \otimes_s P,\nonumber\\
 \lambda  \Gamma_{\lambda,P}^{(1)} &= \int_{P\gH}|u\>\<u|\;d\mu_{P,\lambda}(u) - \lambda P.
\end{align*}
Recalling~\eqref{eq:GammaV-k} we have
$$ \Gamma_{\lambda,P}^{(k)} = P^{\otimes k} \Gamma_{\lambda} ^{(k)} P^{\otimes k}$$
and using~\eqref{eq:counter term bound}, we deduce that 
\begin{align*}
\lambda^2 \Tr \left[ (e_k^-)^{\otimes 2}\Gamma_\lambda^{(2)} \right] & =\frac{1}{2}\int_{P\gH} |\langle u, e_k^-  u \rangle|^2 d\mu_{P,\lambda}(u) + O  \left(\lambda  \Lambda_e ^{3} \right),\\
\lambda  \Tr \left[ e_k^-\Gamma_\lambda^{(1)} \right] &= \int_{P\gH} \langle u, e_k^-  u \rangle d\mu_{P,\lambda}(u) + O  \left(\lambda  \Lambda_e ^2\right).
\end{align*}
Inserting the latter formulas in \eqref{eq:part-local-W-1aaa} and using \eqref{eq:counter term bound} again we obtain 
\begin{equation*}
 \lambda^2\left\langle \left| \dG (e_k^-) - \left\langle \dG (e_k^-) \right\rangle_{0} \right| ^2 \right\rangle_\lambda \geq \int_{P\gH} \left|\langle u, e_k^-  u \rangle - \left\langle \left\langle u, e_k^-  u \right\rangle \right\rangle_{\mu_0} \right|^2 d\mu_{P,\lambda}(u)  -C  \lambda  \Lambda_e ^{3}.
  \end{equation*}
Integrating the latter bound against $\widehat w(k)$ gives
\begin{equation}\label{eq:part-local-W-1aaabbb}
 \lambda^2 \int_{\Omega^*} \hat w(k) \left\langle \left| \dG (e_k^-) - \left\langle \dG (e_k^-) \right\rangle_{0} \right| ^2 \right\rangle_\lambda dk \geq \frac{1}{2} \int_{P\gH} \cD_K [u] d\mu_{P,\lambda}(u)  -C  \lambda  \Lambda_e ^{3}
  \end{equation}
  where $\cD_K$ is the truncated renormalized interaction as in Lemma~\ref{lem:re-interaction}.
  
Finally, we compare   \eqref{eq:part-local-W-1aaabbb} with \eqref{eq:part-local-W-0} to obtain \eqref{eq:local-W-P}
$$
\lambda^2 \Tr [\bWren\Gamma_\lambda] \ge \frac{1}{2} \int_{P\gH} \cD_K [u] d\mu_{P,\lambda}(u) - C  \lambda  \Lambda_e ^{3} - C ( \lambda  + \|Qh^{-1}\|_{\gS^2} )^{1/14}. 
$$
This concludes the proof of Lemma \ref{lem:local-W-P}.
\end{proof}

\subsection{Removing the localization and concluding} \label{sec:low wrap}

In order to estimate further the right side of \eqref{eq:partition-lwb-0} from below, our task is now to compare the lower symbol $\mu_{P,0}$ of $\Gamma_0$ with the cylindrical projection $\mu_{0,K}$ of the Gaussian measure $\mu_0$ on $P\gH$ in Lemma \ref{lem:free-meas}. In this direction we prove the following lemma. For its statement, recall that both measures we are interested in are absolutely continuous with respect to the Lebesgue measure on $P\gH$. 

\begin{lemma}[\textbf{Further comparisons for the projected free state}]\label{lem:free low cyl}\mbox{}\\
Let $h>0$ satisfy $\tr[h^{-2}]<\ii$. Then
\begin{equation}\label{eq:mu_L1}
\norm{\mu_{P,0}-\mu_{0,K}}_{L^1(P\gH)}\leq 2\tr[h^{-2}]\lambda  \Lambda_e ^{3}.
\end{equation}
\end{lemma}

\begin{proof}
Recall that
$$
d\mu_{0,K}(u)=\prod_{j = 1}^K  \left( \frac{\lambda_j }{\pi}e^{ - {\lambda} _j |\alpha_j |^2} \right) d\alpha_j, \quad \text{with}\quad u=\sum_{j=1}^K \alpha_j u_j.
$$
On the other hand, by Definition~\ref{def:lower symbol}, and the explicit action of Fock-space localization on quasi-free states~\cite[Example~12]{Lewin-11}
\begin{align*}
d\mu_{P,0}(u)&= (\lambda\pi)^{-K} \Big\langle \xi (u/\sqrt{\lambda}),  (\Gamma_0)_P  \xi\left(u /\sqrt{\lambda}\right) \Big\rangle du\\
& = (\lambda\pi)^{-K}  \left[ \Tr \left(e^{-\lambda \dGamma(Ph)}\right) \right]^{-1}   \Big\langle \xi (u/\sqrt{\lambda}) ,   e^{-\lambda\dGamma(Ph)}   \xi (u/\sqrt{\lambda}) \Big\rangle du. 
\end{align*}
Using the Peierls-Bogoliubov inequality $\pscal{x,e^Ax}\geq e^{\pscal{x,Ax}}$ and the coherent states' definition~\eqref{eq:coherent state}, we have  
\begin{align*}
\Big\langle \xi (u/\sqrt{\lambda}),  e^{-\lambda\dGamma(Ph)}  \xi (u/\sqrt{\lambda}) \Big\rangle  
&\ge  \exp\left[ - \Big\langle \xi (u/\sqrt{\lambda}), \lambda\dGamma(P h)  \xi (u/\sqrt{\lambda}) \Big\rangle \right] \\
& =   \exp \left[ -\langle u, P h u\rangle \right] = \exp \left[ - \sum_{j=1}^K \lambda_j |\alpha_j|^2 \right].  
\end{align*}
Combining with the explicit formula for the free partition function (c.f. \eqref{eq:partition_quasi_free}), we arrive at
\begin{align} \label{eq:muP0-mu0P-0}
\mu_{P,0}(u) \ge \prod_{j=1}^K \left[  \frac{1}{\lambda\lambda_j} (1-e^{-\lambda\lambda_j}) \right] \mu_{0,K}(u).
\end{align}
Using
$$
\frac{1-e^{-t}}{t} \ge 1- \frac{t}{2}, \quad \forall t>0
$$ 
and Bernoulli's inequality, recalling~\eqref{eq:CLR} we can estimate
\begin{equation} \label{eq:Ber-inq}
 \prod_{j=1}^K \left[  \frac{1}{\lambda\lambda_j} (1-e^{-\lambda\lambda_j}) \right] \ge \prod_{j=1}^K  \left( 1- \frac{\lambda\lambda_j}{2}\right) \ge  1- \lambda\frac{\Lambda_e K}{2} \ge 1- \lambda\Tr[h^{-2}] \Lambda_e^{3}.
\end{equation}
Thus
\begin{align} \label{eq:muP0-mu0P-1}
\mu_{P,0}(u) \ge \Big(1- \tr[h^{-2}]\lambda  \Lambda_e ^{3} \Big) \mu_{0,K}(u),
\end{align}
which implies
$$\left(\mu_{P,0}-\mu_{0,K}\right)_-(u)\leq  \tr[h^{-2}]\lambda  \Lambda_e ^{3} \mu_{0,K}(u)$$
where $f_-=\max(-f,0)$ is the negative part. Integrating over $u\in P\gH$ we find
$$\int_{P\gH}\left(\mu_{P,0}-\mu_{0,K}\right)_-\leq \Tr[h^{-2}]\lambda  \Lambda_e ^{3}.$$
Notice then that
$$0=\int_{P\gH}\left(\mu_{P,0}-\mu_{0,K}\right)=\int_{P\gH}\left(\mu_{P,0}-\mu_{0,K}\right)_+-\int_{P\gH}\left(\mu_{P,0}-\mu_{0,K}\right)_-$$
so we get as announced that
$$\int_{P\gH}\left|\mu_{P,0}-\mu_{0,K}\right| = 2 \int_{P\gH}\left(\mu_{P,0}-\mu_{0,K}\right)_-  \leq 2\tr[h^{-2}]\lambda  \Lambda_e ^{3}.$$
\end{proof}

We now conclude the 

\begin{proof}[Proof of Proposition~\ref{pro:low bound}] Using Lemma~\ref{lem:free low cyl}   and the fact that $\cD_K[u]\geq 0$, we can estimate
\begin{align} \label{eq:cDu-K}
\int_{P\gH} e^{-\cD_K[u]} \; d\mu_{P,0}(u) &\le \int_{P\gH} e^{-\cD_K[u]} \; d\mu_{0,K}(u) +  \norm{\mu_{P,0}-\mu_{0,K}}_{L^1(P\gH)} \nn\\ 
&\leq \int_{P\gH} e^{-\cD_K[u]} \; d\mu_{0,K}(u) + 2 \lambda \Lambda_e^{3}. 
\end{align}
Inserting this bound in the right side of \eqref{eq:partition-lwb-0}, we arrive at the lower bound
\begin{align} \label{eq:partition-lwb-1}
 -\log \frac{\cZ(\lambda)}{\cZ_0(\lambda)} \ge -\log \left(  \int_{P\gH} e^{-\cD_K[u]} \; d\mu_{0,K}(u) + C\lambda \Lambda_e^{3} \right)-C  \lambda  \Lambda_e ^{3} - C ( \lambda  + \|Qh^{-1}\|_{\gS^2} )^{\frac1{14}}.\end{align}
Moreover, note that when $\lambda\to0$, we have $K={\rm dim} (P\gH)\to \infty$ since $\Lambda_e\to \infty$. Therefore, $\cD_K[u]\to \cD[u]$ in $L^1(\mu_0)$ by Lemma \ref{lem:re-interaction}. Consequently,
\begin{equation} \label{eq:DKT}
\lim_{\lambda\to0} \int_{P\gH} e^{-\cD_K[u]} \; d\mu_{0,K}(u) = \int e^{-\cD[u]} \; d\mu_{0}(u) \in (0,1). 
\end{equation}
by the dominated convergence theorem. Using the fact that $\log (1+t)=O(t)$ for $|t|$ small, we obtain
$$
 -\log \left(  \int_{P\gH} e^{-\cD_K[u]} \; d\mu_{0,K}(u) + C\lambda \Lambda_e^{3} \right) \ge -\log \left( \int_{P\gH} e^{-\cD_K[u]} \; d\mu_{0,K}(u) \right) -C\lambda \Lambda_e^{3}. 
$$ 
Thus from \eqref{eq:partition-lwb-1} we obtain the final quantitative lower bound
\begin{equation} \label{eq:total-error-to-Lambdae}
 -\log \frac{\cZ(\lambda)}{\cZ_0(\lambda)} \ge -\log \left(  \int_{P\gH} e^{-\cD_K[u]} \; d\mu_{0,K}(u) \right) -C  \lambda  \Lambda_e ^{3} - C ( \lambda  + \|Qh^{-1}\|_{\gS^2} )^{\frac1{14}}. \end{equation}
In particular, in the limit $\lambda\to0$, with the choice $1\ll \Lambda_e \ll \lambda^{-1/3}$ from \eqref{eq:total-error-to-Lambdae} and \eqref{eq:DKT} we conclude that 
$$
 -\log \frac{\cZ(\lambda)}{\cZ_0(\lambda)} \ge -\log \left(  \int  e^{-\cD[u]} \; d\mu_{0}(u) \right) + o(1) = -\log z + o(1). 
$$
This concludes the proof of the lower bound~\eqref{eq:low bound}. 
\end{proof}

\section{Free energy upper bound}\label{sec:up bound}

Now we complete the proof of the free energy convergence ~\eqref{eq:CV free ener} by providing a free-energy upper bound which complements Proposition~\ref{pro:low bound}:

\begin{proposition}[\textbf{Free-energy upper bound}]\label{pro:up bound}\mbox{}\\
Let $h>0$ satisfy \eqref{eq:Schatten h-strong-1}-\eqref{eq:Schatten h-strong-2} and let   $w:\R^d \to \R$ satisfy \eqref{eq:interaction}.  Let $z$ be the classical relative partition function defined in Lemma~\ref{lem:re-interaction}.
Then we have 
\begin{equation}\label{eq:up bound}
\limsup_{\lambda\to0^+}\left(-\log \frac{\cZl}{\cZ_0(\lambda)}\right) \leq - \log z.
\end{equation}
\end{proposition}

This part is conceptually easier than the free-energy lower bound. We rely on the variational principle and simply evaluate the free-energy of a suitable trial state. We split the proof into two main steps:
\begin{itemize}
 \item \emph{Reduction to a finite-dimensional estimate, Section~\ref{sec:up bound reduc}.} Our trial state coincides with the Gaussian state on high kinetic energy modes, and with a projected finite-dimensional interacting Gibbs state on low modes. We prove that to leading order its free-energy reduces to that in the low-energy sector, up to affordable errors. This is fairly similar to the analysis in Sections~\ref{sec:correl} and~\ref{sec:low ener}, but somewhat simpler. Some details will thus be skipped. 
\item \emph{Finite-dimensional semi-classics, Section~\ref{sec:up bound finite dim}.} Once we are reduced to treating a problem posed in a finite dimensional one-particle space, we are on more familiar terrain~\cite{Lieb-73b,Simon-80,Knowles-thesis,Gottlieb-05}, see e.g.~\cite[Appendix~B]{Rougerie-LMU,Rougerie-cdf}. We provide a proof of the needed free-energy upper bound for self-containedness and because we need to keep track of the dependence on the finite, but large, dimension.
\end{itemize}

\subsection{Reduction to a finite-dimensional estimate}\label{sec:up bound reduc}

We use similar low- and high-kinetic energy projectors as previously:
$$P =  \one(h\leq \Lambda_e), \quad Q=\1-P, \quad 1\ll \Lambda_e \ll \lambda^{-1/3}.$$
Let us define the interacting Gibbs state in $\gF(P\gH)$:
\begin{equation}\label{eq:Gibbs localized}
\Gamma_{\lambda,P}= \frac{e^{-\lambda(\dGamma(Ph)+\lambda \bWren_P)}}{\Tr_{\gF(P\gH)} e^{-\lambda(\dGamma(Ph)+\lambda \bWren_P)}} 
\end{equation}
where $\bWren_P$ is the localized interaction\footnote{Note that the expectation $\langle\dG (Pe ^{ik \cdot x}P) \rangle_0$ in $\Gamma_0$ is the same as that in $(\Gamma_0)_P$.} 
\begin{align}\label{eq:WP}
\bWren_P &= \frac12 \int_{\R^d} \hat{w} (k) \left|\dG (Pe ^{ik \cdot x}P) - \left\langle \dG (Pe ^{ik \cdot x}P) \right\rangle_0 \right| ^2 dk \nonumber \\
&= \frac12 \int_{\R^d} \hat{w} (k)  \left|\dG (P\cos (k \cdot x) P) - \left\langle \dG (P\cos (k \cdot x)P) \right\rangle_0\right| ^2 dk \nonumber \\
&+  \frac12 \int_{\R^d} \hat{w} (k)\left|\dG (P\sin (k \cdot x)P) - \left\langle \dG (P\sin (k \cdot x)P) \right\rangle_0\right| ^2 dk.
\end{align}
Note that $\Gamma_{\lambda,P}$ does not coincide with the state $(\Gamma_\lambda)_P$ obtained by $P$-localizing the full interacting Gibbs state, except in the non-interacting case
$$ \Gamma_{\lambda=0,P} = (\Gamma_0)_P.$$
Let 
$$ F_\lambda ^P := - \lambda^{-1} \log \left( \Tr \left(e^{-\lambda(\dGamma(Ph)+\lambda \bWren_P)} \right) \right) $$
be the free-energy of the $P$-localized problem.

In this subsection we prove the following, which reduces  Proposition~\ref{pro:up bound} to the corresponding estimate in a finite dimensional  subspace $P\gH$.

\begin{lemma}[\textbf{Reduction to low kinetic energy modes}]\label{lem:up reduc finite}\mbox{}\\
Let $h>0$ satisfy  $\tr[h^{-2}]<\ii$. Then 
\begin{align} \label{eq:partition-local-PP}
-\log \frac{\cZl}{\cZ_0(\lambda)} = \lambda\left(F_\lambda - F_0\right)
\le \lambda\left(F_\lambda^P - F_0^P\right) +  C  \left( \lambda  + \|Qh^{-1}\|_{\gS^2} \right)^{\frac1{14}}.
\end{align}
 \end{lemma}

Before proving this we interject  

\begin{lemma}[\textbf{Entropy relative to a product state}]\label{lem:rel ent}\mbox{}\\
Let $\gH_1$ and $\gH_2$ be two complex separable Hilbert spaces. Let $A$ be a state on $\gH_1\otimes \gH_2$ with the partial traces $A_1=\Tr_{\gH_2}A$, $A_2=\Tr_{\gH_1}A$ and let $B_1,B_2$ be states on $\gH_1,\gH_2$. Then
$$
\cH(A,B_1\otimes B_2) = \cH(A,A_1\otimes A_2) +  \cH(A_1,B_1)+ \cH(A_2,B_2).
$$
\end{lemma}

\begin{proof} 
Writing the spectral decompositions of $B_1,B_2$ one can easily see that
$$
\log(B_1\otimes B_2)= \log(B_1) \otimes \1 + \1 \otimes \log(B_2)
$$
and thus we can write 
\begin{align*}
\cH(A,B_1\otimes B_2)&=\Tr\left(A(\log A - \log (B_1\otimes B_2))\right) \\
&=\Tr\left(A(\log A - \log (A_1\otimes A_2))\right) + \Tr\left(A( \log (A_1)\otimes \1 - \log (B_1)\otimes \1))\right) \\
&\quad + \Tr\left(A( \1 \otimes \log (A_2) -\1 \otimes \log (B_2) )\right) \\
&= \cH(A,A_1\otimes A_2) + \cH(A_1,B_1)+ \cH(A_2,B_2).
\end{align*}
\end{proof}

In this section we only use Lemma \ref{lem:rel ent} in the simple case $A=A_1\otimes A_2$. The general version will be useful later in Section \ref{sec:DM}.

\begin{proof}[Proof of Lemma~\ref{lem:up reduc finite}]
In the last identity of~\eqref{eq:partition-local-PP} we use the fact that $(\Gamma_0)_P$ and $\Gamma_{\lambda,P}$ are the free and interacting Gibbs states in $\gF(P\gH)$, similarly as in~\eqref{eq:rel-energy}. The inequality is proved by a trial state argument.

\medskip

\noindent\textbf{Step 1: Trial state.} Using the unitary $\cU$ in \eqref{eq:Fock factor}, we define 
\begin{equation}\label{eq:trial state}
\Gammat= \cU^* \Big(  \Gamma_{\lambda,P} \otimes (\Gamma_0)_Q \Big) \cU
\end{equation}
where $\Gamma_{\lambda,P}$ is as in~\eqref{eq:Gibbs localized} and $(\Gamma_0)_Q$ is the $Q$-localization of the Gaussian state, c.f.  Definition~\ref{def:localization}. Importantly, from~\eqref{eq:loc creation} and~\eqref{eq:DM ann cre} one shows that 
\begin{equation}\label{eq:trial one body}
\Gammat ^{(1)} = P \Gamma_{\lambda,P} ^{(1)} P + Q \Gamma_0 ^{(1)} Q 
\end{equation}
and 
\begin{equation}\label{eq:trial two body}
\Gammat ^{(2)} = P^{\otimes 2} \Gamma_{\lambda,P} ^{(2)} P^{\otimes 2} + Q^{\otimes 2} \Gamma_0 ^{(2)} Q ^{\otimes 2} + \left( \Gamma_{\lambda,P} ^{(1)} \otimes Q \Gamma_0 ^{(1)} Q + Q \Gamma_0 ^{(1)} Q \otimes \Gamma_{\lambda,P} ^{(1)} \right). 
\end{equation} 
Also, since the relative entropy is unaffected by the unitary and the Gaussian state is factorized,
$$ \Gamma_0 =\cU^* \Big( (\Gamma_0)_P \otimes (\Gamma_0)_Q \Big) \cU, $$
we obtain from Lemma~\ref{lem:rel ent} that
\begin{align*}
\cH(\Gammat, \Gamma_0) &=\cH(\Gamma_{\lambda,P} \otimes (\Gamma_0)_Q, (\Gamma_0)_P \otimes (\Gamma_0)_Q)= \cH(\Gamma_{\lambda,P}, (\Gamma_0)_P).
\end{align*}
Hence, by the variational principle~\eqref{eq:rel-energy} 
\begin{align*} 
-\log \frac{\cZl}{\cZ_0(\lambda)} =  \lambda\left(F_\lambda - F_0\right) \le \cH(\Gammat, \Gamma_0) + \lambda^2\Tr [\bWren \Gammat] = \cH(\Gamma_{\lambda,P}, (\Gamma_0)_P) + \lambda^2 \Tr [\bWren \Gammat]. 
\end{align*}
On the other hand, from the choice of $\Gamma_{\lambda,P}$ we have 
$$
\lambda\left(F_\lambda^P - F_0^P\right) = \cH(\Gamma_{\lambda,P}, (\Gamma_0)_P) + \lambda^2 \Tr [\bWren_P \Gamma_{\lambda,P}]. 
$$
Thus there remains to evaluate the interaction energy of the trial state.

\medskip

\noindent\textbf{Step 2: Bound on the renormalized interaction energy.} We finish the proof of \eqref{eq:partition-local-PP} with the following claim: 
\begin{equation} \label{eq:local-W-P-2-PP}
\lambda^2 \Tr [\bWren \Gammat]  \le \lambda^2 \Tr [\bWren_P \Gamma_{\lambda,P}]+ C  ( \lambda  + \|Qh^{-1}\|_{\gS^2} )^{\frac1{14}}. 
\end{equation}
This is very similar in spirit to Lemma~\ref{lem:local-W-P} and we shall skip some details for brevity. In particular, since $(\Gamma_0)_P$ and $\Gamma_{\lambda,P}$ are the free and the interacting Gibbs states in the Fock space $\gF(P\gH)$, we can adapt to them (with the same proofs) most of the bounds on $\Gamma_0$ and $\Gamma_\lambda$ we used previously. 

First, using~\eqref{eq:trial one body} and arguing as for the proof of~\eqref{eq:estim_HS_relative_entropy}, we have
$$
\lambda \norm{ h^{1/2}(\Gammat^{(1)}-\Gamma_0^{(1)}) h^{1/2}}_{\gS^2} = \lambda \norm{ P h^{1/2}(\Gamma^{(1)}_{\lambda,P}-\Gamma_0^{(1)}) h^{1/2} P }_{\gS^2}  \le C. 
$$
Consequently, if we use again the notation
$$e_k^{-}=Pe_kP, \quad e_k^+=e_k-e_k^-$$
with $e_k$ being either $\cos(k\cdot x)$ or $\sin (k\cdot x)$, then 
following the proof of Theorem \ref{thm:correlation} , we obtain the variance estimate 
$$
\lambda^2\Tr \left[ \left|\dG (e_k^+) - \langle \dGamma(e_k^+) \rangle_0\right|^2 \Gammat\right] \le   C (1+|k|^2) ( \lambda  + \|Qh^{-1}\|_{\gS^2} )^{1/7}. 
$$
By using the Cauchy-Schwarz inequality
we find that 
\begin{align}  \label{eq:part-local-W-0-PP}
\lambda^2\Tr \left[  \left|  \dG (e_k) - \left\langle \dG (e_k) \right\rangle_0 \right| ^2 \Gammat \right] &\le (1+\eps)\lambda^2 \Tr \left[  \left| \dG (e_k^-) - \left\langle \dG (e_k^-) \right\rangle_0 \right| ^2 \Gammat  \right]  \nn \\
& \quad +  (1+\eps^{-1}) C (1+|k|^2) ( \lambda  + \|Qh^{-1}\|_{\gS^2} )^{1/7}. 
\end{align}
Integrating \eqref{eq:part-local-W-0-PP} against $\widehat w(k)$ over $k\in \Omega^*$ gives 
\begin{equation} \label{eq:local-W-P-2-PP-0000}
\lambda^2 \Tr \left[\bWren \Gammat\right]  \le (1+\eps)  \lambda^2 \Tr \left[\bWren_P \Gamma_{\lambda,P}\right] + (1+\eps^{-1})C  ( \lambda  + \|Qh^{-1}\|_{\gS^2} )^{1/7}. \end{equation}

Note that $\lambda^2\Tr \left[\bWren_P \Gamma_{\lambda,P}\right]$ is bounded uniformly in $\lambda$, which follows by inserting the trial state $(\Gamma_0)_P$ in  a variational formula similar to \eqref{eq:rel-energy}, and the fact that $\bWren_P\ge 0$. The desired result \eqref{eq:local-W-P-2-PP} thus follows from \eqref{eq:local-W-P-2-PP-0000} by optimizing over $\eps>0$.
\end{proof} 

\subsection{Finite-dimensional semi-classics}\label{sec:up bound finite dim}

The missing ingredient for the proof of Proposition~\eqref{pro:up bound} is the analysis of the partition functions in $\gF(P\gH)$ appearing in the right-hand side of~\eqref{eq:partition-local-PP}. We have 

\begin{lemma}[\textbf{Finite-dimensional semi-classics}]\label{lem:up finite anal}\mbox{}\\
Let $h>0$ satisfy  $\tr[h^{-2}]<\ii$. Let $ P =  \one(h\leq \Lambda_e)$ with $1\ll \Lambda_e \ll \lambda^{-1/3}$. Then 
\begin{align} \label{eq:up-finite-anal}
\lambda\left(F^P_\lambda - F^P_0\right) \leq -\log\left( \int_{P\gH} e^{-\cD_K[u]} d \mu_{0,K}(u) \right) +  C\lambda \Lambda_e^{3},
\end{align}
where $\cD_K$ is the truncated renormalized interaction from Lemma~\ref{lem:re-interaction}.
 \end{lemma}
 
\begin{proof}  
Recall that~\eqref{eq:partition_quasi_free} yields
\begin{equation}\label{eq:partition_quasi_free-PP}
\Tr e^{-\lambda\dGamma(Ph)} = \prod_{j=1}^K \frac{1}{1-e^{-\lambda\lambda_j }}
\end{equation}
where $\{\lambda_j\}_{j=1}^K$ are the eigenvalues of $PhP$ 
and that 
$$ \lambda\left(F^P_\lambda - F^P_0\right) = - \log \frac{\Tr e^{-\lambda(\dGamma(Ph)+\lambda \bWren_P)} } { \Tr e^{-\lambda\dGamma(Ph)} }.$$
To estimate the interacting partition function in the right-hand side, we use (a rescaled version of) the coherent-state resolution of the identity~\eqref{eq:resolution_coherent2}:
$$
(\lambda\pi)^{-K} \int_{P\gH} \left|\xi(u/\sqrt{\lambda})\right\rangle \left\langle \xi\left(u/\sqrt{\lambda}\right)\right| du =  \1_{\F(P\gH)}
$$
and the Peierls-Bogoliubov inequality $\pscal{x,e^Ax}\geq e^{\pscal{x,Ax}}$ to obtain
\begin{align} \label{eq:partition_int-PPPP} 
\Tr e^{-(\dGamma(Ph)+\lambda \bWren_P)/T} &= \frac1{(\lambda\pi)^{K}} \int_{P\gH} \Tr \left[ e^{-\lambda(\dGamma(Ph)+\lambda \bWren_P)} \left|\xi\left(u/\sqrt{\lambda}\right)\right\rangle \left\langle \xi\left(u/\sqrt{\lambda}\right)\right|   \right] du \nn\\
&= \frac1{(\lambda\pi)^{K}}  \int_{P\gH}  \left\langle \xi\left(u/\sqrt{\lambda}\right),  e^{-\lambda (\dGamma(Ph)+\lambda \bWren_P)}  \xi\left(u/\sqrt{\lambda}\right) \right\rangle du \nn\\ 
&\geq \frac\lambda{(\lambda\pi)^{K}} \int_{P\gH}    \exp \left[ - { \left\langle  \xi\left(u/\sqrt{\lambda}\right), \left(\dGamma(Ph)+\lambda \bWren_P\right) \xi\left(u/\sqrt{\lambda}\right) \right\rangle}  \right] du. 
\end{align}
Then, for $u\in P\gH$, similarly as in the proof of Lemma~\ref{lem:free low cyl}
$$
\lambda\left\langle \xi\left(u/\sqrt{\lambda}\right), \dGamma(Ph) \xi\left(u/\sqrt{\lambda}\right) \right\rangle = \langle u,h u\rangle.
$$
Moreover, calculating as in \eqref{eq:part-local-W-1} and recalling~\eqref{eq:DM_coherent}, \eqref{eq:ek-Gamma0}   we have
\begin{align}\label{eq:part-WP}
&\lambda^2\left\langle \xi\left(u/\sqrt{\lambda}\right),  \left| \dG (e_k^-) - \left\langle \dG (e_k^-) \right\rangle_{0} \right| ^2 \xi\left(u/\sqrt{\lambda}\right) \right\rangle \nn\\
&= \langle u, e_k^- u\rangle^2  -2 \lambda \langle u, e_k^- u\rangle \Tr \left[ e_k^{-}\Gamma_0^{(1)} \right]  + \lambda^2 \left(\Tr \left[ e_k^-\Gamma_0^{(1)} \right] \right)^2 + \lambda  \langle u, (e_k^-)^2 u\rangle  \nn\\
&\le  \Big( \langle u, e_k^- u\rangle - \left\langle\langle u, e_k^- u\rangle\right\rangle_{\mu_0} \Big)^2 +C \|u\|^2 \lambda \Lambda_e^2.
\end{align}
Since $\widehat w\in L^1$, we find that
\begin{align}
\left\langle \xi\left(u/\sqrt{\lambda}\right) \lambda^2 \bWren_P, \xi\left(u/\sqrt{\lambda}\right) \right\rangle \le  \cD_K[u] + C \|u\|^2 \lambda \Lambda_e^2.
\end{align}
Inserting the latter bound in~\eqref{eq:partition_int-PPPP} we arrive at
\begin{align} \label{eq:partition_int-PPPPP} 
\Tr e^{-\lambda(\dGamma(Ph)+\lambda \bWren_P)} \ge (\lambda\pi)^{-K}  \int_{P\gH} \exp \left[- \langle u, hu\rangle -  \cD_K [u] - C \|u\|^2 \lambda \Lambda_e^2 \right]  du. 
\end{align}
Combining with~\eqref{eq:partition_quasi_free-PP}, we find 
\begin{align} \label{eq:rel-partition-PPPP} 
\frac{\Tr e^{-\lambda(\dGamma(Ph)+\lambda \bWren_P)} } { \Tr e^{-\lambda\dGamma(Ph)} } &\ge \left[  \prod_{j=1}^K \frac{1}{\lambda\lambda_j} (1-e^{\lambda\lambda_j}) \right]  \int_{P\gH} \exp \left[- \cD_K [u] -   C \|u\|^2 \lambda \Lambda_e^2 \right]  d\mu_{0,K}(u)
\end{align}
where $d\mu_{0,K}$ is the cylindrical projection of $d\mu_0$ on $P\gH$, defined in \eqref{eq:def-mu0K}. Then, recall from~\eqref{eq:Ber-inq} that
$$
 \prod_{j=1}^K \left[  \frac{1}{\lambda\lambda_j} (1-e^{-\lambda\lambda_j}) \right] \ge 1 - C\lambda \Lambda_e^{3}.
$$
Using that $\cD_K [u]\ge 0$ by Lemma~\ref{lem:re-interaction}, we have
\begin{align*}
\exp \left[- \cD_K [u] - C \|u\|^2 \lambda \Lambda_e^2 \right] &= \exp \left[- \cD_K [u] \right] \exp \left[- C \|u\|^2 \lambda \Lambda_e^2 \right] \\
&\ge \exp \left[- \cD_K [u] \right]  (1- C \|u\|^2 \lambda \Lambda_e^2)\\
&\ge \exp \left[- \cD_K [u] \right]  - C \|u\|^2 \lambda \Lambda_e^2.
\end{align*}
Moreover, by~\eqref{eq:DM free meas} and \eqref{eq:def_mu0} we can bound
$$
\int_{P\gH} \| u\|^2 d\mu_{0,K}(u) = \Tr [Ph^{-1}] \le  \Tr [(\Lambda_e/h) h^{-1}] \le C \Lambda_e.
$$
Thus we infer
\begin{align*}
\int_{P\gH} \exp \left[-    \cD_K [u] - C \|u\|^2 \lambda \Lambda_e^2 \right]  d\mu_{0,K}(u) 
&\ge \int_{P\gH} \left( \exp \left[-  \cD_K [u] \right] -  C \|u\|^2 \lambda \Lambda_e^2 \right)  d\mu_{0,K}(u) \\
&\ge \int_{P\gH} \exp \left[-  \cD_K [u] \right] d\mu_{0,K}(u)   - C  \lambda  \Lambda_e^{3}.
\end{align*}
Therefore, it follows from~\eqref{eq:rel-partition-PPPP} that
\begin{align*} 
\frac{\Tr e^{-\lambda(\dGamma(Ph)+\lambda \bWren_P)} } { \Tr e^{-\lambda\dGamma(Ph)} } &\ge (1- C\lambda \Lambda_e^{3}) \left[ \int_{P\gH} e^{-\cD_K[u]}d\mu_{0,K}(u) - C  \lambda  \Lambda_e^{3}\right] \nn \\
&\ge \int_{P\gH} e^{-\cD_K[u]} d \mu_{0,K}(u) -  C\lambda \Lambda_e^{3}.
\end{align*}
Taking the $\log$ and using the fact that $\log (1+t)=O(t)$ for $|t|$ small concludes the proof.
\end{proof}

Now we can conclude the 

\begin{proof}[Proof of Proposition \ref{pro:up bound}]
Inserting \eqref{eq:up-finite-anal} in  \eqref{eq:partition-local-PP} we get the quantitative estimate
\begin{align}  \label{eq:partition-up-qua}
-\log \frac{\cZ(\lambda) } {\cZ_0(\lambda)} \le -\log \left( \int_{P\gH} e^{-\cD_K[u]} d \mu_{0,K}(u)\right) + C\lambda \Lambda_e^{3} + C  ( \lambda  + \|Qh^{-1}\|_{\gS^2} )^{1/14}
\end{align}
In the limit $\lambda\to0$ and $1\ll \Lambda_e\ll \lambda^{-1/3}$, using \eqref{eq:DKT} we obtain the  desired upper bound~\eqref{eq:up bound}
$$
-\log \frac{\cZ(\lambda) } {\cZ_0(\lambda)} \le -\log \left( \int  e^{-\cD[u]} d \mu_{0}(u)\right) + o(1) = -\log z + o(1). 
$$
The proof of Proposition~\ref{pro:up bound}, hence that of~\eqref{eq:CV free ener},  is complete.
\end{proof}

\section{Convergence of density matrices} \label{sec:DM}

In this section we prove the convergence of reduced density matrices stated in our main results. As in the previous sections we denote
\begin{equation*}
P = \one(h\leq \Lambda_e), \quad Q = \1 - P, \quad 1\ll \Lambda_e \ll \lambda^{-1/3}. 
\end{equation*}

\subsection{Collecting useful bounds}\label{sec:DM bounds} First, we collect several positive terms previously dropped in our analysis, and use them to derive some new information.

\begin{lemma}[\textbf{Trace-class estimates for projected states}]\label{lem:CV projected}\mbox{}\\
Let $h>0$ satisfy \eqref{eq:Schatten h-strong-1}-\eqref{eq:Schatten h-strong-2} and let  $w:\R^d \to \R$ satisfy \eqref{eq:interaction}. Then in the limit $\lambda\to0^+$  we have
\begin{align} \label{eq:cv-state-4b}
\Tr\Big|(\Gamma_\lambda)_Q- (\Gamma_0)_Q\Big| \to 0,  \\
 \label{eq:cv-state-4c}
\Tr\Big|\Gamma_\lambda - \cU^* \Big(  (\Gamma_\lambda)_P \otimes (\Gamma_\lambda)_Q \Big) \cU \Big| \to 0.
\end{align} 
Here $\cU$ is the unitary in~\eqref{eq:Fock factor} and $(\Gamma_\lambda)_P$, $(\Gamma_\lambda)_Q$ are localized states in $\gF(P\gH)$, $\gF(Q\gH)$, respectively, as in Definition \ref{def:localization}. Moreover, we have
\begin{equation} \label{eq:cv-state-4a}
\norm{\mu_{P,\lambda} - \widetilde \mu}_{L^1(P\gH)} \to 0
\end{equation} 
where 
\begin{equation}  \label{eq:def-wide-w}
d\widetilde \mu (u) := \frac{e^{-\cD_K[u]} \; d\mu_{0,K}(u)}{\int_{P\gH} e^{-\cD_K[v]} \; d\mu_{0,K}(v) }.
\end{equation}
Here $\mu_{P,\lambda}$ is the lower symbol of the Gibbs state $\Gamma_\lambda$ associated with $P$ and the scale $\eps=\lambda $ as in \eqref{eq:Husimi}, $\mu_{0,K}$ is the cylindrical Gaussian measure and $\cD_K[u]$ is the truncated renormalized interaction (all defined in Section \ref{sec:class meas}). 
\end{lemma}

Note that~\eqref{eq:cv-state-4b}-\eqref{eq:cv-state-4c} precisely confirm the expectation that the interacting and Gaussian Gibbs states almost coincide on high kinetic energy modes, whereas~\eqref{eq:cv-state-4a} quantifies the precision of the mean-field/semi-classical approximation on low kinetic energy modes.  

\begin{proof}[Proof of Lemma~\ref{lem:CV projected}]
After conjugating by the unitary $\cU$ in \eqref{eq:Fock factor}, the Gaussian quantum state is factorized:
\begin{equation}\label{eq:factor free}
\Gamma_0 = \cU^* \Big( \, (\Gamma_0)_P\otimes (\Gamma_0)_Q \Big)\, \cU.
\end{equation}
Hence we may apply Lemma \ref{lem:rel ent} to deduce
\begin{align} \label{eq:entropy-local-PQ}
\cH(\Gamma_\lambda,\Gamma_0)&=  \cH(\cU \Gamma_\lambda \cU^*,(\Gamma_0)_P\otimes (\Gamma_0)_Q)  \nn\\
&= \cH(\cU \Gamma_\lambda \cU^*, (\Gamma_\lambda)_P\otimes (\Gamma_\lambda)_Q) + \cH((\Gamma_\lambda)_P,(\Gamma_0)_P) +  \cH((\Gamma_\lambda)_Q,(\Gamma_0)_Q) \nn\\
&= \cH\Big(\Gamma_\lambda , \cU^* \Big( (\Gamma_\lambda)_P\otimes (\Gamma_\lambda)_Q \Big) \cU \Big) + \cH((\Gamma_\lambda)_P,(\Gamma_0)_P) +  \cH((\Gamma_\lambda)_Q,(\Gamma_0)_Q)
\end{align}
Combining \eqref{eq:entropy-local-PQ} with  the energy lower bound \eqref{eq:local-W-P} we obtain
\begin{align} \label{eq:cv-state-0}
 -\log \frac{\cZ(\lambda)}{\cZ_0(\lambda)} &= \cH(\Gamma_\lambda,\Gamma_0) + \lambda^2 \Tr [\bWren\Gamma_\lambda] \nn\\
&\ge \cH(\Gamma_\lambda, (\Gamma_\lambda)_P\otimes (\Gamma_\lambda)_Q) +  \cH((\Gamma_\lambda)_P,(\Gamma_0)_P) +  \cH((\Gamma_\lambda)_Q,(\Gamma_0)_Q) \nn\\
& + \int_{P\gH} \cD_K [u] d\mu_{P,\lambda}(u)+ o(1). 
\end{align}
Here recall that $\mu_{P,\lambda}$ is the lower symbol of $(\Gamma_\lambda)_P$. By the Berezin-Lieb inequality \eqref{eq:BL-G-G0} and the classical variational principle  \eqref{eq:zr-rel}, we refine~\eqref{eq:partition-lwb-0} to 
 \begin{align} \label{eq:cv-state-1}
&\cH((\Gamma_\lambda)_P,(\Gamma_0)_P) + \int_{P\gH} \cD_K [u] d\mu_{P,\lambda}(u) \nn\\
&\qquad \ge \cH_{\rm cl}(\mu_{P,\lambda},\mu_{P,0}) + \int_{P\gH} \cD_K [u] d\mu_{P,\lambda}(u)
 \nn\\
&\qquad =\cH_{\rm cl}(\mu_{P,\lambda},  \mu')  -\log \left( \int_{P\gH} e^{-\cD_K[u]} \; d\mu_{P,0}(u)\right).
\end{align}
where $\mu_{P,0}$ is the corresponding lower symbol of $(\Gamma_0)_P$ and
\begin{equation}  \label{eq:def-wide-w'}
d \mu' (u) := \frac{e^{-\cD_K[u]} \; d\mu_{P,0}(u)}{\int_{P\gH} e^{-\cD_K(v)} \; d\mu_{P,0}(v) }.
\end{equation}
Note that from \eqref{eq:mu_L1} we know that 
\begin{equation}  \label{eq:wt-w'}
\| \widetilde \mu-\mu'\|_{L^1(P\gH)} \to 0.
\end{equation}
We have already proved in~\eqref{eq:cDu-K} that
\begin{align}\label{eq:cv-state-2}
-\log \left( \int_{P\gH} e^{-\cD_K[u]} \; d\mu_{P,0}(u)\right) \ge   -\log \left( \int_{P\gH} e^{-\cD_K[u]} \; d\mu_{0,K}(u) \right) + o(1). 
\end{align}
Putting \eqref{eq:cv-state-0}, \eqref{eq:cv-state-1} and \eqref{eq:cv-state-2} together, we find that
\begin{align}\label{eq:cv-state-3}
 -\log \frac{\cZ(\lambda)}{\cZ_0(\lambda)} &\ge  \cH\Big(\Gamma_\lambda, \cU^* \Big( (\Gamma_\lambda)_P\otimes (\Gamma_\lambda)_Q\Big) \cU \Big) + \cH((\Gamma_\lambda)_Q,(\Gamma_0)_Q)+ \cH_{\rm cl}(\mu_{P,\lambda}, \mu') \nn\\
 & -\log \left( \int_{P\gH} e^{-\cD_K[u]} \; d\mu_{0,K}(u) \right) + o(1). 
\end{align}
Comparing with the the upper bound \eqref{eq:partition-up-qua}:
$$
-\log \frac{\cZ(\lambda) } {\cZ_0(\lambda)} \le -\log \left( \int_{P\gH} e^{-\cD_K[u]} d \mu_{0,K}(u) \right) + o(1),
$$
we obtain 
\begin{align} \label{eq:collect}
\cH\Big(\Gamma_\lambda, \cU^* \Big( (\Gamma_\lambda)_P\otimes (\Gamma_\lambda)_Q\Big) \cU \Big) + \cH((\Gamma_\lambda)_Q,(\Gamma_0)_Q)+ \cH_{\rm cl}(\mu_{P,\lambda}, \mu')  \to 0. 
\end{align}
Thanks to the (quantum and classical) Pinsker inequalities (see~\cite{CarLie-14} and~\cite[Section~5.4]{Hayashi-06}),
$$\cH(A,B)\geq \frac12(\tr|A-B|)^2,\qquad \cH_{\rm cl}(\nu_1,\nu_2)\geq \frac12 \Big(|\nu_1-\nu_2|(\gH)\Big)^2,$$
the convergence \eqref{eq:collect} implies the desired convergences ~\eqref{eq:cv-state-4b}, \eqref{eq:cv-state-4c} and
$$
\|\mu_{P,\lambda}-\mu'\|_{L^1(P\gH)} \to 0. 
$$
The latter bound and \eqref{eq:wt-w'} imply \eqref{eq:cv-state-4a} by the triangle inequality. 
\end{proof}

\subsection{Hilbert-Schmidt convergence of all density matrices}\label{sec:cv-all-DM} 

In this subsection we derive the Hilbert-Schmidt convergence for all density matrices using the additional condition \eqref{eq:Schatten h-strong-3}. 

\begin{lemma}[\textbf{Hilbert-Schmidt convergence of all density matrices}]\label{lem:CV-DM-high}\mbox{}\\
Let $h>0$ satisfy  \eqref{eq:Schatten h-strong-1}-\eqref{eq:Schatten h-strong-2}-\eqref{eq:Schatten h-strong-3} and let $w$ satisfy \eqref{eq:interaction}.  Then in the limit $\lambda\to0^+$, for all $k\ge 1$ we have 
$$
\lambda^k\,k!\, \Gamma_\lambda^{(k)} \to  \int |u^{\otimes k} \rangle \langle u^{\otimes k} | d \mu(u)
$$
strongly in the Hilbert-Schmidt space $\gS^2(\gH^{\otimes^k_s})$. 
\end{lemma}

We will need a uniform bound on all density matrices in the Hilbert-Schmidt norm. This is the only place where we need the condition \eqref{eq:Schatten h-strong-3}. 

\begin{lemma}[\textbf{Hilbert-Schmidt estimate}]\label{lem:HS-norm-bd}\mbox{}\\
Let $h$ satisfy  \eqref{eq:Schatten h-strong-1}, \eqref{eq:Schatten h-strong-3} and let $w$ satisfy \eqref{eq:interaction}. Then for every $k\ge 1$, we have 
$$ \left\|\lambda^{k}\Gamma_\lambda^{(k)}\right\|_{\gS^2}\le C_k.$$
\end{lemma}

\begin{proof}  From the positivity $e^{-th}(x,y)\ge 0$ and $\lambda\bWren\ge 0$, a standard argument using the Trotter product formula (see e.g.~\cite[Theorem VIII.30]{ReeSim1} or~\cite[Theorem~1.1]{Simon-05}) and the relative bound on partition functions in Lemma \ref{lem:partition}, we obtain the kernel estimate
$$
0\le \Gamma_\lambda^{(k)}(X_k;Y_k)\le C_k\Gamma_{0}^{(k)}(X_k;Y_k). 
$$
See e.g.~\cite[Lemma 4.3]{LewNamRou-18a} for a detailed explanation. Consequently, for every $k\ge 1$ we have the Hilbert-Schmidt estimate
$$
\left\|\lambda^{k}\Gamma_\lambda^{(k)}\right\|_{\gS^2}\le C_k \left\| \lambda^{k} \Gamma_0^{(k)}\right\|_{\gS^2}= C_k \left\| \frac{\lambda^k}{(e^{\lambda h}-1)^{\otimes k}}\right\|_{\gS^2} \le C_k \|h^{-1}\|_{\gS^2}^k.
$$
Note that the bound is uniform in $\lambda$ and depends on $h$ only via $\tr[h^{-2}]$. 
\end{proof}

Next, we have

\begin{lemma} [\textbf{From states to density matrices, Hilbert-Schmidt estimate}]\label{lem:states to DMs-II}\mbox{}\\ 
Let $\Gamma, \Gamma'$ be two states on Fock space that commute with the number operator $\cN$. Then for all $k\ge 1$, we have the Hilbert-Schmidt norm estimate on the associated density matrices
\begin{align}\label{eq:G-DMk-HS}
\| \Gamma^{(k)} - \Gamma'^{(k)} \|_{\gS^2}^2 \le C_k  \Big( \Tr |\Gamma-\Gamma'| \Big) \left( \sum_{\ell=k}^{2k} \Big(\| \Gamma^{(\ell)}\|_{\gS^2} +\| \Gamma'^{(\ell)}\|_{\gS^2}  \Big)  \right). 
\end{align}
\end{lemma}

\begin{proof} Let $A_k$ be a non-negative Hilbert-Schmidt operator on $\gH^{\otimes_s k}$ and $\bA_k$ the associated second-quantized operator on the Fock space from Definition~\ref{def:quantiz}. Using~\eqref{eq:def k body} we have
\begin{align}\label{eq:G-DMk-HS-00}
\left| \Tr \Big[ A_k (\Gamma^{(k)} - \Gamma'^{(k)})\Big] \right|^2 &= \left| \Tr \Big[ \bA_k(\Gamma-\Gamma') \Big] \right|^2\nn\\
&\le \Big( \Tr |\Gamma-\Gamma'| \Big) \left(  \Tr \Big[ (\bA_k)^2 |\Gamma-\Gamma'| \Big]  \right)\nn \\
&\le  \Big( \Tr |\Gamma-\Gamma'| \Big) \left(  \Tr \Big[ (\bA_k)^2 (\Gamma+\Gamma') \Big]  \right). 
\end{align}
On the $n$-particle sector, we can compute explicitly 
\begin{align*}
\left(\sum_{1\le i_1<...<i_k\le n} (A_k)_{i_1,...,i_k}\right)^2 &= \sum_{\substack{1\le i_1<...<i_k\le n\\ 1\le j_1<...<j_k \le n }} (A_k)_{i_1,...,i_k}(A_k)_{j_1,...,j_k} \\
&=\sum_{\ell=k}^{\min\{2k,n\}}\sum_{1\le i_1<...<i_\ell \leq n} (B_\ell)_{i_1,i_2,...,i_\ell} 
\end{align*}
where $B_\ell$ is an operator on $\gH^{\otimes_s \ell}$ defined by
\begin{align} \label{eq:def-Bell}
(B_\ell)_{1,2,...,\ell} := \sum_{\substack{1\le i_1<...<i_k\le \ell \\ 1\le j_1<...<j_k\le \ell \\ \{i_1,...,i_k\}\cup \{j_1,...,j_k\}=\{1,...,\ell\}}} (A_k)_{i_1,...,i_k}(A_k)_{j_1,...,j_k}.
\end{align}
Therefore we have
$$
\bA_k ^2= \sum_{\ell=k}^{2k} \bB_\ell
$$
where $\bB_\ell$ is the second quantization of $B_\ell$, as in Definition~\ref{def:quantiz} again.

On the other hand, since $A_k$ is a Hilbert-Schmidt operator on $\gH^{\otimes_s k}$, we can prove that $B_\ell$ is a Hilbert-Schmidt operator on $\gH^{\otimes_s \ell}$ and 
\begin{align} \label{eq:Bell-HS}
\|B_\ell\|_{\gS^2} \le C_k \|A_k\|_{\gS^2}^2. 
\end{align}
To prove \eqref{eq:Bell-HS}, let us come back to the definition \eqref{eq:def-Bell}. Consider a general $\ell$-particle operator of the form 
$$\bA=(A_k)_{X,Y} (A_k)_{X,Z}$$
with $(X,Y), (X,Z)$ are $k$-particle variables. If the kernel of $(A_k)$ is $(A_k)(X,Y;X',Y')$, then the kernel of $\bA$ is 
$$
\bA(X,Y,Z; X',Y',Z')= \int dX'' (A_k)(X,Y;X'',Y') (A_k)(X'',Z;X',Z'). 
$$
By the Cauchy-Schwarz inequality we have
$$
|\bA(X,Y,Z; X',Y',Z')|^2 \le \left( \int d X'' |(A_k)(X,Y;X'',Y')|^2   \right) \left( \int dX'' |(A_k)(X'',Z;X',Z')|^2 \right).
$$
Therefore,
\begin{align*}
\|\bA\|_{\gS^2}^2&=\int dX dY dZ dX' dY' dZ' |\bA(X,Y,Z; X',Y',Z')|^2\\
&\le  \int dX dY dZ dX' dY' dZ'  \left( \int d X'' |(A_k)(X,Y;X'',Y')|^2   \right) \times\\
&\qquad \left( \int dX'' |(A_k)(X'',Z;X',Z')|^2 \right) \\
&= \left( \int dX d Y dX' dY' |(A_k)(X,Y;X',Y')|^2   \right)^2 = \|A_k\|_{\gS^2}^4.   
\end{align*}
We thus obtain \eqref{eq:Bell-HS} immediately from the definition \eqref{eq:def-Bell}. 

Using \eqref{eq:Bell-HS}, we can estimate 
\begin{align} \label{eq:dGA-dGB-Gk}
\Tr \Big[ \bA_k ^2 \Gamma \Big] &= \sum_{\ell=k}^{2k} \Big[ \bB_\ell \Gamma \Big] =  \sum_{\ell=k}^{2k} \Tr \Big[ B_\ell \Gamma^{(\ell)}\Big] \nn\\
&\le \sum_{\ell=k}^{2k} \|B_\ell\|_{\gS^2} \| \Gamma^{(\ell)}\|_{\gS^2} \le C_k \|A_k\|_{\gS^2}^2 \sum_{\ell=k}^{2k} \| \Gamma^{(\ell)}\|_{\gS^2}. 
\end{align}
Inserting \eqref{eq:dGA-dGB-Gk} and a similar estimate for $\Gamma'$ in \eqref{eq:G-DMk-HS-00} we arrive at
\begin{align}\label{eq:G-DMk-HS-23}
\left| \Tr \Big[ A_k (\Gamma^{(k)} - \Gamma'^{(k)})\Big] \right|^2 \le C_k  \|A_k\|_{\gS^2}^2  \Big( \Tr |\Gamma-\Gamma'| \Big) \left( \sum_{\ell=k}^{2k} \Big(\| \Gamma^{(\ell)}\|_{\gS^2} +\| \Gamma'^{(\ell)}\|_{\gS^2}  \Big)  \right). 
\end{align}
This being true for any $k$-body Hilbert-Schmidt operator $A_k$ leads to the desired bound \eqref{eq:G-DMk-HS} by duality.  
\end{proof}

Now we are ready to conclude the 

\begin{proof}[Proof of Lemma \ref{lem:CV-DM-high}]  Let 
$$\Gammat_\lambda = \cU^* \Big( (\Gamma_\lambda)_P \otimes (\Gamma_\lambda)_Q \Big) \cU.$$
From the action~\eqref{eq:loc creation} of the partial isometry $\cU$ on creation/annihilation operators one can compute that 
\begin{equation}\label{eq:DMs tilde}
\Gammat_\lambda ^{(k)} = P^{\otimes k} \Gamma_\lambda ^{(k)} P^{\otimes k} + Q^{\otimes k} \Gamma_\lambda ^{(k)} Q^{\otimes k} + \mathrm{Cross}
\end{equation}
where $\mathrm{Cross}$ is a sum of finite coefficients (depending only on $k$) times terms of the form 
\begin{equation}\label{eq:cross terms}
\mathrm{Cross}_{l} = A_1^{\otimes j_1} \Gamma_\lambda ^{(j_1)} A_1^{\otimes j_1} \otimes \ldots \otimes A_l^{\otimes j_l} \Gamma_\lambda ^{(j_l)} A_l^{\otimes j_l}  
\end{equation}
where $\sum_{i=1} ^l j_i = k$ and $A_i = P$ or $Q$, but not all $A_i$ are simultaneously equal to $P$ or $Q$. The precise expression does not matter for us, but we have already used the expressions for $k=1,2$, so let us write them explicitly once more:
\begin{align}\label{eq:DMs tilde 12}
\Gammat_\lambda ^{(1)} &= P \Gamma_\lambda ^{(1)} P + Q \Gamma_\lambda ^{(1)} Q, \nn \\ 
\Gammat_\lambda ^{(2)} &= P^{\otimes 2} \Gamma_\lambda ^{(2)} P^{\otimes 2} + Q^{\otimes 2} \Gamma_\lambda ^{(2)} Q^{\otimes 2} + P\Gamma_{\lambda} ^{(1)} P \otimes Q \Gamma_\lambda ^{(1)} Q + Q \Gamma_\lambda ^{(1)} Q \otimes P \Gamma_{\lambda} ^{(1)} P. 
\end{align}
From formula \eqref{eq:DMs tilde} and the uniform bound $ \| \lambda^{k} \Gamma_\lambda\|_{\gS^2} \le C_k$ in Lemma \ref{lem:HS-norm-bd}, we deduce the similar bound for $\Gammat_\lambda$:
$$
\| \lambda^{k} \Gammat_\lambda^{(k)}\|_{\gS^2} \le C_k, \quad \forall k\ge 1. 
$$
Therefore, using the state convergence \eqref{eq:cv-state-4c} and Lemma \ref{lem:states to DMs-II} we obtain
\begin{align} \label{eq:Gk-Gtk-HS}
\left\| \lambda^{k} \Big( \Gamma_\lambda^{(k)} - \Gammat_\lambda^{(k)} \Big) \right\|_{\gS^2} \to 0, \quad \forall k\ge 1. 
\end{align}
Thus it remains to prove the convergence for $\Gammat_\lambda^{(k)}$. 

After multiplying \eqref{eq:DMs tilde} by $\lambda^k$, the main claim is that the $P$-localized term (first term) converges to the desired limit strongly in $\gS^2$, while the $Q$-localized term (second term) converges to $0$ strongly in $\gS^2$. Combining these two facts, the cross-terms must also converge to $0$ strongly in $\gS^2$. 

\medskip

\noindent{\bf Analysis of the $P$-localized term.} We use the quantitative quantum de Finetti theorem~\ref{thm:quant deF}. Recalling the lower symbol $\mu_{P,\lambda}$ of $(\Gamma_\lambda)_P$, we have from \eqref{eq:Chiribella} 
\begin{equation}\label{eq:Chiribella-a1}
\int_{P\gH}|u^{\otimes k}\>\<u^{\otimes k}|\;d\mu_{P,\lambda}(u) = \lambda^k\,k!\, P^{\otimes k}\Gamma^{(k)} P^{\otimes k} + \lambda^k\,k!\, \sum_{\ell = 0} ^{k-1} {k \choose \ell} P^{\otimes \ell} \Gamma^{(\ell)} P^{\otimes \ell} \otimes_s \one_{\otimes_s ^{k-\ell} P\gH}.
\end{equation}

From the lower symbol expression \eqref{eq:Chiribella-a1}, taking the Hilbert-Schmidt norm on both sides, then using the uniform bound in Lemma \ref{lem:HS-norm-bd} and the fact that $\dim(P\gH)\le C\Lambda_e^2\ll \lambda^{-1}$, we find that for every $k\ge 1$, 
\begin{align} \label{eq:CV-HS-a1}
\left\| \lambda^k\,k!\, P^{\otimes k}\Gamma^{(k)} P^{\otimes k}  - \int_{P\gH}|u^{\otimes k}\>\<u^{\otimes k}|\;d\mu_{P,\lambda}(u) \right\|_{\gS^2}\to 0.
\end{align}
Consequently,  
\begin{align} \label{eq:CV-HS-ab}
\left\| \int_{P\gH}|u^{\otimes k}\>\<u^{\otimes k}|\;d\mu_{P,\lambda}(u) \right\|_{\gS^2}\le C_k.
\end{align}
A similar estimate with $\mu_{P,\lambda}$  replaced by $\widetilde \mu$ in \eqref{eq:def-wide-w} holds thanks to the operator inequality
\begin{align} \label{eq:int-u^k-tildemu}
\int_{P\gH} |u^{\otimes k}\rangle \langle u^{\otimes k}| d\widetilde \mu \le C \int_{P\gH} |u^{\otimes k}\rangle \langle u^{\otimes k}| d\mu_{0,K} = Ck! (Ph^{-1})^{\otimes k}. 
\end{align}

Next, for every Hilbert-Schmidt operator  $X\ge 0$ on $\gH^{\otimes_s k}$, we can estimate
\begin{align*}
&\left| \Tr \left[ X \left( \int_{P\gH} |u^{\otimes k}\>\<u^{\otimes k}| (d\mu_{P,\lambda} - d \widetilde \mu)(u) \right)  \right] \right|^2 \nn\\
&\qquad\qquad = \left|  \int_{P\gH} \langle  u^{\otimes k}, X u^{\otimes k}\rangle (d\mu_{P,\lambda}- d \widetilde \mu)(u)  \right|^2 \\
&\qquad\qquad \le   \left( \int_{P\gH} |\langle  u^{\otimes k}, X u^{\otimes k}\rangle|^2  |d\mu_{P,\lambda}- d \widetilde \mu| (u) \right)  \left( \int_{P\gH}  |d\mu_{P,\lambda}- d \widetilde \mu| (u) \right) \\
&\qquad\qquad \le \left( \int_{P\gH} \langle  u^{\otimes 2k}, X\otimes X u^{\otimes 2k}\rangle  (d\mu_{P,\lambda}+ d \widetilde \mu) (u) \right) |\mu_{P,\lambda}-\widetilde\mu|(P\gH)  \\
&\qquad\qquad =   \Tr \left[ X\otimes X \left( \int_{P\gH} |  u^{\otimes 2k}\rangle \langle u^{\otimes 2k}| (d\mu_{P,\lambda}+ d \widetilde \mu) (u) \right) \right] |\mu_{P,\lambda}-\widetilde\mu|(P\gH) \\
&\qquad\qquad \le    \|X\otimes X\|_{\gS^2} \left\| \int_{P\gH} |  u^{\otimes 2k}\rangle \langle u^{\otimes 2k}| (d\mu_{P,\lambda}+ d \widetilde \mu) (u)  \right\|_{\gS^2}  |\mu_{P,\lambda}-\widetilde\mu|(P\gH) \to 0. 
\end{align*}
Here in the last estimate we have used \eqref{eq:cv-state-4a} and \eqref{eq:CV-HS-ab}.  By duality we deduce that 
\begin{align} \label{eq:CV-HS-a2}
\left\| \int_{P\gH} |u^{\otimes k}\>\<u^{\otimes k}|  (d\mu_{P,\lambda} - d \widetilde \mu)(u)  \right\|_{\gS^2} \to 0. 
\end{align}
Thus, by the triangle inequality,
\begin{align} \label{eq:CV-HS-a4}
\left\| \lambda^k\,k!\, P^{\otimes k} \Gamma_\lambda^{(k)} P^{\otimes k}  - \int_{P\gH}|u^{\otimes k}\>\<u^{\otimes k}|d \widetilde \mu(u) \right\|_{\gS^2} \to 0.
\end{align}
Therefore, for all $k\ge 1$,
\begin{align} \label{eq:CV-HS-a5}
\left\| \lambda^k\,k!\, P^{\otimes k} \Gamma_\lambda^{(k)} P^{\otimes k}  - \int |u^{\otimes k}\>\<u^{\otimes k}|\; d\mu(u) \right\|_{\gS^2} \to 0.
\end{align}

\medskip

\noindent{\bf Analysis of the $Q$-localized term.} Using Lemma \ref{lem:states to DMs-II}, Lemma \ref{lem:HS-norm-bd} and \eqref{eq:cv-state-4b} we can estimate
\begin{multline} \label{eq:CV-HS-a6}
\left\| Q^{\otimes k} \lambda^k\left(\Gamma_\lambda^{(k)} -\Gamma_0^{(k)}\right) Q^{\otimes k}  \right\|_{\gS^2}^2 = \left\| \lambda^k\left((\Gamma_\lambda)_Q^{(k)} -(\Gamma_0)_Q^{(k)}\right)  \right\|_{\gS^2}^2\\
\le C_k  \lambda^{2k} \Big( \Tr |(\Gamma_\lambda)_Q-(\Gamma_0)_Q| \Big)  \sum_{\ell=k}^{2k} \Big(\| (\Gamma_\lambda)_Q^{(\ell)}\|_{\gS^2} +\| (\Gamma_0)^{(\ell)}\|_{\gS^2}  \Big) \to 0
\end{multline}
for all $k\ge 1$. From \eqref{eq:CV-HS-a5} and \eqref{eq:CV-HS-a6} we can go back to~\eqref{eq:DMs tilde}, control all the cross terms, and conclude that
$$
\left\| \lambda^k\,k!\, \Gammat_\lambda^{(k)}  - \int |u^{\otimes k}\>\<u^{\otimes k}|\; d\mu(u) \right\|_{\gS^2} \to 0
$$
for all $k\ge 1$. The desired convergence of $\Gammat_\lambda^{(k)}$ then follows from \eqref{eq:Gk-Gtk-HS} and the triangle inequality.  This concludes  the proof of Lemma \ref{lem:CV-DM-high}. 
\end{proof}

\subsection{Trace class convergence of relative one-body density matrix}

To conclude the proof of Theorem \ref{thm:main-3}, it remains to prove the convergence of the relative density matrices  in the trace class norm. 

\begin{lemma}[\textbf{Trace class convergence of the relative one-body density matrix}]\label{lem:CV-1pdm}\mbox{}\\
Let $h>0$ satisfy  \eqref{eq:Schatten h-strong-1}-\eqref{eq:Schatten h-strong-2}-\eqref{eq:Schatten h-strong-3} and let $w$ satisfy \eqref{eq:interaction}.  Then in the limit $T=\lambda^{-1}\to \infty$, for all $k\ge 1$ we have 
$$
\lambda\left( \Gamma_\lambda^{(1)}  - \Gamma_0^{(1)} \right)\to  \int |u  \rangle \langle u  | (d \mu(u)-d\mu_0(u))
$$
strongly in the trace class space $\gS^1(\gH)$. 
\end{lemma}

\begin{proof} Denote 
$$
X_\lambda:=\lambda\left( \Gamma_\lambda^{(1)}  - \Gamma_0^{(1)} \right) , \quad X_0:=  \int |u \rangle \langle u  | (d \mu(u)-d\mu_0(u)). 
$$
Note that $X_0$ is a trace class operator, thanks to \eqref{eq:DM rel clas norm}.  From Lemma \ref{lem:CV-DM-high} we  have the Hilbert-Schmidt convergence $X_\lambda \to X_0$. Moreover, the uniform estimate in Theorem \ref{thm:estim_relative_entropy} ensures that 
$\| h^{1/2} X_\lambda h^{1/2}\|_{\gS^2} \le C$. Let
$$P_L=\1_{h\le L}, \quad Q_L=\1-P_L=\1_{h\ge L}.$$
By the triangle and Cauchy-Schwarz inequalities we can bound
\begin{align*}
\|X_\lambda - X_0 \|_{\gS^1} &\le \|P_L (X_\lambda -X_0)\|_{\gS^1}   + \|Q_LX_\lambda \|_{\gS^1} + \|Q_L X_0\|_{\gS^1} \\
&\le  \|P_L\|_{\gS^2} \|X_\lambda - X_0 \|_{\gS^2} + \| h^{-1/2}Q_L h^{-1/2}\|_{\gS^2} \|h^{1/2}X_\lambda h^{1/2}\|_{\gS^2} + \|Q_L X_0\|_{\gS^1}. 
\end{align*}
For any fixed $L>0$, taking $\lambda\to \infty$ we get 
$$
\limsup_{\lambda\to0} \|X_\lambda -X_0 \|_{\gS^1} \le  \| h^{-1/2}Q_L h^{-1/2}\|_{\gS^2} + \|Q_L X_0\|_{\gS^1}. 
$$
Then taking $L\to \infty$ in the latter estimate and using $h^{-1}\in \gS^2, X_0\in \gS^1$ we get 
$$
\lim_{\lambda\to0} \|X_\lambda\|_{\gS^1} =0.
$$
This concludes the proof of Lemma \ref{lem:CV-1pdm}, and hence that of Theorem \ref{thm:main-3}. 
\end{proof}


\begin{remark}[Relative higher density matrices]\label{rm:G2}\mbox{}\\
If $\Tr(h^{-1})=+\infty$, then the difference $\lambda^{k}(\Gamma_\lambda^{(k)} -  \Gamma_0^{(k)})$ is {\em not} bounded in trace class for every $k\ge 2$.  For brevity we only explain this for $k=2$, in the homogeneous case in dimension $d=2$. We prove that $\lambda^2(\Gamma_\lambda^{(2)} -  \Gammat_\lambda^{(2)})$ converges to $0$ in trace class, but $\lambda^2(\Gammat_\lambda^{(2)} -  \Gamma_0^{(2)})$ is not bounded in trace class, where $ \Gammat_\lambda= \cU^* \left(  (\Gamma_\lambda)_P \otimes (\Gamma_0)_Q \right) \cU$. 

First, in two dimensions we have since $\Tr(h^{-p})<\infty$ for any $1<p<2$ (see \eqref{eq:p-N0}), from the energy estimates \eqref{eq:partition-lwb-1} and \eqref{eq:partition-up-qua}, we can use   
$$
\| Q h^{-1}\|_{\gS^2} =\|\1_{h\ge \Lambda_e} h^{-1}\|_{\gS^2}  \le \sqrt{  \Lambda_e^{p-2} \Tr(  h^{-p})}. 
$$
and optimize over $\Lambda_e$ to get the explicit bound  
\begin{align} \label{eq:partition-explicit-error}
\left|  -\log \frac{\cZ(\lambda)}{\cZ_0(\lambda)} + \log \left(  \int_{P\gH} e^{-\cD_K[u]} \; d\mu_{0,K}(u) \right) \right| \le C\lambda^{\eta}
\end{align}
for a small constant $\eta>0$. Consequently, repeating the proof of \eqref{eq:cv-state-4c} we have 
\begin{align}  \label{eq:cv-state-4c-error}
\Tr\Big|\Gamma_\lambda - \Gammat_\lambda \Big| \le CT^{-\eta/2}.
\end{align} 
Since $\lambda^{k} \Tr(\langle \cN^k\rangle_{\Gamma_\lambda})$ diverges like $|\log \lambda|^k$ due to \eqref{eq:number free int} and \eqref{eq:p-N0}, we can deduce from \eqref{eq:cv-state-4c-error} and Lemma \ref{lem:states to DMs}  below that 
$$
\Tr \Big| \lambda^2 \Big( \Gamma_\lambda^{(2)} - \Gammat_\lambda^{(2)} \Big) \Big| \to 0.  
$$

On the other hand, from \eqref{eq:DMs tilde 12} we can write 
$$ 
\lambda^2 P\otimes Q \Big( \Gammat_\lambda^{(2)} - \Gamma_0^{(2)} \Big) P\otimes Q =  \Big( \lambda  P(\Gamma_{\lambda} ^{(1)}-\Gamma_0^{(1)}) P \Big) \otimes \Big( T^{-1 } Q \Gamma_0^{(1)} Q \Big).
$$
By Lemma \ref{lem:CV-1pdm} the first term $\lambda P(\Gamma_{\lambda} ^{(1)}-\Gamma_0^{(1)}) P$ converges in trace class to a nontrivial limit. However, the second term $\lambda Q \Gamma_0^{(1)} Q$ in unbounded in trace class since $\Tr(h^{-1})=+\infty$. Therefore, $\lambda^2 \Big( \Gammat_\lambda^{(2)} - \Gamma_0^{(2)} \Big)$ is unbounded in trace class. Thus we conclude that $\lambda^2(\Gamma_\lambda^{(2)} -  \Gamma_0^{(2)})$ is unbounded in trace class, in the 2D homogeneous case.  \hfill$\diamond$ 
\end{remark}

In the above remark, we have used the following result (c.f. Lemma \ref{lem:states to DMs-II}). 

\begin{lemma}[\textbf{From states to density matrices, trace-class estimate}]\label{lem:states to DMs}\mbox{}\\
Let $\Gamma, \Gamma'$ be two states on Fock space that commute with the number operator $\cN$. Then for all $q,q'>1$ with $1/q+1/q'=1$ we have 
\begin{equation}\label{eq:DMk-cN-Tr}
\Tr |\Gamma^{(k)} -\Gamma'^{(k)} | \le \left( \Tr |\Gamma-\Gamma'| \right)^{1/q'} \left( \Tr[ \cN^{qk}(\Gamma+\Gamma')]\right)^{1/q}.
\end{equation}
\end{lemma}

\begin{proof}
We write
$$\Gamma=\bigoplus_{n=0}^\infty G_n,\quad \Gamma'=\bigoplus_{n=0}^\infty G_n'$$
where $G_n,G_n'$ are operators on the $n$-particle sectors, and denote 
$$ G_n ^{(k)} = \tr_{k+1\to n} [G_n].$$
By H\"older's inequality we may estimate
\begin{align*}
\Tr |\Gamma^{(k)} -\Gamma'^{(k)} | &\le \sum_{n=k}^\infty n^k \Tr |G_n^{(k)} -{G'}_n^{(k)}| \\
&\le \left( \sum_{n=0}^\infty \Tr |G_n^{(k)} -{G'}_n^{(k)}| \right)^{1/q'} \left(  \sum_{n=0}^\infty n^{qk} \Tr |G_n^{(k)} -{G'}_n^{(k)}| \right)^{1/q} \\
&\le \left( \sum_{n=0}^\infty \Tr |G_n-{G'}_n| \right)^{1/q'} \left(  \sum_{n=0}^\infty n^{qk} (\Tr G_n + \Tr {G'}_n) \right)^{1/q} \\
&= \left( \Tr |\Gamma-\Gamma'|\right)^{1/q'} \left( \Tr[ \cN^{qk}(\Gamma+\Gamma')] \right)^{1/q}.
\end{align*}
\end{proof}

\section{Conclusion of the proofs} \label{sec:conclusion}

Theorem~\ref{thm:main-1} is a particular instance of Theorem~\ref{thm:main-3}. It remains to explain how to deduce Theorem \ref{thm:main-2} from Theorem \ref{thm:main-3}.

\begin{proof}[Proof of Theorems \ref{thm:main-2}] Theorem~\ref{thm:main-2} essentially follows from Theorem~\ref{thm:main-3} with 
$h=-\Delta+V_\lambda$
where $V_\lambda$ solves the counter-term problem~\eqref{eq:counter-term}. To be precise, the Gibbs state 
\begin{equation}\label{eq:tilde Gibbs state}
\widetilde \Gamma_\lambda = \widetilde \cZ(\lambda)^{-1} e^{- \lambda\widetilde\bH_\lambda} 
\end{equation}
with 
$$
 \widetilde \bH_\lambda = \dGamma(h) + \frac  \lambda 2 \int_{\R^d} \hat{w} (k) \left|\dG (e ^{ik \cdot x}) - \left\langle \dG (e ^{ik \cdot x}) \right\rangle_{\widetilde \Gamma_0} \right| ^2 dk, \quad \widetilde \Gamma_0=\frac{1}{\widetilde \cZ_0(\lambda)} e^{-\lambda\dGamma(h)}.
$$
can be treated by following exactly the proof of Theorem~\ref{thm:main-3} (all estimates depend on $h$ only via $\tr[h^{-2}]$, which is bounded uniformly thanks to the convergence \eqref{eq:VT-Vinf-Sp 11}) and we have
\begin{align} 
&\lambda\left(\widetilde F_\lambda - \widetilde F_0\right) = -\log \frac{\widetilde \cZ(\lambda)}{\widetilde \cZ_0(\lambda)} = -\log z_\lambda + o (1) =  -\log \left(\int e^{-\cD_\lambda [u]} d \mu_{0,\lambda} (u)\right) + o (1), \label{eq:cv-tG-1}\\
&\tr\left|\lambda^k\,k!\, \widetilde\Gamma_\lambda^{(k)} -  \int |u^{\otimes k} \rangle \langle u^{\otimes k} | d \mu_\lambda (u)\right|^2 \to0, \quad \forall k\ge 1, \label{eq:cv-tG-2} \\
&\Tr \left| \lambda\left( \widetilde\Gamma_\lambda^{(1)} -  \widetilde\Gamma_0^{(1)} \right) - \int |u \rangle \langle u | \left( d \mu_\lambda (u) - d \mu_{0,\lambda}(u) \right)  \right| \to 0. \label{eq:cv-tG-3}
\end{align}
Here $\mu_{0,\lambda}$ is the Gaussian measure constructed from $h = -\Delta + V_\lambda$, $\cD_\lambda,\mu_\lambda,z_\lambda$ the corresponding renormalized energy, interacting Gibbs measure and partition function, respectively. We emphasize in the notation that these objects still depend on $\lambda$.    

Moreover, the analysis in~\cite[Section~3]{FroKnoSchSoh-17} shows that 
\begin{equation} \label{eq:CV pert clas zzz} \log z_\lambda \underset{\lambda\to0^+}{\longrightarrow} \log z
\end{equation} 
and 
\begin{equation}\label{eq:CV pert clas DM}
 \int |u^{\otimes k} \rangle \langle u^{\otimes k} | d \mu_\lambda (u) \underset{\lambda\to0^+}{\longrightarrow} \int |u^{\otimes k} \rangle \langle u^{\otimes k} | d \mu (u) \mbox{ strongly in } \gS^2 
\end{equation}
with $z$ and $\mu$ the partition function and interacting Gibbs measure associated with $h = -\Delta + V_0$. Indeed, both the partition function and correlation functions (= kernels of reduced density matrices) based on $-\Delta + V_\lambda$ can be computed from a perturbative expansion whose coefficients clearly converge to those of the corresponding expansion based on $-\Delta + V_0$, using~\eqref{eq:VT-Vinf-Sp 11}. A similar argument is used at the end of~\cite[Proof of Lemma~3.1]{FroKnoSchSoh-17}. On the other hand, the remainder terms of both expansions are controlled using only the Hilbert-Schmidt norms of $(-\Delta + V_\lambda) ^{-1}$ and $(-\Delta + V_0) ^{-1}$ as in~\cite[Lemma~3.3]{FroKnoSchSoh-17}. The desired convergences then follow from Borel summation as described in~\cite[Appendix~A]{FroKnoSchSoh-17}. 

For the strong trace-class convergence of the relative one-particle density matrix, observe first from~\eqref{eq:DM rel clas norm} that 
\begin{equation}\label{eq:dif clas rel DM} 
\int |u \rangle \langle u | \left( d \mu_\lambda (u) - d\mu_{0,\lambda} + d\mu_{0} - d\mu \right) 
\end{equation}
is uniformly bounded in trace-class (for the difference $d \mu_\lambda (u) - d\mu_{0,\lambda}$ we also use $\Tr[(-\Delta+V_\lambda)^{-2}]$ is bounded uniformly). Hence we can assume that it converges weakly-$\ast$ in $\gS^1$ to $0$. Then, testing against a bounded operator $A$ yields a quantity converging to $0$ uniformly in the operator norm of $A$, by the techniques in~\cite{FroKnoSchSoh-17} again. We deduce that the trace-norm of~\eqref{eq:dif clas rel DM} converges to $0$, which shows that it must converge to $0$ strongly in trace-norm.   Putting differently, we have
\begin{equation}\label{eq:CV pert clas DM1}
\int |u \rangle \langle u | \left( d \mu_\lambda (u) - d\mu_{0,\lambda} \right) \to \int |u \rangle \langle u | \left( d \mu (u) - d\mu_{0} \right)   \mbox{ strongly in } \gS^1
\end{equation}
as desired.
\end{proof}

\appendix \section{The counter-term problem}\label{app:counter}

In this appendix we discuss the counter-term problem in detail. 

\subsection{Hartree versus reduced Hartree energy functional} \label{sec:app-H-vs-rH} 

We recall that to any one-body density matrix $\gamma\geq0$ one can associate a unique Gaussian state 
$$\Gamma=\frac{e^{-\dG(a)}}{\tr_\gF[e^{-\dG(a)}]}$$ 
on the Fock space which has the one-particle density matrix $\Gamma^{(1)}=\gamma$~\cite{BacLieSol-94,Solovej-notes}. The unique corresponding one-body operator $a$ is given by
\begin{equation}
 \gamma=\frac1{e^a-1}\quad\Longleftrightarrow \quad a=\log\frac{1+\gamma}{\gamma}.
 \label{eq:PDM-quasi-free}
\end{equation}
Then its energy terms and entropy can be expressed as in \eqref{eq:quasi free Hartree}, resulting the  \emph{Hartree free energy}
\begin{multline*}
\cF^{\rm H}[\gamma]:=\Tr\left[(-\Delta+V-\nu)\gamma\right]+\frac{\lambda}{2} \iint \gamma(x;x) w(x-y) \gamma(y;y)\, dx\, dy\\
+ \frac{\lambda}{2} \iint w(x-y) |\gamma(x;y)|^2 dx dy -T\tr\left[(1+\gamma)\log(1+\gamma)-\gamma\log\gamma\right]
\end{multline*}
Thus if we are interested in equilibrium states minimizing the free energy, in the quasi-free class this leads to the following variational problem
\begin{equation}
F^{\rm H}_\lambda=\inf_{\gamma=\gamma^*\geq0}\cF^{\rm H}[\gamma].
\label{eq:Hartree}
\end{equation}
When $\widehat{w}\geq0$, the functional $\cF^{\rm H}[\gamma]$ turns out to be strictly convex. Hence, with the confining potential $V$ it admits a unique minimizer $\gamma^{\rm H}$, that defines a unique corresponding quasi-free state in Fock space $\Gamma^{\rm H}$ (the proof is the same as that of Lemma~\ref{lem:min quasi free}). The optimal density matrix solves the nonlinear equation
$$\gamma^{\rm H}=\left\{\exp\left(\frac{-\Delta+V-\nu+\lambda\rho^{\rm H}\ast w+\lambda X^{\rm H}}{T}\right)-1\right\}^{-1}$$
where $\rho^{\rm H}(x)=\gamma^{\rm H}(x;x)$ is the density and $X^{\rm H}$ is the exchange operator with integral kernel $X^{\rm H}(x;y)=w(x-y)\gamma^{\rm H}(x;y)$.

In the limit $\lambda\to0$ with $T=1/\lambda$, the quasi-free state $\Gamma^{\rm H}$ is rather badly behaved. Its density $\rho^{\rm H}$ diverges very fast. However, it turns out that, although $\rho^{\rm H}(x)$ depends on $x$, its growth as $\lambda\to0$ is more or less uniform in $x$ and can be captured by
\begin{equation}
\rho^{\rm H}(x) \sim \rhoO^{\kappa}(\lambda) :=  \left[ \frac{1}{e^{\lambda(-\Delta+\kappa)}-1} \right] (x;x) = \frac{1}{(2\pi)^d\lambda^{\frac{d}2}} \int_{k\in \R ^d} \frac{dk}{e^{|k|^2+\lambda\kappa} -1} 
\label{eq:unif_DV}
\end{equation}
provided that 
\begin{equation}\label{eq:tune chem}
\nu=\lambda \hat{w} (0) \rhoO^\kappa -\kappa. 
\end{equation}
Recall that $\lambda\rhoO^\kappa(\lambda)$ diverges in dimension $d\ge 2$ but it does not depend on $x$ by translation invariance of $-\Delta + \kappa$. On the other hand, the exchange term $\lambda X^{\rm H}$ typically stays bounded, for instance in the Hilbert-Schmidt norm.

This suggests to simplify things a little bit by removing the exchange term as we did in the paper, that is, to consider the simplified minimization problem associated with the \emph{reduced Hartree} or, simply, \emph{mean-field free energy} \eqref{eq:red-Hartree_functional} which, we recall, is given by
\begin{multline*}
\cF^{\rm MF}[\gamma]:=\Tr\left[(-\Delta+V-\nu)\gamma\right]+\frac{\lambda}{2} \iint \gamma(x;x) w(x-y) \gamma(y;y)\, dx\, dy\\
 -T\tr\left[(1+\gamma)\log(1+\gamma)-\gamma\log\gamma\right]
\end{multline*}
By doing so we will pick as reference state a quasi-free state which is not the absolute minimizer of the true quantum free energy in the quasi-free class. However, manipulating states depending only on a potential simplifies the analysis. Fortunately, it turns out to be sufficient for our purpose, as justified in Lemma \ref{lem:min quasi free}, which we now prove. 

\subsection{Proof of Lemma~\ref{lem:min quasi free}}\label{app:QF} Under our assumptions, the first eigenvalue of the Friedrichs realization of $h=-\Delta+V$ is positive. In addition, we have 
$$\tr[e^{-h/T}]\leq (2\pi)^{-d}\int_{\R^d}e^{-p^2/T}\,dp\int_{\R^d}e^{-V(x)/T}\,dx<\ii,$$ 
by the Golden-Thompson inequality, see~\cite[Section~8.1]{Simon-79}. The same properties hold if we shift $V$ by $\nu\in\R$ and only keep the positive part. Then we obtain
\begin{multline}
\tr\left(\left(-\Delta+(V-\nu)_+\right)\gamma\right)-T\tr\left((1+\gamma)\log(1+\gamma)-\gamma\log\gamma\right) \\
\geq \frac12\tr\left(-\Delta\gamma\right) +  T\tr\left(\log\left(1-e^{-\frac{-\Delta/2+(V-\nu)_+}{T}}\right)\right)
\label{eq:first_bound_min_quasi_free}
\end{multline}
for all $\gamma\geq0$. The last expression is the minimum free energy associated with $-\Delta/2+(V-\nu)_+$. 

In order to prove that the functional $\cF^{\rm MF}$ is bounded from below, it therefore remains to show that 
$$-\int_{\R^d}\rho(x)(V-\nu)_-(x)\,dx+\frac\lambda2 \int_{\R^d}\int_{\R^d}w(x-y)\rho(x)\rho(y)\,dx\,dy$$
is bounded from below, uniformly in $\rho\geq0$, under the assumption that $\widehat{w}\geq0$ and $w\not\equiv 0$ (if $w\equiv 0$ then the Lemma holds true trivially). Since $w\in L^1(\R^d)$, we can find a $k_0\in\R^d$ and a small radius $r>0$ such that the continuous non-negative function $\widehat{w}$ is at least equal to $\widehat{w}(k_0)/2$ on $B(k_0,r)$. We then choose $\phi$ in the Schwartz class such that $\phi>0$ and $\widehat{\phi}\geq0$ with ${\rm supp}(\widehat\phi)\subset B(k_0,r)$. Since $V\geq0$ and $V\to+\ii$ at infinity, the function $(V-\nu)_-$ has compact support and is bounded by $\nu_+$. Therefore, we have
$$(V-\nu)_-\leq C\phi$$
for $C=\norm{\phi^{-1}(V-\nu)_-}_{L^\ii(\R^d)}<\ii$. After completing the square and using $\widehat{w}\geq0$, we then deduce
\begin{multline*}
-\int_{\R^d}(V-\nu)_-\rho+\frac\lambda2 \int_{\R^d}\int_{\R^d}w(x-y)\rho(x)\rho(y)\,dx\,dy\\
\qquad\geq -C\int_{\R^d}\phi\rho+\frac\lambda2 \int_{\R^d}\int_{\R^d}w(x-y)\rho(x)\rho(y)\,dx\,dy\geq -\frac{C^2}{2\lambda} \int_{\R^d}\frac{|\widehat{\phi}(k)|^2}{\widehat{w}(k)}\,dk\geq -\frac{C^2\|\widehat\phi\|^2_{L^2}}{\lambda\widehat{w}(k_0)}.
\end{multline*}
Combining with~\eqref{eq:first_bound_min_quasi_free} we find as stated that 
$$\inf_{\gamma=\gamma^*\geq0}\cF^{\rm MF}[\gamma]>-\ii,$$
for all $\nu\in\R$.

Let us now prove the existence of a minimizer. Writing
$$\cF^{\rm MF}_\nu[\gamma]=\cF^{\rm MF}_{\nu+1}[\gamma]+\tr\gamma\geq \inf_{\gamma'} \cF^{\rm MF}_{\nu+1}[\gamma']+\tr(\gamma),$$
where we have displayed the parameter $\nu$ for convenience, we obtain that minimizing sequences $\{\gamma_n\}$ for $\cF^{\rm MF}_\nu$ are necessarily bounded in the trace-class. In particular, $\|\gamma_n\|$ is also bounded. In addition, the inequality~\eqref{eq:first_bound_min_quasi_free} implies that $\tr(-\Delta)\gamma_n$ is bounded. From the Hoffmann-Ostenhof inequality~\cite{Hof-77}
$$\tr(-\Delta)\gamma\geq \int_{\R^d}\left|\nabla\sqrt{\rho_\gamma(x)}\right|^2\,dx$$
and the Sobolev inequality, we deduce that $\rho_{\gamma_n}$ is bounded in $L^{p^*/2}(\R^d)$ where $p^*=+\ii$ in dimension $d=1$, $p^*<\ii$ arbitrarily in dimension $d=2$ and $p^*=2d/(d-2)$ in dimensions $d\geq3$. 
Hence, up to extraction of a subsequence, we have $\gamma_n\wto\gamma$ weakly-$\ast$ in $\gS^1$ and $\rho_{\gamma_n}\wto\rho_{\gamma}$ weakly in $L^1(\R^d)\cap L^{p^*/2}(\R^d)$. Using Fatou's lemma and the concavity (hence weak upper semi-continuity) of the entropy, we obtain that $\gamma$ is a minimizer for $\cF_\nu^{\rm MF}$. The nonlinear equation~\eqref{eq:gammarH} follows from classical arguments. Then, according to~\eqref{eq:PDM-quasi-free}, $V_\lambda$ must solve~\eqref{eq:counter-term} because the minimizer $\gamma^{\rm MF}$ is the one-body density matrix of the quasi-free state associated with $\dG (-\Delta + V_\lambda)$.
\qed

\subsection{Comments on Theorem~\ref{thm:counter}}\label{app:FKSS}

Let us briefly discuss Theorem~\ref{thm:counter}. In~\cite[Section~5]{FroKnoSchSoh-17}, the existence of the solution $V_\lambda$ to~\eqref{eq:counter-term} was established by means of a fixed-point argument (which requires that $d\le 3$ and that $\kappa$ is sufficiently large). The fixed point is performed in the (complete) metric space
$$
B(V)=\left\{f \in L^\infty_{\rm loc} (\R^d): \|f\|_{B(V)}= \left\| \frac{f}{V} - 1 \right\|_{L^\infty} \le 1/2 \right\}
$$ 
for the unknown $u=V_\lambda-\kappa$ and provides the Hilbert-Schmidt convergence 
\begin{align*} 
\Tr \Big| (-\Delta+V_\lambda)^{-1} -(-\Delta+V_0)^{-1}  \Big|^2 \to 0.
\end{align*}
Our notation here is slightly different from \cite{FroKnoSchSoh-17} as we shift potentials by a constant. Moreover, since $V_\lambda-\kappa\in B(V)$ we have 
$$
\frac{V}2 \le V_\lambda -\kappa\le 3\frac{V}2.
$$

There remains to discuss the nonlinear equation~\eqref{eq:equation_V_infty} for $V_0$, which we can rewrite in the form
\begin{equation}
V_0=w\ast\left(V+\kappa+\rho\left[\big(-\Delta+V_0\big)^{-1}-\big(-\Delta+\kappa\big)^{-1}\right]\right).
\label{eq:equation_V_infty_repeated}
\end{equation}
Here we just need to pass to the limit in the similar equation  at $\lambda>0$
$$V_\lambda=w\ast\left(V+\kappa+\rho\left[\frac{\lambda}{ e^{\lambda(-\Delta+V_\lambda)}-1}-\frac{\lambda}{ e^{\lambda(-\Delta+\kappa)}-1}\right]\right).$$
Since we know that $V_\lambda/V\to V_0/V$ in $L^\ii$, we have $V_\lambda\to V_0$ in $L^\ii_{\rm loc}$ and it suffices to prove the convergence of the density on the right side, which we denote for simplicity
$$\rho_\lambda^{V_\lambda}(x):=\left[\frac{\lambda}{ e^{\lambda(-\Delta+V_\lambda)}-1}-\frac{\lambda}{e^{\lambda(-\Delta+\kappa)}-1}\right](x;x).$$
In~\cite[Eq.~(5.21)]{FroKnoSchSoh-17} it is shown that 
\begin{equation}
 |\rho_\lambda^{V_\lambda}(x)|\leq C\kappa^{d/2-2}\norm{\frac{V_\lambda- \kappa}{V}}_{L^\ii} V(x).
 \label{eq:uniform_bound_rho_T}
\end{equation}
Hence from the dominated convergence theorem and the assumptions on $w$, it suffices to prove that 
$$\rho_\lambda^{V_\lambda}(x)\to \left(\big(-\Delta+V_0\big)^{-1}-\big(-\Delta+\kappa\big)^{-1}\right)(x;x)$$
almost everywhere. Applying again~\cite[Eq.~(5.21)]{FroKnoSchSoh-17} we find that 
$$\left|\left(\frac{\lambda}{e^{\lambda(-\Delta+V_\lambda)}-1}-\frac{\lambda}{ e^{\lambda(-\Delta+V_0)}-1}\right)(x;x)\right|\leq C\kappa^{d/2-2}\norm{\frac{V_\lambda-V_0}{V}}_{L^\ii} V(x)$$
which tends to 0 in $L^\ii_{\rm loc}$ since $(V_\lambda-V_0)/V\to0$. 
Hence we can replace $V_\lambda$ by $V_0$ throughout. Next we write, following~\cite[Eq. (5.16)]{FroKnoSchSoh-17},
\begin{equation*}
\rho_\lambda^{V_0}(x)=-\int_0^1{\rm d}s\int_{\R^d}{\rm d}z\frac{\lambda\,e^{s\lambda(-\Delta+V_0)}}{e^{\lambda(-\Delta+V_0)}-1}(x;z)V_0(z)\frac{\lambda\,e^{(1-s)\lambda(-\Delta+\kappa)}}{e^{\lambda(-\Delta+\kappa)}-1}(z;x).
\end{equation*}
Using that $V_0\geq\kappa$, we have the pointwise bound on the operator kernels
$$0\leq \frac{e^{s\lambda(-\Delta+V_0)}}{e^{\lambda(-\Delta+V_0)}-1}(x;z)\leq \frac{e^{s\lambda(-\Delta+\kappa)}}{e^{\lambda(-\Delta+\kappa)}-1}(x;z)$$
by the same argument as in Lemma~\ref{lem:CV-DM-high} and in~\cite[Eq. (5.17)]{FroKnoSchSoh-17}. Using~\cite[Lemma 5.4]{FroKnoSchSoh-17} we see that we get a convergent domination. So by the dominated convergence theorem, the strong local convergence of $\rho^{V_0}_\lambda$ follows from that of the kernels
$$\frac{\lambda\, e^{\lambda(-\Delta+V_0)}}{e^{\lambda(-\Delta+V_0)}-1}(x;z)\to\frac{1}{-\Delta+V_0}(x;z),$$
$$\frac{\lambda\,e^{s\lambda(-\Delta+\kappa)}}{e^{\lambda(-\Delta+\kappa)}-1}(x;z)\to\frac{1}{-\Delta+\kappa}(x;z).$$
In fact this convergence is strong in $L^2(\R^d\times\R^d)$ since the corresponding operators converge in the Hilbert-Schmidt norm, by~\cite[Lemma~C.1]{FroKnoSchSoh-17}. Passing to the limit, this proves the strong local convergence 
\begin{multline*}
\rho^{V_0}_\lambda(x)\underset{\lambda\to0^+}{\longrightarrow} -\int_0^1{\rm d}s\int_{\R^d}{\rm d}z\frac{1}{-\Delta+V_0}(x;z)V_0(z)\frac{1}{-\Delta+\kappa}(z;x)\\
=\left(\big(-\Delta+V_0\big)^{-1}-\big(-\Delta+\kappa\big)^{-1}\right)(x;x), 
\end{multline*}
where in the last equality we have used the resolvent formula. The uniform bound~\eqref{eq:uniform_bound_rho_T} then allows to pass to the limit in the equation for $V_\lambda$ and obtain~\eqref{eq:equation_V_infty_repeated}. \qed

\section{Interpretation in terms of the phase transition of the Bose gas}\label{app:interpretation}

Here we reformulate and discuss our main results (Theorems~\ref{thm:main-1} and~\ref{thm:main-2}) in microscopic variables and clarify the link with the phase transition of the infinite Bose gas. 

\subsection{Homogeneous case}

We start with the homogeneous case which is the usual setting in which the thermodynamics of the free Bose gas is formulated, see for instance~\cite[Section~5.2.5]{BraRob2},~\cite[Sections~2.5.19--20]{Thirring} and~\cite[Chapter~4]{Verbeure-11}. 

Let us consider the non-interacting Bose gas on the large torus $L\bT^d$ of side length $L$. In the grand canonical setting we choose a chemical potential $\widetilde\nu<0$ and a temperature $T>0$. Our system is then represented by the Gaussian quantum state in Fock space, associated with the one-particle operator $(-\Delta_L-\widetilde\nu)/T$ where $-\Delta_L$ is the $L$-periodic Laplacian. Its one-body density matrix is\footnote{Our convention in this section is that all microscopic density matrices have a tilde, whereas the macroscopic ones do not.}
$$\widetilde\Gamma^{(1)}_L=\frac{1}{e^{\frac{-\Delta_L-\widetilde\nu}{T}}-1}$$
and the number of particles per unit volume is given by
$$\frac{1}{L^d}\sum_{k\in 2\pi \Z^d/L}\frac{1}{e^{\frac{|k|^2-\widetilde\nu}{T}}-1}\underset{L\to\ii}{\longrightarrow}\frac{T^{\frac{d}2}}{(2\pi)^d}\int_{\R^d}\frac{dk}{e^{k^2-\widetilde\nu/T}-1}.$$
The critical density for Bose-Einstein condensation is obtained in the limit $\widetilde\nu/T\to0^-$ and it equals
\begin{align}
\rho_c(T)=\frac{T^{\frac{d}2}}{(2\pi)^d} \int_{\R^d}\frac{dk}{e^{k^2}-1} 
=\begin{cases}
+\ii&\text{for $d=1,2$,}\\
\frac{T^{\frac{d}2}\zeta\left(\frac{d}2\right)}{2^d\pi^{\frac{d}2}}<\ii&\text{for $d\geq3$.}
 \end{cases} 
\end{align}
The grand-canonical model in infinite space stops to exist at $\widetilde\nu=0$, where the one-body density matrix converges weakly to 
\begin{equation}
 \widetilde\Gamma_\ii^{(1)}=\frac{1}{e^{\frac{-\Delta}{T}}-1}.
 \label{eq:limit_1PDM_GC}
\end{equation}
The corresponding infinite Gaussian state (properly defined over the $C^*$-algebra of Canonical Commutation Relations~\cite{BraRob2}) has the number of particles per unit volume equal to $\rho_c(T)$. 

The phenomenon of Bose-Einstein Condensation (BEC) is better understood in the \emph{canonical} setting with $N$ particles, going back to the thermodynamic limit. In dimensions $d\geq3$, fixing the density $N/L^d=\rho>\rho_c(T)$ one obtains a density matrix which has a two-scale behavior~\cite[Sec.~5.2.5]{BraRob2}. It contains a rank-one part with the diverging eigenvalue $L^d(\rho-\rho_c(T))$ and constant eigenfunction $f_0(x)=L^{-d/2}$, living at the macroscopic scale (the Bose-Einstein condensate). When this rank-one operator is removed, the rest of the density matrix converges weakly to $\widetilde\Gamma_\ii^{(1)}$ in~\eqref{eq:limit_1PDM_GC}.

To understand the behavior of the system just before the phase transition, we have to look at the simultaneous limit $L\to\ii$ with $\widetilde\nu\to0^-$, see~\cite[Sec.~2.5.20.3]{Thirring}. At the macroscopic scale (that is, after rescaling length by $L$), the (rescaled) one-particle density matrix equals
$$\Gamma^{(1)}_L=\frac{1}{e^{\frac{-\Delta-L^2\widetilde\nu}{L^2T}}-1}\qquad\text{in $L^2(\bT^d)$}$$
and thus we see that the natural scaling for $\widetilde\nu$ is
\begin{equation}
\widetilde\nu(L)=-\frac{\kappa}{L^2}
\label{eq:tilde_nu_micro}
\end{equation}
with a fixed $\kappa>0$, in all dimensions $d\geq1$. We are therefore exactly in the setting studied in this paper and in~\cite{LewNamRou-15} with the choice 
$$\boxed{\lambda=\frac1{TL^2}\underset{L\to\ii}{\longrightarrow}0^+.}$$
The density matrix converges to 
\begin{equation}
 \frac{\Gamma_L^{(1)}}{TL^2}\underset{L\to\ii}{\longrightarrow}\frac{1}{-\Delta+\kappa}=\int|u\rangle\langle u|\,d\mu_\kappa(u)
 \label{eq:CV_PDM_Bose_gas}
\end{equation}
strongly in the Schatten space $\gS^p(L^2(\bT^d))$ for all $p>d/2$ ($p\geq1$ if $d=1$), where $\mu_\kappa$ is the classical Gaussian measure with covariance $(-\Delta+\kappa)^{-1}$. Equivalently,
$$ \frac{\Gamma_L^{(1)}}{L^2}\underset{L\to\ii}{\longrightarrow}\frac{T}{-\Delta+\kappa}=\int|u\rangle\langle u|\,d\mu_{\kappa,T}(u)$$
where $\mu_{\kappa,T}$ is the Gaussian measure with covariance $T(-\Delta+\kappa)^{-1}$.  Similar properties hold for higher density matrices. 

Our conclusion is that, close to its phase transition, the free Bose gas is properly described by the classical Gaussian measure $\mu_{\kappa,T}$ on $\bT^d$ (macroscopic scale) where $\kappa$ describes the speed at which the chemical potential $\widetilde\nu$ approaches $0$ via~\eqref{eq:tilde_nu_micro} or, equivalently, at which the corresponding density approaches the critical density $\rho_c(T)$. This elementary fact is rarely mentioned in textbooks on Bose gases. 

The speed at which the density approaches $\rho_c(T)$ is computed by using the following

\begin{lemma}[\textbf{Particle number of the free Bose gas}]\label{lem:equivalent}\mbox{}\\
In dimensions $d\geq 1$, we have for every fixed $\kappa>0$
\begin{equation}
\sum_{k\in2\pi\Z^d}\frac{1}{e^{\lambda(|k|^2+\kappa)}-1}=\frac{\lambda^{-\frac{d}2}}{(2\pi)^d}\int_{\R^d}\frac{dk}{e^{k^2+\lambda\kappa}-1}+\frac{\phi_d(\kappa)}{\lambda}+ o\left(\lambda^{-1}\right)_{\lambda\to0^+}
\label{eq:equivalent}
\end{equation}
with the positive decreasing function
\begin{align}
\phi_d(\kappa)&=\sum_{\ell\in\Z^d\setminus\{0\}}\int_0^\ii \frac{e^{-t\kappa}}{(4\pi t)^{\frac{d}2}} e^{-\frac{|\ell|^2}{4t}}\,dt =\begin{cases}
\dps \frac{1}{2}\sum_{\ell\in\Z\setminus\{0\}} e^{-\kappa|\ell|}&\text{for $d=1$,}\\[0.5cm]
\dps \frac{1}{4\pi}\sum_{\ell\in\Z^3\setminus\{0\}}\frac{e^{-\sqrt{\kappa}|\ell|}}{|\ell|}&\text{for $d=3$.}
  \end{cases}
  \label{eq:phi_d_app}
\end{align}
\end{lemma}

The proof is provided below in Section~\ref{sec:proof_Lemma_expansion}. The integral in the first equality of~\eqref{eq:phi_d_app} is the Fourier transform of $(2\pi)^{-d/2}(|k|^2+\kappa)^{-1}$, that is, the Klein-Gordon Green function. In dimensions $d\leq3$, we can also write by Poisson's formula
$$\phi_d(\kappa)=\sum_{k\in2\pi\Z^d}\left(\frac1{|k|^2+\kappa}-\frac{1}{(2\pi)^d}\int_{(-\pi,\pi)^d}\frac{dp}{|k+p|^2+\kappa}\right)$$
but the sum is not absolutely convergent in dimensions $d\geq4$. Since we have the expansions
\begin{align}
 \frac{1}{(2\pi)^d}\int_{\R^d}\frac{dk}{e^{|k|^2+a}-1}
&=\begin{cases}
\frac{1}{2\sqrt{a}}+\frac{\zeta(\frac12)}{2\sqrt{\pi}}+o(1)_{a\to0^+}&\text{for $d=1$,}\\[0.4cm]
\frac{-\log(a)}{4\pi}+\frac{a}{8\pi}+o(a)_{a\to0^+}&\text{for $d=2$,}
  \end{cases}\nn\\
&=\frac{\zeta\left(\frac{d}2\right)}{2^d\pi^{\frac{d}2}}-\begin{cases}
\frac{\sqrt{a}}{4\pi}+O(a)_{a\to0^+}&\text{for $d=3$,}\\[0.2cm]
\frac{(\log(1/a)+1)a}{16\pi^2}+o(a)_{a\to0^+}&\text{for $d=4$,}\\[0.2cm]
\frac{\zeta\left(\frac{d}2-1\right)}{2^d\pi^{\frac{d}2}}a+o(a)_{a\to0^+}&\text{for $d\geq5$.}
  \end{cases}
\label{eq:expansion_integral}
\end{align}
we find from~\eqref{eq:equivalent} that the density approaches the critical density $\rho_c(T)$ from below as
\begin{align}
 \frac{\tr[\Gamma^{(1)}_L]}{L^d}
&=\begin{cases}
\left(\frac{1}{2\sqrt\kappa}+\phi_1(\kappa)\right)TL+o(L)&\text{for $d=1$,}\\[0.4cm]
\frac{T}{2\pi}\log(L)+\left(\phi_2(\kappa)-\frac{\log(\kappa/T)}{4\pi}\right)T+o(1)&\text{for $d=2$,}
  \end{cases}\nn\\
&=\frac{T^{\frac{d}2}\zeta\left(\frac{d}2\right)}{2^d\pi^{\frac{d}2}}-\begin{cases}
\left(\frac{\sqrt\kappa}{4\pi}-\phi_3(\kappa)\right)\frac{T}{L}+o(L^{-1})&\text{for $d=3$,}\\[0.2cm]
\frac{\kappa T}{8\pi^2}\frac{\log L}{L^2}+\left(\frac{1+\log\frac{T}{\kappa}}{16\pi^2}\kappa-\phi_4(\kappa)\right)\frac{T}{L^2}+o(L^{-2})&\text{for $d=4$,}\\[0.2cm]
\frac{\zeta\left(\frac{d}2-1\right)}{2^d\pi^{\frac{d}2}}\frac{T^{\frac{d}2-1}\kappa}{L^2}+o(L^{-2})&\text{for $d\geq5$.}
\end{cases}
\label{eq:density_speed_homogeneous}
\end{align}
Note that $\phi_3(\kappa)\leq\sqrt\kappa/(4\pi)$ which is its behavior when $\kappa\to0^+$. 

So far our discussion applies to any dimension $d\geq1$. Next we discuss the inclusion of interactions for $d\leq3$. In our work the potential $w$ is introduced at the macroscopic level.  Re-expressed in microscopic variables and taking $T=1$ for simplicity, we obtain the microscopic $n$-particle Hamiltonian
\begin{equation}
\tilde H_{n,L}=\sum_{j=1}^n(-\Delta_L-\widetilde\nu)_{x_j}+\frac{1}{L^4}\sum_{1\leq j<k\leq n}w\left(\frac{x_j-x_k}{L}\right).
\label{eq:H_micro}
\end{equation}
The interaction has the very small intensity $L^{-4}$ but varies on length scales comparable with the size of the box. The form of the microscopic Hamiltonian~\eqref{eq:H_micro} is the same in all dimensions. 
The thermodynamic limit $L\to\ii$ of this model at fixed $\widetilde\nu<0$ is the same as the non-interacting case. This is because we have the lower bound
$$\tilde H_{n,L}\geq \sum_{j=1}^n-(\Delta_L)_{x_j}-\left(\widetilde\nu+\frac{w(0)}{2L^4}\right) n$$
due to the fact that $\widehat{w}\geq0$ (for an upper bound on the free energy, use the non-interacting state). We conclude that $w$ does not, to leading order, change the phase diagram as compared to the non-interacting case. The effect of $w$ is only visible when zooming just before the phase transition. From Theorem~\ref{thm:main-1}, when the chemical potential goes to zero as
$$\widetilde\nu(L)=\frac{\nu(L^{-2})}{L^2}=\begin{cases}
\dps \frac{\hat{w} (0)\,\log (L)}{2\pi L^2} -\frac{\nu_0}{L^2}+o(1)_{\lambda\to0^+}&\text{for  $d=2$,}\\[0.2cm]
\dps \frac{\hat{w} (0)\zeta\left(\frac32\right)}{8\pi^{\frac32}L}-\frac{\nu_0}{L^2}+o(1)_{\lambda\to0^+}&\text{for $d=3$,}
\end{cases}
$$
then the behavior close to the transition is described by the nonlinear Gibbs measure $\mu$ at the macroscopic scale, which depends on $w$ and $\kappa$ solving
\begin{equation}
\nu_0=\begin{cases}
\kappa+\widehat{w}(0)\frac{\log(\kappa)}{4\pi}-\widehat{w}(0)\,\phi_2(\kappa)&\text{for $d=2$,}\\[0.2cm]
\kappa+\widehat{w}(0)\frac{\sqrt\kappa}{4\pi}-\widehat{w}(0)\,\phi_3(\kappa)&\text{for $d=3$,}
\end{cases}
\label{eq:nu_0_kappa_app}
\end{equation}
More physical interactions are much bigger and have a much shorter range. Although a universal behavior can still be expected at the phase transition, the phase diagram depends on $w$ at leading order and a mathematical treatment seems out of reach with the present techniques. A simpler behavior is however expected in the dilute regime $\rho\to0$ with $\rho \sim T^{d/2}$ (Gross-Pitaevskii regime~\cite{DeuSei-19_ppt}). In dimension $d=3$ and at our macroscopic scale, the Gross-Pitaevskii limit corresponds to replacing $\lambda w$ by $\lambda w_\lambda$ with $w_\lambda(x)=\lambda^{-3} w(x/\lambda)$ in our many-particle Hamiltonian. In this case one would expect the phase transition to be described by the (appropriately renormalized) nonlinear Gibbs measure $\mu$ over the torus $\bT^3$, with $w$ replaced by the Dirac delta $8\pi a \delta_0$ where $a$ is the scattering length of $w$~\cite{LieSeiYng-00,LieSeiSolYng-05,DeuSei-19_ppt}. Proving such a result seems a formidable task. 

\subsection{Trapped gases}

The theory of Bose-Einstein condensation for trapped gases is analogous to the homogeneous case, but the formulas are slightly different, see for instance~\cite{BagPriKle-87}, \cite[Sec.~2.5.15]{Thirring} and~\cite{DeuSeiYng-19}. Here we only discuss the case $V(x)=|x|^s$ for simplicity.
At the microscopic scale the one-body Hamiltonian takes the form
$$\tilde h_L:=-\Delta +\frac{|x|^s}{L^{2+s}}-\tilde\nu$$
where $L$ is now a parameter used to open the trap whenever the number of particles grows. At fixed $\tilde\nu<0$, the number of particles in the non-interacting Gaussian state is given by
\begin{align}
\tr\left(\frac{1}{e^{T^{-1}(-\Delta +L^{-2-s}|x|^s-\tilde\nu)}-1}\right)&=\tr\left(\frac{1}{e^{-T^{-1}\left(L^{-\frac{4+2s}{s}}\Delta +|x|^s-\tilde\nu\right)}-1}\right)\nn\\
&\!\!\!\!\underset{L\to\ii}{\sim} \frac{(L^2T)^{d\left(\frac12+\frac1s\right)}}{(2\pi)^d}\iint_{\R^d\times\R^d}\frac{dx\,dk}{e^{|k|^2+|x|^s-\tilde\nu/T}-1}.\label{eq:semi-classics} 
\end{align}
In the first equality we have rescaled lengths by the factor $L^{(2+s)/s}$, which places the system in a conventional semi-classical limit with effective parameter $\hbar=L^{-(2+s)/s}\to0$, hence the second limit. 
Computing in the same manner the average against $|x|$ one sees that the gas is extended at the length scale
\begin{equation}
 \ell_{\rm gas}\sim (L^2T)^{\left(\frac12+\frac1s\right)}T^{-\frac12}.
 \label{eq:extension_cloud}
\end{equation}
Dividing by the effective volume $(\ell_{\rm gas})^d$, the density obtained in the thermodynamic limit is therefore proportional to 
$$\frac{T^{\frac{d}2}}{(2\pi)^d}\iint_{\R^d\times\R^d}\frac{dx\,dk}{e^{|k|^2+|x|^s-\tilde\nu/T}-1}.$$
Bose-Einstein condensation is obtained as before when $\tilde\nu/T\to0^-$, with the critical density
$$\rho'_c(T)=\frac{T^{\frac{d}2}}{(2\pi)^d}\iint_{\R^d\times\R^d}\frac{dx\,dk}{e^{|k|^2+|x|^s}-1}.$$
This is finite for all $s>0$ in dimensions $d\geq2$. In dimension $d=1$, $\rho'_c(T)$ is finite for $s<2$ and infinite otherwise. 

As before the Bose-Einstein condensate emerges in the limit $L\to\ii$ in the canonical setting, when $N>\rho'_c(T)(\ell_{\rm gas})^d$. The corresponding condensate wavefunction is the first eigenfunction of $-\Delta+L^{-2-s}|x|^s$. This is nothing but that of $-\Delta+|x|^s$ dilated to the scale $L$. Therefore, in this system the BEC length scale is $L$ and it is always smaller than the natural extension length $\ell_{\rm gas}$ of the cloud  in~\eqref{eq:extension_cloud}, at a fixed temperature $T>0$. The two coincide only in a sharp container ($s=+\ii$). 

The Gaussian Gibbs measure based on $h=-\Delta+|x|^s$ emerges at the BEC length scale $L$, whenever the chemical potential is chosen as
$$\tilde\nu(L)=-\frac{\kappa}{L^2}.$$
At this scale the one-particle Hamiltonian just becomes $(-\Delta+|x|^s+\kappa)/L^2$ so that $\lambda=1/(TL^2)$ like in the homogeneous case. The arguments from~\cite{FroKnoSchSoh-17} and Appendix~\ref{app:FKSS} in the non-interacting case give 
\begin{multline}
\rho\left[\frac{1}{e^{\frac{-\Delta +|x|^s+\kappa}{TL^2}}-1}\right](x)=\frac{T^{\frac{d}2}L^d}{(2\pi)^d}\int_{\R^d}\frac{dk}{e^{|k|^2+\frac{\kappa}{TL^2}}-1}\\
+TL^2\;\rho\left[\frac{1}{-\Delta+|x|^s+\kappa}-\frac{1}{-\Delta+\kappa}\right](x)+o(TL^2)_{L\to\ii}
\label{eq:equivalent_density}
\end{multline}
for every $x\in\R^d$. Here the density on the second line plays the role of $\phi_d$ in the homogeneous case. In particular we obtain all the same formulas as in~\eqref{eq:density_speed_homogeneous} with $\phi_d(\kappa)$ replaced by
$$\phi_d'(\kappa):=\tr\left[\frac{1}{-\Delta+|x|^s+\kappa}-\frac{1}{-\Delta+\kappa}\right].$$
In the inhomogeneous case, interactions were introduced at the scale $L$ of the BEC and the interpretation is the same as in the homogeneous case. 

\subsection{Proof of Lemma~\ref{lem:equivalent}}\label{sec:proof_Lemma_expansion}
We write
$$\frac{\lambda}{e^{\lambda h}-1}=\lambda\sum_{n\geq1}e^{-n\lambda h}$$
and obtain
\begin{multline*}
\lambda\left(\sum_{k\in2\pi\Z^d}\frac{1}{e^{\lambda(|k|^2+\kappa)}-1}-\frac{1}{(2\pi)^d}\int_{\R^d}\frac{dp}{e^{\lambda(|p|^2+\kappa)}-1}\right)\\
=\lambda\sum_{n\geq1}e^{-n\lambda\kappa}\sum_{k\in2\pi\Z^d}\left(e^{-n\lambda|k|^2}-\frac{1}{(2\pi)^d}\int_{(-\pi,\pi)^d}e^{-n\lambda|k-p|^2}\,dp\right).
\end{multline*}
We use Poisson's formula
$$\sum_{k\in2\pi\Z^d}\widehat{f}(k)=\frac1{(2\pi)^{\frac{d}2}}\sum_{\ell\in\Z^d}f(\ell)$$
for $f(x)=e^{-n\lambda|x|^2}-e^{-n\lambda|\cdot|^2}\ast\chi(x)$ where $\chi=(2\pi)^{-d}\1_{(-\pi,\pi)^d}$ and find
\begin{equation*}
\lambda\left(\sum_{k\in2\pi\Z^d}\frac{1}{e^{\lambda(|k|^2+\kappa)}-1}-\frac{1}{(2\pi)^d}\int_{\R^d}\frac{dp}{e^{\lambda(|p|^2+\kappa)}-1}\right)
=\frac{\lambda}{(4\pi)^{\frac{d}2}}\sum_{n\geq1}\frac{e^{-n\lambda\kappa}}{(4\pi n\lambda)^{\frac{d}2}}\sum_{\ell\in\Z^d\setminus\{0\}} e^{-\frac{|\ell|^2}{4n\lambda}}.
\end{equation*}
This is a Riemann sum which converges to
$$\sum_{\ell\in\Z^d\setminus\{0\}}\int_0^\ii \frac{e^{-t\kappa}}{(4\pi t)^{\frac{d}2}} e^{-\frac{|\ell|^2}{4t}}\,dt$$
where the right side is the Fourier transform of $k\mapsto (2\pi)^{-d/2}(|k|^2+\kappa)^{-1}$, see~\cite{LieLos-01}.
\qed


\end{document}